\documentclass[1 [leqno, 11pt]{amsart}
\usepackage{amssymb,  amsmath, latexsym, amssymb, amsfonts, amsbsy, amsthm,mathtools, graphicx, CJKutf8, CJKnumb, CJKulem, extarrows, color, dsfont}
\usepackage{amssymb, amsmath,amsmath,latexsym,amssymb,amsfonts,amsbsy, amsthm,mathtools,graphicx,CJKutf8,CJKnumb,CJKulem,color}

\usepackage{tikz}
\usetikzlibrary{shapes.geometric, arrows,plotmarks,backgrounds}
\usepackage{pgfplots}

\numberwithin{equation}{section}

\allowdisplaybreaks

\let\al=\alpha

\let\d=\delta

\let\f=\frac

\let\om=\omega

\let\ep=\epsilon
\let\Om=\Omega
\let\wt=\widetilde

\let\th=\theta
\let\pa=\partial
\let\va=\varphi
\let\Th=\Theta

\def\R{\mathbf R}

\def\no{\noindent}

\def\eqdef{\buildrel\hbox{\footnotesize def}\over =}

\newcommand{\beq}{\begin{equation}}
\newcommand{\eeq}{\end{equation}}
\newcommand{\beqno}{\begin{equation*}}
\newcommand{\eeqno}{\end{equation*}}
\newcommand{\ben}{\begin{eqnarray}}
\newcommand{\een}{\end{eqnarray}}
\newcommand{\beno}{\begin{eqnarray*}}
\newcommand{\eeno}{\end{eqnarray*}}

\newtheorem{theorem}{Theorem}[section]
\newtheorem{definition}[theorem]{Definition}
\newtheorem{lemma}[theorem]{Lemma}
\newtheorem{proposition}[theorem]{Proposition}

\newtheorem{remark}[theorem]{Remark}

\begin{document}
\begin{CJK*}{GBK}{song}

\title[the generation of the magnetic island]{Long time behavior of Alfv\'en waves in a flowing plasma: mathematical analysis on the generation of the magnetic island}

\author{Cuili Zhai}
\address{School of Mathematical Science, Peking University, 100871, Beijing, P. R. China}
\email{zhaicuili@math.pku.edu.cn}

\author{Zhifei Zhang}
\address{School of Mathematical Science, Peking University, 100871, Beijing, P. R. China}
\email{zfzhang@math.pku.edu.cn}

\author{Weiren Zhao}
\address{Department of Mathematics, New York University in Abu Dhabi, Saadiyat Island, P.O. Box 129188, Abu Dhabi, United Arab Emirates.}
\email{zjzjzwr@126.com, wz19@nyu.edu}
\maketitle

\begin{abstract}
\par In this paper, we consider the generation of magnetic island for the linearized MHD equations around the steady flowing plasma with velocity field $U_s=(u(y),0)$
and magnetic field $H_s=(b(y),0)$ in the finite channel.
\end{abstract}
\section{Introduction}
In this paper, we consider the Alfv\'en waves governed by the two-dimensional incompressible MHD equations in a finite channel $\Omega=\big\{(x,y)|x\in\mathbb{T}, y\in[-1,1]\big\}$:
\begin{equation}\label{eq: ideal MHD}
 \left\{\begin{array}{l}
\partial_tU+U\cdot\nabla U-H\cdot\nabla H+\nabla P=0,\\
\partial_tH+U\cdot\nabla H-H\cdot\nabla U=0,\\
\nabla\cdot U=0,\ \ \nabla\cdot H=0,\\
U_2(t,x,y)|_{y=-1,1}=0, \ \ H_2(t,x,y)|_{y=-1,1}=0.\\
\end{array}\right.
\end{equation}
with initial data $U(0,x,y)$ and $H(0,x,y)$. Here $U=(U_1,U_2)$, $H=(H_1,H_2)$
and $P$ denote the velocity field, magnetic field, and the total  pressure(kinetic plus magnetic) of the magnetic fluid, respectively.

This system has an equilibrium $U_s=(u(y),0)$, $H_s=(b(y),0)$, $P_s=\mathrm{const}$.
 We focus on the secular behavior of the 2D linearized MHD equations around this equilibrium, which take the form
 \begin{equation}\label{eq: linearized MHD}
 \left\{\begin{array}{l}
\partial_tV_1+u\partial_xV_1+\partial_xp+u'V_2-b\partial_xB_1-b'B_2=0,\\
\partial_tV_2+u\partial_xV_2+\partial_yp-b\partial_xB_2=0,\\
\partial_tB_1+u\partial_xB_1+b'V_2-b\partial_xV_1-u'B_2=0,\\
\partial_tB_2+u\partial_xB_2-b\partial_xV_2=0,\\
\nabla\cdot V=0,\ \ \nabla\cdot B=0,\\
V_2(t,x,y)|_{y=-1,1}=0, \ \ B_2(t,x,y)|_{y=-1,1}=0.\\
\end{array}\right.
\end{equation}
with initial data $V(0,x,y)=(v_1(x,y), v_2(x,y))$
and $B(0,x,y)=(b_1(x,y), b_2(x,y))$. Let $w_0=\partial_xv_2-\partial_yv_1$
and $j_0=\partial_xb_2-\partial_yb_1$ be the initial vorticity and current density.

It is easy to deduce form \eqref{eq: linearized MHD} that the vorticity $w=\partial_xV_2-\partial_yV_1$
and the current density $j=\partial_xB_2-\partial_yB_1$ satisfy the following equations:
\begin{align}\label{eq: vorticity and current density}
 \left\{\begin{array}{l}
\partial_tw+u\partial_xw-b\partial_xj=u''V_2-b''B_2,\\
\partial_tj+u\partial_xj-b\partial_xw=b''V_2-u''B_2+u'\partial_xB_1
-u'\partial_yB_2+b'\partial_yV_2-b'\partial_xV_1,
\end{array}\right.
\end{align}
here and in what follows, we use the notions $u, u',u''$ and $b, b',b''$ in stead by $u(y), \pa_yu(y), \pa_{yy}u(y)$ and $b(y), \pa_yb(y), \pa_{yy}b(y)$ for brevity.

In terms of  the stream functions $\psi$: $V=(\partial_y\psi, -\partial_x\psi)$,
then we have $w=-\Delta\psi$. Similarly, there is a scalar function $\phi$
 such that $B=(\partial_y\phi, -\partial_x\phi)$ and $j=-\Delta\phi$.
 Then we can deduce the following system on $(\psi, \phi)$:
\begin{equation}\label{eq: psi phi}
 \left\{\begin{array}{l}
\partial_t(\Delta\psi)+u\partial_x(\Delta\psi)-b\partial_x(\Delta\phi)=u''\partial_x\psi
-b''\partial_x\phi,\\
\partial_t(\Delta\phi)+u\partial_x(\Delta\phi)-b\partial_x(\Delta\psi)=b''\partial_x\psi
-u''\partial_x\phi-2u'\partial_x\partial_y\phi
+2b'\partial_x\partial_y\psi,
\end{array}\right.
\end{equation}
with boundary condition $\psi(t,x,\pm1)=\phi(t,x,\pm1)=0$.
And then taking the Fourier transform in $x$ of the above equations and inverting the operator $(\partial_y^2-\alpha^2)$
we get for $\alpha\neq0$,
\begin{equation}\label{eq: Fourier equations2}
 \left\{\begin{array}{l}
\partial_t\widehat{\psi}+i\alpha u\widehat{\psi}-i\alpha b\widehat{\phi}=2i\alpha(\partial_y^2-\alpha^2)^{-1}\big(u''\widehat{\psi}
-b''\widehat{\phi}
+u'\partial_y\widehat{\psi}-b'\partial_y\widehat{\phi}\big),\\
\partial_t\widehat{\phi}+i\alpha u\widehat{\phi}-i\alpha b\widehat{\psi}=0.
\end{array}\right.
\end{equation}
Let
\beq\label{eq: matrix}
M_{\al}=-\Delta_{\al}^{-1}
\left[
\begin{matrix}
& u''-u\Delta_{\al} & -b''+b\Delta_{\al}\\
& b\Delta_{\al}+b''+2b'\pa_y & -u\Delta_{\al}-u''-2u'\pa_y
\end{matrix}
\right]
\eeq
where $\Delta_{\al}=\pa_y^2-\al^2$ and its inverse $\Delta_{\al}^{-1}$ satisfies $(\pa_y^2-\al^2)\Delta_{\al}^{-1}\psi(\al,y)=\psi(\al,y)$ with the boundary value
$\Delta_{\al}^{-1}\psi(\al,y)|_{y=\pm 1}=0$. Then
\begin{equation}\label{eq: matrix-form}
\partial_t\Big(
  \begin{array}{ccc}
    \widehat{\psi}\\
    \widehat{\phi}\\
  \end{array}
\Big)(t,\al,y)=-i\alpha M_{\al}\Big(
  \begin{array}{ccc}
     \widehat{\psi}\\
    \widehat{\phi}\\
  \end{array}
\Big)(t,\al,y).
\end{equation}

The study of the (in)stability of Alfv\'{e}n waves for MHD equations is a very active field in physics and mathematics and there is a number of works in this field \cite{Biskamp, PAD01, HP83, AL89,  CU72, HZ15}.

If the equilibrium is not flowing ($U_s\equiv 0$), then for the Alfv\'{e}n waves in homogeneous magnetic fields (the case $b(y)=1$), it is easy to obtain the linear stability by using the fact that the current density $j$ and the vorticity  $w$ satisfy the wave equation
\beqno
\pa_{tt}j-\pa_{xx}j=0,\ \ \pa_{tt}w-\pa_{xx}w=0,
\eeqno
with initial data $(j,w)(0,x,y)=(j_0,w_0)$ and $(\pa_tj, \pa_tw)(0,x,y)=(\pa_xw_0, \pa_xj_0)$. The above fact also implies that there is no hope of obtaining a decay estimate of the velocity. For nonlinear global stability in the homogeneous case, we refer the reader to \cite{BSS88}. For the nonresistive MHD equation in which the termm $-\mu\Delta U$ appears in (\ref{eq: ideal MHD}), we refer the reader to \cite{CMRR16, FMRR14, RW16, WZ17} for the local well-posedness results. The global well-posedness and stability results may be found in \cite{RWXZ14, TZ16} for the 2D case and \cite{AZ17,DZ17} for the 3D case. For the fully diffusive MHD equation in which both the terms $-\mu\Delta U$ and $-\nu\Delta H$ appear in (\ref{eq: ideal MHD}), we refer the reader to \cite{CL18, HXY18, WZ17}. For Alfv\'{e}n waves in inhomogeneous magnetic fields, there are few rigorous mathematical results. Grossmann and Tataronis \cite{GT72, TG72, TG73} predicted that the decay rate of the velocity is $O(t^{-1})$. Recently Ren and Zhao \cite{RZ17} gave a rigorous proof for the strict monotone positive magnetic field case in a finite channel. The mechanism leading to the damping is the phase mixing phenomenon, which is similar to the well-known Landau damping found by Landau in 1946 \cite{LL46} and proved by Mouhot and Villani in their remarkable work \cite{MV11}. This common phenomenon also appears in the Euler equation, which is called the inviscid damping. One may refer to \cite{BM10,KMC60, WO07, RS66, SAS95, WZZ18, WZZ1, WZZ2, CZ17} for the linear inviscid damping results and to \cite{ BM15, LZ11} for the nonlinear inviscid damping results.

If the plasma is flowing, the long-time behavior of the solution to linearized equation is not easy as before. One of the reasons is that $M_{\al}$ is not a self-adjoint operator when $u\neq 0$. Another reason is from the physical observation: the reconnection phenomenon. Reconnection of field lines is the process by which the topology of a flux surface structure in a plasma can change. For more details about the stability of Alfv\'en waves in a flowing plasma, we refer to \cite{Hameiri1983, ShiTok2010}. By using Fourier-Laplace analysis, Hirota, Tatsuno and Yoshida \cite{HTY05} studied a special case ($(u(y),b(y))=(k_1y,k_2y)$) and gave a formal analysis about the asymptotic behavior and predicted the existence of the magnetic island by assuming linear profiles of the ambient magnetic field and flow $(u(y),b(y))=(k_1y,k_2y)$ with $k_1,k_2=const$. The linear system is spectral stable, when $0\leq k_1<k_2$, however because of the existence of the magnetic island, the topology of the magnetic field lines in the final state may be different from the initial magnetic field lines. This instability, which tears and reconnects field lines, is called a tearing mode \cite{HZ15}.

In this paper, we study a more general case and provide a justification for Hirota, Tatsuno and Yoshida's prediction about the generation of magnetic island. We also carefully study the behavior of the magnetic island, which may be used as a modifying factor, when we linearized the MHD system and study the nonlinear instability of the Alfv\'en waves in a flowing plasma in the future.

Now we introduce some conditions on the background  magnetic and velocity fields.
\begin{itemize}
\item[1.] Regularity, $(\bf R)$ : $u(y), b(y)\in C^5\left([-1,1]\right)$,
\item[2.] Stern stability, $(\bf SS)$: $|u(y)|\leq |b(y)|$,
\item[3.] Island, $(\bf I)$: $b(0)=u(0)=0$,
\item[4.] Monotone, $(\bf M)$: $b'(y)-|u'(y)|\geq c_0>0$
for some positive constant $c_0$.
\end{itemize}
One may regard $u(y)=ky, b(y)=k_0y$ with $k_0>|k|>0$ as an example. It is easy to check that $(\bf I)$
and $(\bf M)$ imply $(\bf SS)$.
\begin{theorem}\label{main thm}
Assume that $u(y), b(y)$ satisfy $(\bf R), (\bf I)$ and $(\bf M)$ and let $\big(\psi(t,x,y),\phi(t,x,y)\big)$ be the solution of (\ref{eq: psi phi}) with initial data $(\psi_0,\phi_0)\in H^{3}(-1,1)\times H^{4}(-1,1)$.
There holds that, \\
1. for $y=0$,  as $t\rightarrow+\infty$ and $\al\neq0$
\begin{align*}
&\widehat{\psi}(t,\al,0)\rightarrow\frac{u'(0)}
{b'(0)}\widehat{\phi}_0(\al,0),\\
&\widehat{\phi}(t,\al,0)\equiv\widehat{\phi}_0(\al,0);
\end{align*}
2. for $0<y\leq 1$, as $t\rightarrow+\infty$, there exists $\Gamma^{+}(\al,y)$ such that
\begin{align*}
\widehat{\psi}(t,\al, y)&\rightarrow -\frac{u(y)}{b(y)}\big(b(y)\Gamma^{+}(\al,y)\big)\widehat{\phi}_0(\al,0),
\\
\widehat{\phi}(t,\al, y)&\rightarrow -\big(b(y)\Gamma^{+}(\al,y)\big)\widehat{\phi}_0(\al,0);
\end{align*}
3. for $-1\leq y<0$, as $t\rightarrow+\infty$, there exists $\Gamma^{-}(\al,y)$ such that
\begin{align*}
\widehat{\psi}(t,\al, y)&\rightarrow -\frac{u(y)}{b(y)}\big(b(y)\Gamma^{-}(\al,y)\big)\widehat{\phi}_0(\al,0),\\
\widehat{\phi}(t,\al, y)&\rightarrow -\big(b(y)\Gamma^{-}(\al,y)\big)\widehat{\phi}_0(\al,0).
\end{align*}
Moreover we have
\beno
\Gamma^{\pm}(\al,y)=\f{\va_{\pm}(\al,y)(u'(0)^2-b'(0)^2)}{b'(0)}
\int_{\pm1}^y\frac{1}{\big(u(y')^2-b(y')^2\big)\va_{\pm}(\al,y')^2}dy',
\eeno
where $\va_{\pm}$ solves $\pa_y\left((u^2-b^2)\pa_y\va_{\pm}\right)-\al^2(u^2-b^2)\va_{\pm}=0$ with boundary conditions $\va_{\pm}(\al,0)=1$ and $\pa_y\va_{\pm}(\al,0)=0$.
\end{theorem}
\begin{remark}\label{rmk:t revers}
Of course, by time-reversibility, Theorem \ref{main thm} is also true $t\to -\infty$. The limiting profile is independent of the initial stream function. As $V_2=-\pa_x\psi$ and $B_2=-\pa_x\phi$, the limiting profile of $\widehat{V_2}$ and $\widehat{B_2}$ are the same as $\widehat{\psi}$ and $\widehat{\phi}$.
\end{remark}

\begin{remark}\label{rmk: ident phi(t,0)}
The result $\widehat{\phi}(t,\al,0)\equiv \widehat{\phi}_0(\al,0)$ can be obtained by the equation \eqref{eq: Fourier equations2} and assumption ({\textbf{I}}).
\end{remark}

\begin{remark}\label{rmk: discontinuous}
By the fact that $\lim\limits_{y\to 0}b(y)\Gamma^{\pm}(\al,y)=-1$, we get that the limiting profile is continuous at $y=0$, precisely $\lim\limits_{y\to 0}\lim\limits_{t\to +\infty}\widehat{\psi}(t,\al,y)=\lim\limits_{t\to +\infty}\widehat{\psi}(t,\al,0)$. As for the its derivative, in \cite{HTY05}, the authors predict that the limiting profile has a derivative jump at $y=0$. Actually, we can get that if
\beno
-5u'(0)u''(0)+u''(0)b'(0)-u'(0)b''(0)+5b'(0)b''(0)\neq0,
\eeno
then there exists a positive constant $C$ such that
\beno
\big|\pa_y\big(b(y)\Gamma^{\pm}(\al,y)\big)\big|\geq C^{-1}\big(1+\big|\ln|y|\big|\big).
\eeno
That means in this case the final state of $\phi$ and $\psi$ does not decay to $W^{1,\infty}$, which may be useful in the study of the nonlinear instability.
\end{remark}

\begin{remark}\label{rmk: special case}
If $u(y)=ky$, $b(y)=k_0y$ for some constant $k_0>|k|\geq 0$, then $b(y)\Gamma^{\pm}(\al,y)$ are harmonic functions on $\mathbb{T}\times[0,\pm1]$ with boundary condition $b(0)\Gamma^{\pm}(\al,0)=-1$ and $b(\pm1)\Gamma^{\pm}(\al,\pm1)=0$. Then the final state is in $W^{1,\infty}$ and its profile is as follows.\\

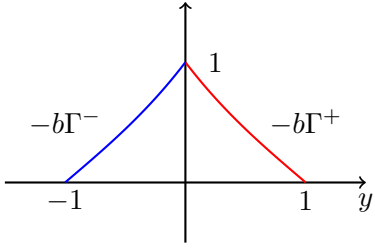
\begin{tikzpicture}[thick, scale=1.6]
\draw[->](-1.5,0)--(1.5,0) node[below]{$y$};
\draw[->](0,-0.5)--(0,1.5);
\draw[color=red] plot[domain=0:1](\x, {-exp(-2)/(1-exp(-2))*exp(\x)+1/(1-exp(-2))*exp(-\x)});
\draw[color=blue] plot[domain=-1:0](\x, {-exp(2)/(1-exp(2))*exp(\x)+1/(1-exp(2))*exp(-\x)});
\draw (1,0.5) node{$-b\Gamma^+$};
\draw (-1,0.5) node {$-b\Gamma^-$};
\draw (-1,-0.15) node {$-1$};
\draw (1,-0.15) node {$1$};
\draw (0.25,1) node {$1$};
\end{tikzpicture}

\end{remark}

\begin{remark}\label{rmk: open problem}
The necessary and sufficient condition of the generation of the magnetic island is an open problem. Roughly speaking, according to \eqref{eq: matrix-form}, assuming there is a magnetic island at the finial state, then there is non-trivial solution of $M_{\al}\left(\begin{aligned}\widehat{\psi}\\\widehat{\phi}\end{aligned}\right)=0$. Formally we can get that the necessary condition of the generation of the magnetic island may be $0\in \sigma(M_{\al})$.
\end{remark}

\begin{remark}\label{rmk: add}
The convergence rate will be discussed in a separated work.
\end{remark}
\section{reduce the problem and the Sturmian equations}
Let $\Om$ be the domain that contains $\sigma(M_{\al})$.
Then we have the following representation formula of the solution to (\ref{eq: matrix-form}):
\begin{equation}\label{eq: psi and phi}
\Big(
  \begin{array}{ccc}
    \widehat{\psi}\\
    \widehat{\phi}\\
  \end{array}
\Big)(t,\alpha,y)=\frac{1}{2\pi i}\int_{\partial\Omega}e^{-i\alpha tc}(cI-M_{\al})^{-1}\Big(
  \begin{array}{ccc}
    \widehat{\psi}\\
    \widehat{\phi}\\
  \end{array}
\Big)(0,\alpha,y)dc.
\end{equation}
Then the large time behavior of the solution $\Big(
  \begin{array}{ccc}
    \widehat{\psi}\\
    \widehat{\phi}\\
  \end{array}
\Big)(t,\alpha,y)$ is reduced to the study of the resolvent $(cI-M_{\al})^{-1}$.

Suppose $\big(cI-M_{\al}\big)^{-1}\Big(\begin{array}{l} \widehat{\psi}_0\\ \widehat{\phi}_0\end{array}\Big)(\al,y)=\Big(\begin{array}{l} \Psi_1\\ \Phi_1\end{array}\Big)(\al,y,c)$, then
\begin{align*}
&\big(cI-M_{\al}\big)\Big(\begin{array}{l} \Psi_1\\ \Phi_1\end{array}\Big)(\al,y,c)=\Big(\begin{array}{l} \widehat{\psi}_0\\ \widehat{\phi}_0\end{array}\Big)(\al,y)\\
&\Leftrightarrow \left\{
\begin{array}{l}
-c\Delta_{\al}\Psi_1+u\Delta_{\al}\Psi_1
-u''\Psi_1-b\Delta_{\al}\Phi_1+b''\Phi_1
=\widehat{\om}_0,\\
-c\Delta_{\al}\Phi_1-b\Delta_{\al}\Psi_1-b''\Psi_1
+2u'\pa_y\Phi_1
+u\Delta_{\al}\Phi_1+u''\Phi_1-2b'\pa_y\Psi_1
=\widehat{j}_0.
\end{array}
\right.\\
&\Leftrightarrow \left\{
\begin{array}{l}
(u-c)\Delta_{\al}\Psi_1-u''\Psi_1-b\Delta_{\al}\Phi_1
+b''\Phi_1
=\widehat{\om}_0,\\
(u-c)\Phi_1-b\Psi_1=-\widehat{\phi}_0.
\end{array}
\right.
\end{align*}
Here $\Delta_{\al}=\partial_y^2-\al^2$.
Let $\Phi_1(\al,y,c)=b(y)\Phi_2(\al,y,c)$ then $\Psi_1(\al,y,c)=(u(y)-c)\Phi_2(\al,y,c)+\widehat{\phi}_0(\al,y)/b(y)$, then we get
\beq\label{eq:ODE1}
\begin{split}
&\pa_y\Big[\Big(\big(u(y)-c\big)^2-b(y)^2\Big)\pa_y\Phi_2(\al, y,c)\Big]
-\al^2\Big(\big(u(y)-c\big)^2-b(y)^2\Big)\Phi_2(\al,y,c)\\
&
=\widehat{\om}_0(\al,y)-\big(u(y)-c\big)\Delta_{\al}\Big(\frac{\widehat{\phi}_0(\al,y)}{b(y)}\Big)
+u''(y)\frac{\widehat{\phi}_0(\al,y)}{b(y)}
\end{split}
\eeq
For this equation, by denoting  $\Phi_2(\al,y,c)=\Phi(\al,y,c)+\widehat{\phi}_0(0)\chi(y)/(cb(y))$
 whereas $0 \leq\chi(y)\in C_0^{\infty}(\R):
 \chi(y)=1$ for $|y|\leq\frac{1}{2}$ and $\chi(y)=0$ for $|y|\geq 3/4$, we can get that
\begin{align*}
&\pa_y\Big[\Big(\big(u(y)-c\big)^2-b(y)^2\Big)\pa_y\Phi(\al,y,c)\Big]
-\al^2\Big(\big(u(y)-c\big)^2-b(y)^2\Big)\Phi(\al,y,c)\\ &=\widehat{w}_0(\al,y)-\big(u(y)-c\big)\Delta_{\al}\Big(\frac{\widehat{\phi}_0(\al,y)
-\widehat{\phi}_0(\al,0)}{b(y)}\Big)
+u''(y)\frac{\widehat{\phi}_0(y)-\widehat{\phi}_0(\al,0)}{b(y)}\\
&\ \
-\frac{\widehat{\phi}_0(\al,0)}{c}\frac{1}{b(y)^3}
\Big\{2\big(u(y)-c\big)\big[c+\big(u(y)-c\big)\chi(y)\big]b'(y)^2\\
&\ \
-b(y)\big\{\big(u(y)-c\big)\big[c+\big(u(y)-c\big)\chi(y)\big]b''(y)
+2\big(u(y)-c\big)\big[\big(u(y)-c\big)\chi'(y)+u'(y)\chi(y)\big]b'(y)\big\}\\
&\ \ -b(y)^2\big\{\al^2\big(u(y)-c\big)\big[c+\big(u(y)-c\big)\chi(y)\big]
+cu''(y)-\big(u(y)-c\big)^2\chi''(y)\\
&\ \ -2\big(u(y)-c\big)u'(y)\chi'(y)\big\}
+b(y)^3b''(y)\chi(y)+b(y)^4(\al^2\chi(y)-\chi''(y))\Big\}\\
&\stackrel{def}{=}G(\al,y,c)
=G_1(\al,y,c)-\frac{\widehat{\phi}_0(\al,0)}{c}\frac{f(\al,y,c)}{b(y)^3},
\end{align*}
whereas
\beq\label{eq: G_1}
G_1(\al,y,c)=\widehat{\om}_0(\al,y)-(u(y)-c)\Delta_{\al}
\Big(\frac{\widehat{\phi}_0(y)-\widehat{\phi}_0(0)}{b(y)}\Big)
+u''(y)\frac{\widehat{\phi}_0(y)-\widehat{\phi}_0(0)}{b(y)}
\eeq
and
\beq\label{eq: f}
\begin{split}
f(\al,y,c)&=2\big(u(y)-c\big)\big[c+\big(u(y)-c\big)\chi(y)\big]b'(y)^2\\
& \ \
-b(y)\big\{\big(u(y)-c\big)\big[c+\big(u(y)-c\big)\chi(y)\big]b''(y)\\
& \ \
+2(u(y)-c)\big[\big(u(y)-c\big)\chi'(y)+u'(y)\chi(y)\big]b'(y)\big\}\\
&\ \
-b(y)^2\big\{\al^2\big(u(y)-c\big)\big[c+\big(u(y)-c\big)\chi(y)\big]
-\big(u(y)-c\big)^2\chi''(y)+cu''(y)\\
&\ \ -2\big(u(y)-c\big)u'(y)\chi'(y)\big\}
+b(y)^3b''(y)\chi(y)+b(y)^4\big(\al^2\chi(y)-\chi''(y)\big).
\end{split}
\eeq
On the one hand, it is easy to show that $f(\al,0,c)=0, (\pa_yf)(\al,0,c)=0$ and $(\pa_y^2f)(\al,0,c)=0$, Thus we can deduce that $f(\al,y,c)$ is $C^2$ function when $u(y), b(y)$ is $C^5$ function. Then we obtain that, due to the fact that $b(0)=0$,
\beqno
\frac{f(\al,y,c)}{b(y)^3}=\frac{y^3}{b(y)^3}\frac{f(\al,y,c)}{y^3}
=\frac{1}{\big(\int_0^1b'(sy)ds\big)^3}
\int_0^1t\int_0^t\int_0^s\pa_y^3f(\al,\tau y,c)d\tau ds dt
\eeqno
and
\beqno
\Big\|\frac{f(\al,y,c)}{b(y)^3}\Big\|_{L^{\infty}}\leq C.
\eeqno
On the other hand, thanks to the fact that
$\frac{\widehat{\phi}_0(y)-\widehat{\phi}_0(0)}{b(y)}
=\frac{\int_0^1(\pa_y\widehat{\phi}_0)(sy)ds}{\int_0^1b'(sy)ds},$ we have,
\beqno
\|G_1(y,c)\|_{H_y^1}\leq C\|\widehat{\om}_0\|_{H_y^{1}}+C\|\widehat{j}_0\|_{H_y^{3}}.
\eeqno

\section{the homogeneous Sturmian equations}
To solve the inhomogeneous Sturmian equation, we first construct two regular solutions of the homogeneous Sturmian equation:
\begin{align*}
& \pa_y\big[\big(u(y)+b(y)-c\big)\big(u(y)-b(y)-c\big)\pa_y\va(\al, y, c)\big]\nonumber\\
&\quad\quad
-\al^2\big(u(y)+b(y)-c\big)\big(u(y)-b(y)-c\big)\va(\al, y, c)=0.
\end{align*}
We will suppress the variable $\al$ for simplicity and write the function $\va(y,c)=\va(\al,y,c)$.

Let $W_+(y)=u(y)+b(y)$, $W_-(y)=u(y)-b(y)$, then
from $(\bf M)$, we get $W'_+(y)>0$, $W'_-(y)<0$. And we define $\mathcal{H}(y,c)=(W_+(y)-c)(W_-(y)-c)$,
where the constant coefficient $c$ will be taken in domain: 
$c\in \Omega_{\ep_0}=\{z\in \mathbb{C}, dist(z,Ran\, W_+ \cup Ran\, W_-)\leq \ep_0\}$(the $\ep_0$ neighborhood of $Ran\, W_+ \cup Ran\, W_-$). 
According to the relationship between $W_+(1), W_-(1), W_+(-1)$ and $W_-(-1)$, the domain have the following nine cases:

\begin{itemize}
\item[1.] $W_+(1)> W_-(-1)>0> W_-(1)>W_+(-1)$,
\item[2.] $W_+(1)= W_-(-1)>0> W_-(1)>W_+(-1)$,
\item[3.] $W_-(-1)> W_+(1)>0> W_-(1)>W_+(-1)$,
\item[4.] $W_+(1)> W_-(-1)>0> W_-(1)=W_+(-1)$,
\item[5.] $W_+(1)> W_-(-1)>0> W_+(-1)>W_-(1)$,
\item[6.] $W_+(1)=W_-(-1)>0> W_-(1)=W_+(-1)$,
\item[7.] $W_+(1)= W_-(-1)>0> W_+(-1)>W_-(1)$,
\item[8.] $W_-(-1)> W_+(1)>0> W_-(1)=W_+(-1)$,
\item[9.] $W_-(-1)> W_+(1)>0> W_+(-1)>W_-(1)$.
\end{itemize}
Before we introduce all these cases, we first introduce the extension lemma. 
\begin{lemma}\label{lem: extend}
Let $f$ be a $C^k$ function defined on $[a,b]$ with $f'(b)>0$. Assume $M>f(b)$. Then for any $d>b$, there exists $F\in C^{k}([a,d])$ such that $F'(x)>0$ for $x\in [b,d]$ and $F(x)=f(x)$ for $x\in [a,b]$ and $F^{(m)}(b)=f^{(m)}(b)$ for $m=1,...,k$. Moreover, it holds that $\|F\|_{C^k}\leq C\|f\|_{C^k}$. 
\end{lemma}
\begin{proof}
Let $\delta_1<\f{d-b}{4}$ be small enough which will be determined later and $0\leq \chi(x)\leq 1$ be a smooth non-negative function with compact support satisfying $|\chi'(x)|\leq C\delta_1^{-1}$,  $\chi(x)=1$ for $0\leq x\leq \f14\delta_1$ and $\chi(x)=0$ for $x\geq \f34\delta_1$. Define for $x\in [b, b+\delta_1]$,
\beno
F(x)=f(b)+f'(b)(x-b)+\sum_{n=2}^{k}\f{f^{(n)}(b)}{n!}(x-b)^{n}\chi(x)
\eeno
Then $F'(x)=f'(b)+\sum_{n=2}^{k}\f{f^{(n)}(b)}{(n-1)!}(x-b)^{n-1}\chi(x)+\sum_{n=2}^{k}\f{f^{(n)}(b)}{n!}(x-b)^{n}\chi'(x-b)$, by the fact that $|\chi'(x)|\leq 1/\delta_1$, we have for $|x-b|\leq \delta_1$
\beno
F'(x)\geq f'(b)-C\max_{n=2,...,k}|f^{(n)}(b)|\delta_1.
\eeno
and $F(b+\delta_1)=f(b)+f'(b)\delta_1$, $F'(b+\delta_1)=f'(b)$ and $F^{(n)}(b+\delta_1)=0$ for $n\geq 2$.
By taking $\delta_1$ small enough such that $C\max_{n=2,...,k}|f^{(n)}(b)|\delta_1\leq \f12f'(b)$ and $5f'(b)\delta_1<M-f(b)$, then we have for $|x-b|\leq \delta_1$, $F'(x)\geq \f12f'(b)$ and $F(b+\delta_1)<M$.

Let $g(x)>0$ be a smooth function, such that $g(x)=f'(b)$ for $x\in [b+\delta_1,b+2\delta_1]$ and
\beno
\int_{b+\delta_1}^dg(x)dx=M-F(b+\delta_1).
\eeno
Let $F(x)=F(b+\delta_1)+\int_{b+\delta_1}^xg(x')dx'$ for $x\in [b+\delta_1,d]$, then $F(x)$ is the extension function.
\end{proof}
For {\bf Case 1}, we let 
\beno
&\ & D_0\stackrel{def}{=}\big\{c\in[W_+(-1), W_+(1)]\big\},\\
&\ & D_{\ep_0}\stackrel{def}{=}\big\{c=c_r+i\ep, \ c_r\in[W_+(-1), W_+(1)], \ 0<|\ep|<\ep_0\big\},\\
&\ & B_{\ep_0}^{l}\stackrel{def}{=}\big\{c=W_+(-1)+\ep e^{i\theta},\  0<\ep<\ep_0, \ \frac{\pi}{2}\leq\theta\leq\frac{3\pi}{2}\big\},\\
&\ & B_{\ep_0}^{r}\stackrel{def}{=}\big\{c=W_+(1)-\ep e^{i\theta}, \ 0<\ep<\ep_0, \ \frac{\pi}{2}\leq\theta\leq\frac{3\pi}{2}\big\},
\eeno
for some $\ep_0\in(0,1).$ We denote
$
\Omega_{\ep_0}\stackrel{def}{=}D_0\cup D_{\ep_0}\cup B_{\ep_0}^{l}\cup B_{\ep_0}^{r}.
$
We also define
\beqno
c_r=Re\, c\ for\  c\in D_0\cup D_{\ep_0}, \ c_r=W_+(-1)
\ for \ c\in B_{\ep_0}^{l}, \ c_r=W_+(1) \ for \ c\in B_{\ep_0}^{r}.
\eeqno
By Lemma \ref{lem: extend}, we can take a $C^5$ extension of $W_-$ to be $\wt{W}_-$ for $y\in[\rm{a}_-,\rm{a}_+]$ such that $\wt{W}_-(\rm{a}_-)=W_+(1)$, $\wt{W}_-(\rm{a}_+)=W_+(-1)$ and $\wt{W}_-'(y)<0$. 
\begin{itemize}
\item For $c\in D_0\cup D_{\ep_0}$ and $c_r\geq 0$, we denote $y_{c_+}\in [0,1]$ with $y_{c_+}=(W_+)^{-1}(c_r)$, so that $W_+(y_{c_+})-c_r=0$ and $y_{c_-}\in [\rm{a}_-,0]$ with $y_{c_-}=(\wt{W}_-)^{-1}(c_r)$, so that $\wt{W}_-(y_{c_-})-c_r=0$. 
\item For $c\in D_0\cup D_{\ep_0}$ and $c_r\leq 0$, we denote $y_{c_+}\in [0,\rm{a}_+]$ with $y_{c_+}=(\wt{W}_-)^{-1}(c_r)$, so that $\wt{W}_-(y_{c_+})=c_r$ and $y_{c_-}\in [-1,0]$ with $y_{c_-}=(W_+)^{-1}(c_r)$, so that $W_+(y_{c_-})-c_r=0$. 
\item For $c\in B_{\ep_0}^{l}$ then $c_r=W_+(-1)=\wt{W}_-(\rm{a}_+)$, we denote $y_{c_+}=\rm{a}_+$ and $y_{c_-}=-1$. 
\item For $c\in B_{\ep_0}^r$ then $c_r=W_+(1)=\wt{W}_-(a_-)$, we denote $y_{c_+}=1$ and $y_{c_-}=\rm{a}_-$.
\end{itemize}

We show the relationship between $y_{c_+}$, $y_{c_-}$ and $c_r$ by the following picture for the above case. 

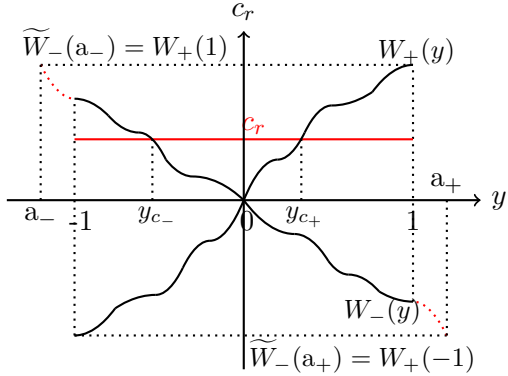
\begin{tikzpicture}[thick, scale=0.45]
    \draw[very thin,color=gray];
    \draw[->] (-7,0) -- (7,0) node[right] {$y$};
    \draw	(0.1,-0.01) node[anchor=north] {0}
		(5,0) node[anchor=north] {1}
        (-4.8,0) node[anchor=north] {-1};
        
    \draw[dotted] (-6,4) -- (5,4);
    \draw[dotted] (-5,-4) -- (6,-4);
    \draw[dotted] (-5,-4) -- (-5,3);
    \draw[dotted] (5,-3) -- (5,4);
    
    \draw[red,thick] (-5,1.8) -- (5,1.8);
    \draw[red] (0.3,2.2) node {$c_r$};
    \draw[->] (0,-5) -- (0,5) node[above] {$c_r$};
    \draw[thick] (0,0) parabola[bend at end] (1,1.2);
    \draw[thick] (1.8,2) parabola[bend at end] (1,1.2);
    \draw[thick] (1.8,2) parabola[bend at end] (2.6,2.8);
     \draw[thick] (3.8,3.2) parabola[bend at end] (2.6,2.8);
     \draw[thick] (3.8,3.2) parabola[bend at end] (5,4);
     
      \draw[thick] (0,0) parabola[bend at end] (-1,-1.2);
    \draw[thick] (-1.8,-2) parabola[bend at end] (-1,-1.2);
    \draw[thick] (-1.8,-2) parabola[bend at end] (-2.6,-2.8);
     \draw[thick] (-3.8,-3.2) parabola[bend at end] (-2.6,-2.8);
     \draw[thick] (-3.8,-3.2) parabola[bend at end] (-5,-4);

     \draw[thick] (0,0) parabola[bend at end] (-1.5,0.7);
    \draw[thick] (-2.3,1.2) parabola[bend at end] (-1.5,0.7);
    \draw[thick] (-2.3,1.2) parabola[bend at end] (-3.1,2.0);
     \draw[thick] (-4,2.5) parabola[bend at end] (-3.1,2.0);
     \draw[thick] (-4,2.5) parabola[bend at end] (-5,3);
    \draw[red,dotted] (-6,4) parabola[bend at end] (-5,3);
    
    \draw[thick] (0,0) parabola[bend at end] (1.5,-1.0);
    \draw[thick] (2.5,-1.8) parabola[bend at end] (1.5,-1.0);
    \draw[thick] (2.5,-1.8) parabola[bend at end] (3.1,-2.2);
     \draw[thick] (4,-2.7) parabola[bend at end] (3.1,-2.2);
     \draw[thick] (4,-2.7) parabola[bend at end] (5,-3);
     \draw[red, dotted] (6,-4) parabola[bend at end] (5,-3);
\draw (5.1,4.4) node {\small $W_+(y)$};
\draw (4.1,-3.3) node {\small $W_-(y)$};
 \draw[dotted] (-2.7,0) -- (-2.7,1.8);
 \draw (-2.55,-0.5) node {\small $y_{c_-}$};
  \draw[dotted] (1.7,0) -- (1.7,1.8);
  \draw (1.8,-0.5) node {\small $y_{c_+}$};
  \draw[dotted] (6,-4)-- (6,0); \draw (6,0.5) node {$\rm{a}_+$};
  \draw[dotted] (-6,4)-- (-6,0); \draw (-6,-0.5) node {$\rm{a}_-$};
  \draw (-3.5,4.6) node {\small $\wt{W}_-({\rm{a}}_-)=W_+(1)$};
\draw (3.5,-4.6) node {\small $\wt{W}_-({\rm{a}}_+)=W_+(-1)$};
\end{tikzpicture}

We also take a $C^5$ extension of $W_+$ to be $\wt{W}_+$ for $y\in [\rm{a}_-,{\rm{a}}_+]$, so that $\wt{W}_+'(y)>0$.

For {\bf{Case 2}}, we let
\beno
&\ & D_0\stackrel{def}{=}\big\{c\in[W_+(-1), W_+(1)]\big\},\\
&\ & D_{\ep_0}\stackrel{def}{=}\big\{c=c_r+i\ep, \ c_r\in[W_+(-1), W_+(1)], \ 0<|\ep|<\ep_0\big\},\\
&\ & B_{\ep_0}^{l}\stackrel{def}{=}\big\{c=W_+(-1)+\ep e^{i\theta},\  0<\ep<\ep_0, \ \frac{\pi}{2}\leq\theta\leq\frac{3\pi}{2}\big\},\\
&\ & B_{\ep_0}^{r}\stackrel{def}{=}\big\{c=W_+(1)-\ep e^{i\theta}, \ 0<\ep<\ep_0, \ \frac{\pi}{2}\leq\theta\leq\frac{3\pi}{2}\big\},
\eeno
for some $\ep_0\in(0,1).$ We denote
$
\Omega_{\ep_0}\stackrel{def}{=}D_0\cup D_{\ep_0}\cup B_{\ep_0}^{l}\cup B_{\ep_0}^{r}.
$
We also define
\beqno
c_r=Re c\ for\  c\in D_0\cup D_{\ep_0}, \ c_r=W_+(-1)
\ for \ c\in B_{\ep_0}^{l}, \ c_r=W_+(1) \ for \ c\in B_{\ep_0}^{r}.
\eeqno
By Lemma \ref{lem: extend}, we can take a $C^5$ extension of $W_-$ to be $\wt{W}_-$ for $y\in[-1,{\rm{a}}_+]$ such that $\wt{W}_-({\rm{a}}_+)=W_+(-1)$ and $\wt{W}'_-(y)<0$.
\begin{itemize}
\item For $c\in D_0\cup D_{\ep_0}$ and $c_r\geq0$, we denote $y_{c_+}\in[0,1]$ with $y_{c_+}=(W_+)^{-1}(c_r)$, so that $W_+(y_{c_+})-c_r=0$ and $y_{c_-}\in[-1,0]$ with $y_{c_-}=(W_-)^{-1}(c_r)$, so that $W_-(y_{c_-})-c_r=0$.
\item For $c\in D_0\cup D_{\ep_0}$ and $c_r\leq0$, we denote $y_{c_+}\in[0,{\rm{a}}_+]$ with $y_{c_+}=(\wt{W}_-)^{-1}(c_r)$, so that $\wt{W}_-(y_{c_+})-c_r=0$ and $y_{c_-}\in[-1,0]$ with $y_{c_-}=(W_+)^{-1}(c_r)$, so that $W_+(y_{c_-})-c_r=0$.
\item For $c\in B_{\ep_0}^l$, then $c_r=W_+(-1)=\wt{W}_-({\rm{a}}_+)$, we denote $y_{c_+}={\rm{a}}_+$ and $y_{c_-}=-1$.
\item For $c\in B_{\ep_0}^r$, then $c_r=W_+(1)=W_-(-1)$, we denote $y_{c_+}=1$ and $y_{c_-}=-1$.
\end{itemize}

We show the relationship between $y_{c_+}, y_{c_-}$ and $c_r$ by the following picture for the above case.

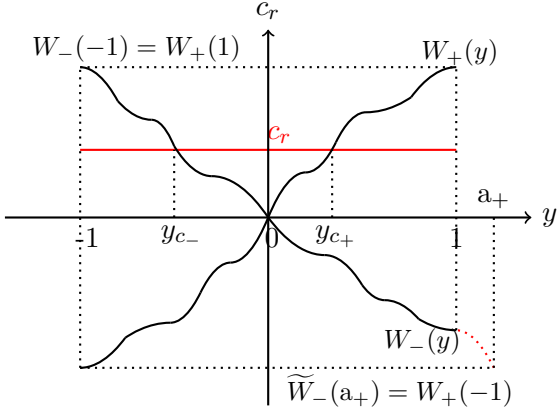
\begin{tikzpicture}[thick, scale=0.5]
    \draw[very thin,color=gray];
    \draw[->] (-7,0) -- (7,0) node[right] {$y$};
    \draw	(0.1,-0.01) node[anchor=north] {0}
		(5,0) node[anchor=north] {1}
        (-4.8,0) node[anchor=north] {-1};

    \draw[dotted] (-5,4) -- (5,4);
    \draw[dotted] (-5,-4) -- (6,-4);
    \draw[dotted] (-5,-4) -- (-5,4);
    \draw[dotted] (5,-3) -- (5,4);

    \draw[red,thick] (-5,1.8) -- (5,1.8);
    \draw[red] (0.3,2.2) node {$c_r$};
    \draw[->] (0,-5) -- (0,5) node[above] {$c_r$};
    \draw[thick] (0,0) parabola[bend at end] (1,1.2);
    \draw[thick] (1.8,2) parabola[bend at end] (1,1.2);
    \draw[thick] (1.8,2) parabola[bend at end] (2.6,2.8);
     \draw[thick] (3.8,3.2) parabola[bend at end] (2.6,2.8);
     \draw[thick] (3.8,3.2) parabola[bend at end] (5,4);

      \draw[thick] (0,0) parabola[bend at end] (-1,-1.2);
    \draw[thick] (-1.8,-2) parabola[bend at end] (-1,-1.2);
    \draw[thick] (-1.8,-2) parabola[bend at end] (-2.6,-2.8);
     \draw[thick] (-3.8,-3.2) parabola[bend at end] (-2.6,-2.8);
     \draw[thick] (-3.8,-3.2) parabola[bend at end] (-5,-4);
     
\draw[thick] (0,0) parabola[bend at end] (-1.5,1.1);
    \draw[thick] (-2.5,1.9) parabola[bend at end] (-1.5,1.1);
    \draw[thick] (-2.5,1.9) parabola[bend at end] (-3.1,2.6);
     \draw[thick] (-4,3.1) parabola[bend at end] (-3.1,2.6);
     \draw[thick] (-4,3.1) parabola[bend at end] (-5,4);

    \draw[thick] (0,0) parabola[bend at end] (1.5,-1.0);
    \draw[thick] (2.5,-1.8) parabola[bend at end] (1.5,-1.0);
    \draw[thick] (2.5,-1.8) parabola[bend at end] (3.1,-2.2);
     \draw[thick] (4,-2.7) parabola[bend at end] (3.1,-2.2);
     \draw[thick] (4,-2.7) parabola[bend at end] (5,-3);
     \draw[red, dotted] (6,-4) parabola[bend at end] (5,-3);
\draw (5.1,4.4) node {\small $W_+(y)$};
\draw (4.1,-3.3) node {\small $W_-(y)$};
 \draw[dotted] (-2.5,0) -- (-2.5,1.8);
 \draw (-2.3,-0.5) node {$y_{c_-}$};
  \draw[dotted] (1.7,0) -- (1.7,1.8);
  \draw (1.8,-0.5) node {$y_{c_+}$};
  \draw[dotted] (6,-4)-- (6,0); \draw (6,0.5) node {${\rm{a}}_+$};
  \draw (-3.5,4.5) node {\small $W_-(-1)=W_+(1)$};
\draw (3.5,-4.6) node {\small $\wt{W}_-({\rm{a}}_+)=W_+(-1)$};
\end{tikzpicture}

We also take a $C^5$ extension of $W_+$ to be $\wt{W}_+$ for $y\in [-1,{\rm{a}}_+]$, so that $\wt{W}_+'(y)>0$.

For {\bf{Case 3}}, we let
\beno
&\ & D_0\stackrel{def}{=}\big\{c\in[W_+(-1), W_-(-1)]\big\},\\
&\ & D_{\ep_0}\stackrel{def}{=}\big\{c=c_r+i\ep, \ c_r\in[W_+(-1), W_-(-1)], \ 0<|\ep|<\ep_0\big\},\\
&\ & B_{\ep_0}^{l}\stackrel{def}{=}\big\{c=W_+(-1)+\ep e^{i\theta},\  0<\ep<\ep_0, \ \frac{\pi}{2}\leq\theta\leq\frac{3\pi}{2}\big\},\\
&\ & B_{\ep_0}^{r}\stackrel{def}{=}\big\{c=W_-(-1)-\ep e^{i\theta}, \ 0<\ep<\ep_0, \ \frac{\pi}{2}\leq\theta\leq\frac{3\pi}{2}\big\},
\eeno
for some $\ep_0\in(0,1).$ We denote
$
\Omega_{\ep_0}\stackrel{def}{=}D_0\cup D_{\ep_0}\cup B_{\ep_0}^{l}\cup B_{\ep_0}^{r}.
$
We also define
\beqno
c_r=Re c\ for\  c\in D_0\cup D_{\ep_0}, \ c_r=W_+(-1)
\ for \ c\in B_{\ep_0}^{l}, \ c_r=W_-(-1) \ for \ c\in B_{\ep_0}^{r}.
\eeqno
By Lemma \ref{lem: extend}, we can take a $C^5$ extension of $W_-$ to be $\wt{W}_-$  and $W_+$ to be $\wt{W}_+$ for $y\in[-1,{\rm{a}}_+]$ such that $\wt{W}_-({\rm{a}}_+)=W_+(-1)$, $\wt{W}'_-(y)<0$ and  $\wt{W}_+({\rm{a}}_+)=W_-(-1)$, $\wt{W}'_+(y)>0$.
\begin{itemize}
\item For $c\in D_0\cup D_{\ep_0}$ and $c_r\geq0$, we denote $y_{c_+}\in[0,{\rm{a}}_+]$ with $y_{c_+}=(\wt{W}_+)^{-1}(c_r)$, so that $\wt{W}_+(y_{c_+})-c_r=0$ and $y_{c_-}\in[-1,0]$ with $y_{c_-}=(W_-)^{-1}(c_r)$, so that $W_-(y_{c_-})-c_r=0$.
\item For $c\in D_0\cup D_{\ep_0}$ and $c_r\leq0$, we denote $y_{c_+}\in[0,{\rm{a}}_+]$ with $y_{c_+}=(\wt{W}_-)^{-1}(c_r)$, so that $\wt{W}_-(y_{c_+})-c_r=0$ and $y_{c_-}\in[-1,0]$ with $y_{c_-}=(W_+)^{-1}(c_r)$, so that $W_+(y_{c_-})-c_r=0$.
\item For $c\in B_{\ep_0}^l$, then $c_r=W_+(-1)=\wt{W}_-({\rm{a}}_+)$, we denote $y_{c_+}={\rm{a}}_+$ and $y_{c_-}=-1$.
\item For $c\in B_{\ep_0}^r$, then $c_r=W_-(-1)=\wt{W}_+({\rm{a}}_+)$, we denote $y_{c_+}={\rm{a}}_+$ and $y_{c_-}=-1$.
\end{itemize}

We show the relationship between $y_{c_+}, y_{c_-}$ and $c_r$ by the following picture for the above case.

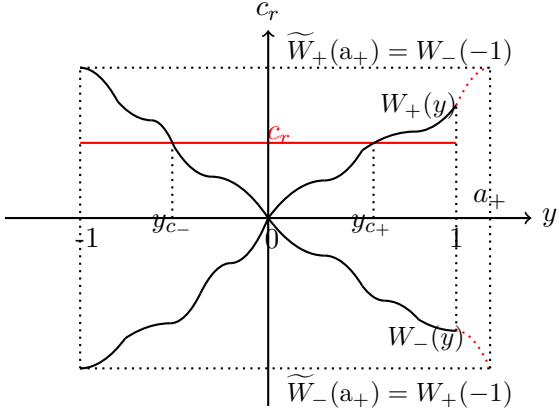
\begin{tikzpicture}[thick, scale=0.5]
    \draw[very thin,color=gray];
    \draw[->] (-7,0) -- (7,0) node[right] {$y$};
    \draw	(0.1,-0.01) node[anchor=north] {0}
		(5,0) node[anchor=north] {1}
        (-4.8,0) node[anchor=north] {-1};
    \draw[dotted] (-5,4) -- (5.9,4);
    \draw[dotted] (-5,-4) -- (5.9,-4);
    \draw[dotted] (-5,-4) -- (-5,4);
    \draw[dotted] (5,-3) -- (5,3);

    \draw[red,thick] (-5,2) -- (5,2);
    \draw[red] (0.3,2.2) node {$c_r$};
    \draw[->] (0,-5) -- (0,5) node[above] {$c_r$};

    \draw[thick] (0,0) parabola[bend at end] (1.5,1.0);
    \draw[thick] (2.5,1.8) parabola[bend at end] (1.5,1.0);
    \draw[thick] (2.5,1.8) parabola[bend at end] (3.9,2.3);
     \draw[thick] (5,3) parabola[bend at end] (3.9,2.3);
      \draw[red,dotted] (5,3) parabola[bend at end] (5.9,4);

      \draw[thick] (0,0) parabola[bend at end] (-1,-1.2);
    \draw[thick] (-1.8,-2) parabola[bend at end] (-1,-1.2);
    \draw[thick] (-1.8,-2) parabola[bend at end] (-2.6,-2.8);
     \draw[thick] (-3.8,-3.2) parabola[bend at end] (-2.6,-2.8);
     \draw[thick] (-3.8,-3.2) parabola[bend at end] (-5,-4);

     \draw[thick] (0,0) parabola[bend at end] (-1.5,1.1);
    \draw[thick] (-2.5,1.9) parabola[bend at end] (-1.5,1.1);
    \draw[thick] (-2.5,1.9) parabola[bend at end] (-3.1,2.6);
     \draw[thick] (-4,3.1) parabola[bend at end] (-3.1,2.6);
     \draw[thick] (-4,3.1) parabola[bend at end] (-5,4);

     \draw[thick] (0,0) parabola[bend at end] (1.5,-1.0);
    \draw[thick] (2.5,-1.8) parabola[bend at end] (1.5,-1.0);
    \draw[thick] (2.5,-1.8) parabola[bend at end] (3.1,-2.2);
     \draw[thick] (4,-2.7) parabola[bend at end] (3.1,-2.2);
     \draw[thick] (4,-2.7) parabola[bend at end] (5,-3);
     \draw[red, dotted] (5.9,-4) parabola[bend at end] (5,-3);

\draw (4.0,3.0) node {\small $W_+(y)$};
\draw (4.2,-3.2) node {\small $W_-(y)$};
 \draw[dotted] (-2.55,0) -- (-2.55,2);
 \draw (-2.55,-0.2) node {$y_{c_-}$};
  \draw[dotted] (2.8,0) -- (2.8,2);
  \draw (2.75,-0.2) node {$y_{c_+}$};
  \draw[dotted] (5.9,-4)-- (5.9,4); \draw (5.9,0.5) node {$a_+$};
  \draw (3.5,4.5) node {\small $\wt{W}_+({\rm{a}}_+)=W_-(-1)$};
\draw (3.5,-4.6) node {\small $\wt{W}_-({\rm{a}}_+)=W_+(-1)$};
\end{tikzpicture}

For {\bf{Case 4}}, we let
\beno
&\ & D_0\stackrel{def}{=}\big\{c\in[W_+(-1), W_+(1)]\big\},\\
&\ & D_{\ep_0}\stackrel{def}{=}\big\{c=c_r+i\ep, \ c_r\in[W_+(-1), W_+(1)], \ 0<|\ep|<\ep_0\big\},\\
&\ & B_{\ep_0}^{l}\stackrel{def}{=}\big\{c=W_+(-1)+\ep e^{i\theta},\  0<\ep<\ep_0, \ \frac{\pi}{2}\leq\theta\leq\frac{3\pi}{2}\big\},\\
&\ & B_{\ep_0}^{r}\stackrel{def}{=}\big\{c=W_+(1)-\ep e^{i\theta}, \ 0<\ep<\ep_0, \ \frac{\pi}{2}\leq\theta\leq\frac{3\pi}{2}\big\},
\eeno
for some $\ep_0\in(0,1).$ We denote
$
\Omega_{\ep_0}\stackrel{def}{=}D_0\cup D_{\ep_0}\cup B_{\ep_0}^{l}\cup B_{\ep_0}^{r}.
$
We also define
\beqno
c_r=Re c\ for\  c\in D_0\cup D_{\ep_0}, \ c_r=W_+(-1)
\ for \ c\in B_{\ep_0}^{l}, \ c_r=W_+(1) \ for \ c\in B_{\ep_0}^{r}.
\eeqno
By Lemma \ref{lem: extend}, we can take a $C^5$ extension of $W_-$ to be $\wt{W}_-$ for $y\in[{\rm{a}}_-,1]$ such that $\wt{W}_-({\rm{a}}_-)=W_+(1)$ and $\wt{W}'_-(y)<0$.
\begin{itemize}
\item For $c\in D_0\cup D_{\ep_0}$ and $c_r\geq0$, we denote $y_{c_+}\in[0,1]$ with $y_{c_+}=(W_+)^{-1}(c_r)$, so that $W_+(y_{c_+})-c_r=0$ and $y_{c_-}\in[{\rm{a}}_-,0]$ with $y_{c_-}=(\wt{W}_-)^{-1}(c_r)$, so that $\wt{W}_-(y_{c_-})-c_r=0$.
\item For $c\in D_0\cup D_{\ep_0}$ and $c_r\leq0$, we denote $y_{c_+}\in[0,1]$ with $y_{c_+}=(W_-)^{-1}(c_r)$, so that $W_-(y_{c_+})-c_r=0$ and $y_{c_-}\in[-1,0]$ with $y_{c_-}=(W_+)^{-1}(c_r)$, so that $W_+(y_{c_-})-c_r=0$.
\item For $c\in B_{\ep_0}^l$, then $c_r=W_+(-1)=W_-(1)$, we denote $y_{c_+}=1$ and $y_{c_-}=-1$.
\item For $c\in B_{\ep_0}^r$, then $c_r=W_+(1)=\wt{W}_-({\rm{a}}_-)$, we denote $y_{c_+}=1$ and $y_{c_-}={\rm{a}}_-$.
\end{itemize}

We show the relationship between $y_{c_+}, y_{c_-}$ and $c_r$ by the following picture for the above case.

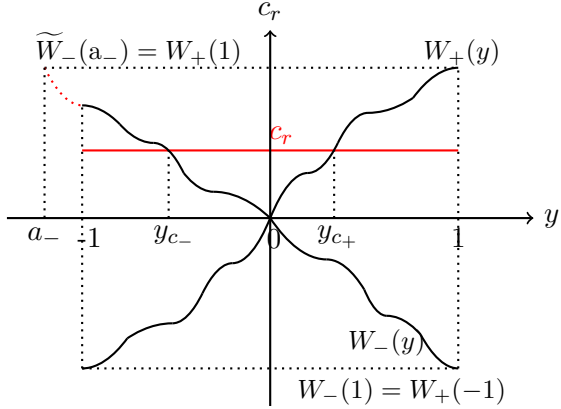
\begin{tikzpicture}[thick, scale=0.5]
    \draw[very thin,color=gray];
    \draw[->] (-7,0) -- (7,0) node[right] {$y$};
    \draw	(0.1,-0.01) node[anchor=north] {0}
		(5,0) node[anchor=north] {1}
        (-4.8,0) node[anchor=north] {-1};

    \draw[dotted] (-6,4) -- (5,4);
    \draw[dotted] (-5,-4) -- (5,-4);
    \draw[dotted] (-5,-4) -- (-5,3);
    \draw[dotted] (5,-4) -- (5,4);

    \draw[red,thick] (-5,1.8) -- (5,1.8);
    \draw[red] (0.3,2.2) node {$c_r$};
    \draw[->] (0,-5) -- (0,5) node[above] {$c_r$};
    \draw[thick] (0,0) parabola[bend at end] (1,1.2);
    \draw[thick] (1.8,2) parabola[bend at end] (1,1.2);
    \draw[thick] (1.8,2) parabola[bend at end] (2.6,2.8);
     \draw[thick] (3.8,3.2) parabola[bend at end] (2.6,2.8);
     \draw[thick] (3.8,3.2) parabola[bend at end] (5,4);

      \draw[thick] (0,0) parabola[bend at end] (-1,-1.2);
    \draw[thick] (-1.8,-2) parabola[bend at end] (-1,-1.2);
    \draw[thick] (-1.8,-2) parabola[bend at end] (-2.6,-2.8);
     \draw[thick] (-3.8,-3.2) parabola[bend at end] (-2.6,-2.8);
     \draw[thick] (-3.8,-3.2) parabola[bend at end] (-5,-4);

     \draw[thick] (0,0) parabola[bend at end] (-1.5,0.7);
    \draw[thick] (-2.3,1.2) parabola[bend at end] (-1.5,0.7);
    \draw[thick] (-2.3,1.2) parabola[bend at end] (-3.1,2.0);
     \draw[thick] (-4,2.5) parabola[bend at end] (-3.1,2.0);
     \draw[thick] (-4,2.5) parabola[bend at end] (-5,3);
    \draw[red,dotted] (-6,4) parabola[bend at end] (-5,3);

    \draw[thick] (0,0) parabola[bend at end] (1.5,-1.1);
    \draw[thick] (2.5,-1.9) parabola[bend at end] (1.5,-1.1);
    \draw[thick] (2.5,-1.9) parabola[bend at end] (3.1,-2.6);
     \draw[thick] (4,-3.1) parabola[bend at end] (3.1,-2.6);
     \draw[thick] (4,-3.1) parabola[bend at end] (5,-4);
\draw (5.1,4.4) node {\small $W_+(y)$};
\draw (3.1,-3.3) node {\small $W_-(y)$};
 \draw[dotted] (-2.7,0) -- (-2.7,1.8);
 \draw (-2.55,-0.5) node {$y_{c_-}$};
  \draw[dotted] (1.7,0) -- (1.7,1.8);
  \draw (1.8,-0.5) node {$y_{c_+}$};
  \draw[dotted] (-6,4)-- (-6,0); \draw (-6,-0.5) node {$a_-$};
  \draw (-3.5,4.5) node {\small $\wt{W}_-({\rm{a}}_-)=W_+(1)$};
\draw (3.5,-4.6) node {\small $W_-(1)=W_+(-1)$};
\end{tikzpicture}

We also take a $C^5$ extension of $W_+$ to be $\wt{W}_+$ for $y\in [{\rm{a}}_-,1]$, so that $\wt{W}_+'(y)>0$.

For {\bf{Case 5}}, we let
\beno
&\ & D_0\stackrel{def}{=}\big\{c\in[W_-(1), W_+(1)]\big\},\\
&\ & D_{\ep_0}\stackrel{def}{=}\big\{c=c_r+i\ep, \ c_r\in[W_-(1), W_+(1)], \ 0<|\ep|<\ep_0\big\},\\
&\ & B_{\ep_0}^{l}\stackrel{def}{=}\big\{c=W_-(1)+\ep e^{i\theta},\  0<\ep<\ep_0, \ \frac{\pi}{2}\leq\theta\leq\frac{3\pi}{2}\big\},\\
&\ & B_{\ep_0}^{r}\stackrel{def}{=}\big\{c=W_+(1)-\ep e^{i\theta}, \ 0<\ep<\ep_0, \ \frac{\pi}{2}\leq\theta\leq\frac{3\pi}{2}\big\},
\eeno
for some $\ep_0\in(0,1).$ We denote
$
\Omega_{\ep_0}\stackrel{def}{=}D_0\cup D_{\ep_0}\cup B_{\ep_0}^{l}\cup B_{\ep_0}^{r}.
$
We also define
\beqno
c_r=Re\, c\ for\  c\in D_0\cup D_{\ep_0}, \ c_r=W_-(1)
\ for \ c\in B_{\ep_0}^{l}, \ c_r=W_+(1) \ for \ c\in B_{\ep_0}^{r}.
\eeqno
By Lemma \ref{lem: extend}, we can take a $C^5$ extension of $W_-$ to be $\wt{W}_-$  and $W_+$ to be $\wt{W}_+$ for $y\in[{\rm{a}}_-,1]$ such that $\wt{W}_-({\rm{a}}_-)=W_+(1)$, $\wt{W}'_-(y)<0$ and  $\wt{W}_+({\rm{a}}_-)=W_-(1)$, $\wt{W}'_+(y)>0$.
\begin{itemize}
\item For $c\in D_0\cup D_{\ep_0}$ and $c_r\geq0$, we denote $y_{c_+}\in[0,1]$ with $y_{c_+}=(W_+)^{-1}(c_r)$, so that $W_+(y_{c_+})-c_r=0$ and $y_{c_-}\in[{\rm{a}}_-,0]$ with $y_{c_-}=(\wt{W}_-)^{-1}(c_r)$, so that $\wt{W}_-(y_{c_-})-c_r=0$.
\item For $c\in D_0\cup D_{\ep_0}$ and $c_r\leq0$, we denote $y_{c_+}\in[0,1]$ with $y_{c_+}=(W_-)^{-1}(c_r)$, so that $W_-(y_{c_+})-c_r=0$ and $y_{c_-}\in[{\rm{a}}_-,0]$ with $y_{c_-}=(\wt{W}_+)^{-1}(c_r)$, so that $\wt{W}_+(y_{c_-})-c_r=0$.
\item For $c\in B_{\ep_0}^l$, then $c_r=W_-(1)=\wt{W}_+({\rm{a}}_-)$, we denote $y_{c_+}=1$ and $y_{c_-}={\rm{a}}_-$.
\item For $c\in B_{\ep_0}^r$, then $c_r=W_+(1)=\wt{W}_-({\rm{a}}_-)$, we denote $y_{c_+}=1$ and $y_{c_-}={\rm{a}}_-$.
\end{itemize}

We show the relationship between $y_{c_+}, y_{c_-}$ and $c_r$ by the following picture for the above case.

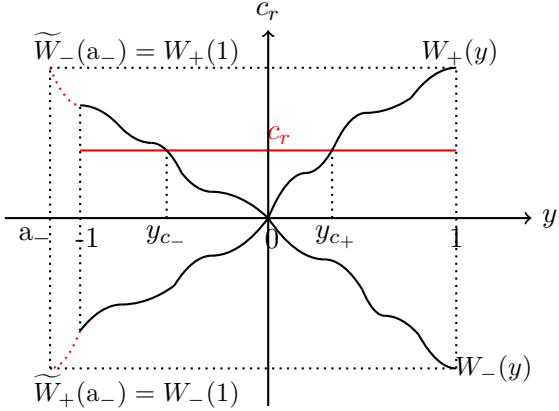
\begin{tikzpicture}[thick, scale=0.5]
    \draw[very thin,color=gray];
    \draw[->] (-7,0) -- (7,0) node[right] {$y$};
    \draw(0.1,-0.01) node[anchor=north] {0}
		(5,0) node[anchor=north] {1}
        (-4.8,0) node[anchor=north] {-1};

    \draw[dotted] (-5.8,4) -- (5,4);
    \draw[dotted] (-5.8,-4) -- (5,-4);
    \draw[dotted] (-5,-3) -- (-5,3);
    \draw[dotted] (5,-4) -- (5,4);

    \draw[red,thick] (-5,1.8) -- (5,1.8);
    \draw[red] (0.3,2.2) node {$c_r$};
    \draw[->] (0,-5) -- (0,5) node[above] {$c_r$};
    \draw[thick] (0,0) parabola[bend at end] (1,1.2);
    \draw[thick] (1.8,2) parabola[bend at end] (1,1.2);
    \draw[thick] (1.8,2) parabola[bend at end] (2.6,2.8);
     \draw[thick] (3.8,3.2) parabola[bend at end] (2.6,2.8);
     \draw[thick] (3.8,3.2) parabola[bend at end] (5,4);

      \draw[thick] (0,0) parabola[bend at end] (-1.5,-1.0);
    \draw[thick] (-2.5,-1.8) parabola[bend at end] (-1.5,-1.0);
    \draw[thick] (-2.5,-1.8) parabola[bend at end] (-3.9,-2.3);
     \draw[thick] (-5,-3) parabola[bend at end] (-3.9,-2.3);
      \draw[red,dotted] (-5,-3) parabola[bend at end] (-5.8,-4);

     \draw[thick] (0,0) parabola[bend at end] (-1.5,0.7);
    \draw[thick] (-2.3,1.2) parabola[bend at end] (-1.5,0.7);
    \draw[thick] (-2.3,1.2) parabola[bend at end] (-3.1,2.0);
     \draw[thick] (-4,2.5) parabola[bend at end] (-3.1,2.0);
     \draw[thick] (-4,2.5) parabola[bend at end] (-5,3);
    \draw[red,dotted] (-5.8,4) parabola[bend at end] (-5,3);

     \draw[thick] (0,0) parabola[bend at end] (1.5,-1.1);
    \draw[thick] (2.5,-1.9) parabola[bend at end] (1.5,-1.1);
    \draw[thick] (2.5,-1.9) parabola[bend at end] (3.1,-2.6);
     \draw[thick] (4,-3.1) parabola[bend at end] (3.1,-2.6);
     \draw[thick] (4,-3.1) parabola[bend at end] (5,-4);

\draw (5.1,4.4) node {\small $W_+(y)$};
\draw (6,-4) node {\small $W_-(y)$};
 \draw[dotted] (-2.7,0) -- (-2.7,1.8);
 \draw (-2.7,-0.5) node {$y_{c_-}$};
  \draw[dotted] (1.7,0) -- (1.7,1.9);
  \draw (1.8,-0.5) node {$y_{c_+}$};
  \draw[dotted] (-5.8,-4)-- (-5.8,4); \draw (-6.2,-0.5) node {${\rm{a}}_-$};
   \draw (-3.5,4.5) node {\small $\wt{W}_-({\rm{a}}_-)=W_+(1)$};
\draw (-3.5,-4.6) node {\small $\wt{W}_+({\rm{a}}_-)=W_-(1)$};
\end{tikzpicture}

For {\bf{Case 6}}, we let
\beno
&\ & D_0\stackrel{def}{=}\big\{c\in[W_-(1), W_+(1)]\big\},\\
&\ & D_{\ep_0}\stackrel{def}{=}\big\{c=c_r+i\ep, \ c_r\in[W_-(1), W_+(1)], \ 0<|\ep|<\ep_0\big\},\\
&\ & B_{\ep_0}^l\stackrel{def}{=}\big\{c=W_-(1)+\ep e^{i\theta},\  0<\ep<\ep_0, \ \frac{\pi}{2}\leq\theta\leq\frac{3\pi}{2}\big\},\\
&\ & B_{\ep_0}^r\stackrel{def}{=}\big\{c=W_+(1)-\ep e^{i\theta}, \ 0<\ep<\ep_0, \ \frac{\pi}{2}\leq\theta\leq\frac{3\pi}{2}\big\},
\eeno
for some $\ep_0\in(0,1).$ We denote
$
\Omega_{\ep_0}\stackrel{def}{=}D_0\cup D_{\ep_0}\cup B_{\ep_0}^l\cup B_{\ep_0}^r.
$
We also define
\beqno
c_r=Re c\ for\  c\in D_0\cup D_{\ep_0}, \ c_r=W_-(1)
\ for \ c\in B_{\ep_0}^l, \ c_r=W_+(1) \ for \ c\in B_{\ep_0}^r.
\eeqno
\begin{itemize}
\item For $c\in D_0\cup D_{\ep_0}$ and $c_r\geq0$, we denote $y_{c_+}\in[0,1]$ with $y_{c_+}=(W_+)^{-1}(c_r)$, so that $W_+(y_{c_+})-c_r=0$ and $y_{c_-}\in[-1,0]$ with $y_{c_-}=(W_-)^{-1}(c_r)$, so that $W_-(y_{c_-})-c_r=0$.
\item For $c\in D_0\cup D_{\ep_0}$ and $c_r\leq0$, we denote $y_{c_+}\in[0,1]$ with $y_{c_+}=(W_-)^{-1}(c_r)$, so that $W_-(y_{c_+})-c_r=0$ and $y_{c_-}\in[-1,0]$ with $y_{c_-}=(W_+)^{-1}(c_r)$, so that $W_+(y_{c_-})-c_r=0$.
\item For $c\in B_{\ep_0}^l$, then $c_r=W_-(1)=W_+(-1)$, we denote $y_{c_+}=1$ and $y_{c_-}=-1$.
\item For $c\in B_{\ep_0}^r$, then $c_r=W_+(1)=W_-(-1)$, we denote $y_{c_+}=1$ and $y_{c_-}=-1$.
\end{itemize}

 We show the relationship between $y_{c_+}$, $y_{c_-}$ and $c_r$ by the following picture.

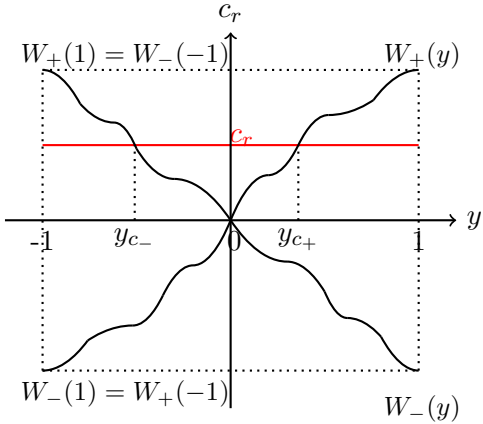
\begin{tikzpicture}[thick, scale=0.5]
    \draw[very thin,color=gray];
    \draw[->] (-6,0) -- (6,0) node[right] {$y$};
    \draw	(0.1,-0.01) node[anchor=north] {0}
		(5,0) node[anchor=north] {1}
        (-5,0) node[anchor=north] {-1};
    \draw[dotted] (-5,4) -- (5,4);
    \draw[red,thick] (-5,2) -- (5,2);
    \draw[red] (0.3,2.2) node {$c_r$};
    \draw[dotted] (-5,-4) -- (5,-4);
    \draw[dotted] (-5,-4) -- (-5,4);
    \draw[dotted] (5,-4) -- (5,4);
    \draw[->] (0,-5) -- (0,5) node[above] {$c_r$};
    \draw[thick] (0,0) parabola[bend at end] (1,1.2);
    \draw[thick] (1.8,2) parabola[bend at end] (1,1.2);
    \draw[thick] (1.8,2) parabola[bend at end] (2.6,2.8);
     \draw[thick] (3.8,3.2) parabola[bend at end] (2.6,2.8);
     \draw[thick] (3.8,3.2) parabola[bend at end] (5,4);
      \draw[thick] (0,0) parabola[bend at end] (-1,-1.2);
    \draw[thick] (-1.8,-2) parabola[bend at end] (-1,-1.2);
    \draw[thick] (-1.8,-2) parabola[bend at end] (-2.6,-2.8);
     \draw[thick] (-3.8,-3.2) parabola[bend at end] (-2.6,-2.8);
     \draw[thick] (-3.8,-3.2) parabola[bend at end] (-5,-4);

     \draw[thick] (0,0) parabola[bend at end] (-1.5,1.1);
    \draw[thick] (-2.5,1.9) parabola[bend at end] (-1.5,1.1);
    \draw[thick] (-2.5,1.9) parabola[bend at end] (-3.1,2.6);
     \draw[thick] (-4,3.1) parabola[bend at end] (-3.1,2.6);
     \draw[thick] (-4,3.1) parabola[bend at end] (-5,4);

     \draw[thick] (0,0) parabola[bend at end] (1.5,-1.1);
    \draw[thick] (2.5,-1.9) parabola[bend at end] (1.5,-1.1);
    \draw[thick] (2.5,-1.9) parabola[bend at end] (3.1,-2.6);
     \draw[thick] (4,-3.1) parabola[bend at end] (3.1,-2.6);
     \draw[thick] (4,-3.1) parabola[bend at end] (5,-4);

\draw (5.1,4.4) node {\small $W_+(y)$};
\draw (5.1,-5) node {\small $W_-(y)$};
\draw (-2.8,4.4) node {\small $W_+(1)=W_-(-1)$};
\draw (-2.8,-4.6) node {\small $W_-(1)=W_+(-1)$};
 \draw[dotted] (-2.55,0) -- (-2.55,2);
 \draw (-2.55,-0.5) node {$y_{c_-}$};
  \draw[dotted] (1.8,0) -- (1.8,2);
  \draw (1.8,-0.5) node {$y_{c_+}$};
\end{tikzpicture}

For {\bf{Case 7}}, we let
\beno
&\ & D_0\stackrel{def}{=}\big\{c\in[W_-(1), W_+(1)]\big\},\\
&\ & D_{\ep_0}\stackrel{def}{=}\big\{c=c_r+i\ep, \ c_r\in[W_-(1), W_+(1)], \ 0<|\ep|<\ep_0\big\},\\
&\ & B_{\ep_0}^{l}\stackrel{def}{=}\big\{c=W_-(1)+\ep e^{i\theta},\  0<\ep<\ep_0, \ \frac{\pi}{2}\leq\theta\leq\frac{3\pi}{2}\big\},\\
&\ & B_{\ep_0}^{r}\stackrel{def}{=}\big\{c=W_+(1)-\ep e^{i\theta}, \ 0<\ep<\ep_0, \ \frac{\pi}{2}\leq\theta\leq\frac{3\pi}{2}\big\},
\eeno
for some $\ep_0\in(0,1).$ We denote
$
\Omega_{\ep_0}\stackrel{def}{=}D_0\cup D_{\ep_0}\cup B_{\ep_0}^{l}\cup B_{\ep_0}^{r}.
$
We also define
\beqno
c_r=Re c\ for\  c\in D_0\cup D_{\ep_0}, \ c_r=W_-(1)
\ for \ c\in B_{\ep_0}^{l}, \ c_r=W_+(1) \ for \ c\in B_{\ep_0}^{r}.
\eeqno
By Lemma \ref{lem: extend}, we can take a $C^5$ extension of $W_+$ to be $\wt{W}_+$ for $y\in[{\rm{a}}_-,1]$ such that $\wt{W}_+({\rm{a}}_-)=W_-(1)$, $\wt{W}'_+(y)>0$.
\begin{itemize}
\item For $c\in D_0\cup D_{\ep_0}$ and $c_r\geq0$, we denote $y_{c_+}\in[0,1]$ with $y_{c_+}=(W_+)^{-1}(c_r)$, so that $W_+(y_{c_+})-c_r=0$ and $y_{c_-}\in[-1,0]$ with $y_{c_-}=(W_-)^{-1}(c_r)$, so that $W_-(y_{c_-})-c_r=0$.
\item For $c\in D_0\cup D_{\ep_0}$ and $c_r\leq0$, we denote $y_{c_+}\in[0,1]$ with $y_{c_+}=(W_-)^{-1}(c_r)$, so that $W_-(y_{c_+})-c_r=0$ and $y_{c_-}\in[{\rm{a}}_-,0]$ with $y_{c_-}=(\wt{W}_+)^{-1}(c_r)$, so that $\wt{W}_+(y_{c_-})-c_r=0$.
\item For $c\in B_{\ep_0}^l$, then $c_r=W_-(1)=\wt{W}_+({\rm{a}}_-)$, we denote $y_{c_+}=1$ and $y_{c_-}={\rm{a}}_-$.
\item For $c\in B_{\ep_0}^r$, then $c_r=W_+(1)=W_-(-1)$, we denote $y_{c_+}=1$ and $y_{c_-}=-1$.
\end{itemize}

We show the relationship between $y_{c_+}, y_{c_-}$ and $c_r$ by the following picture for the above case.

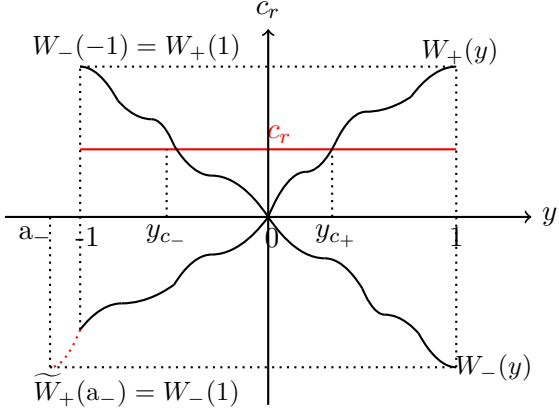
\begin{tikzpicture}[thick, scale=0.5]
    \draw[very thin,color=gray];
    \draw[->] (-7,0) -- (7,0) node[right] {$y$};
    \draw(0.1,-0.01) node[anchor=north] {0}
		(5,0) node[anchor=north] {1}
        (-4.8,0) node[anchor=north] {-1};

    \draw[dotted] (-5,4) -- (5,4);
    \draw[dotted] (-5.8,-4) -- (5,-4);
    \draw[dotted] (-5,-3) -- (-5,4);
    \draw[dotted] (5,-4) -- (5,4);

    \draw[red,thick] (-5,1.8) -- (5,1.8);
    \draw[red] (0.3,2.2) node {$c_r$};
    \draw[->] (0,-5) -- (0,5) node[above] {$c_r$};
    \draw[thick] (0,0) parabola[bend at end] (1,1.2);
    \draw[thick] (1.8,2) parabola[bend at end] (1,1.2);
    \draw[thick] (1.8,2) parabola[bend at end] (2.6,2.8);
     \draw[thick] (3.8,3.2) parabola[bend at end] (2.6,2.8);
     \draw[thick] (3.8,3.2) parabola[bend at end] (5,4);

      \draw[thick] (0,0) parabola[bend at end] (-1.5,-1.0);
    \draw[thick] (-2.5,-1.8) parabola[bend at end] (-1.5,-1.0);
    \draw[thick] (-2.5,-1.8) parabola[bend at end] (-3.9,-2.3);
     \draw[thick] (-5,-3) parabola[bend at end] (-3.9,-2.3);
      \draw[red,dotted] (-5,-3) parabola[bend at end] (-5.8,-4);

    \draw[thick] (0,0) parabola[bend at end] (-1.5,1.1);
    \draw[thick] (-2.5,1.9) parabola[bend at end] (-1.5,1.1);
    \draw[thick] (-2.5,1.9) parabola[bend at end] (-3.1,2.6);
     \draw[thick] (-4,3.1) parabola[bend at end] (-3.1,2.6);
     \draw[thick] (-4,3.1) parabola[bend at end] (-5,4);

     \draw[thick] (0,0) parabola[bend at end] (1.5,-1.1);
    \draw[thick] (2.5,-1.9) parabola[bend at end] (1.5,-1.1);
    \draw[thick] (2.5,-1.9) parabola[bend at end] (3.1,-2.6);
     \draw[thick] (4,-3.1) parabola[bend at end] (3.1,-2.6);
     \draw[thick] (4,-3.1) parabola[bend at end] (5,-4);

\draw (5.1,4.4) node {\small $W_+(y)$};
\draw (6,-4) node {\small $W_-(y)$};
 \draw[dotted] (-2.7,0) -- (-2.7,1.8);
 \draw (-2.7,-0.5) node {$y_{c_-}$};
  \draw[dotted] (1.7,0) -- (1.7,1.9);
  \draw (1.8,-0.5) node {$y_{c_+}$};
  \draw[dotted] (-5.8,-4)-- (-5.8,0); \draw (-6.2,-0.5) node {${\rm{a}}_-$};
   \draw (-3.5,4.5) node {\small $W_-(-1)=W_+(1)$};
\draw (-3.5,-4.6) node {\small $\wt{W}_+({\rm{a}}_-)=W_-(1)$};
\end{tikzpicture}

We also take a $C^5$ extension of $W_-$ to be $\wt{W}_-$ for $y\in [{\rm{a}}_-,1]$, so that $\wt{W}_-'(y)<0$.

For {\bf{Case 8}}, we let
\beno
&\ & D_0\stackrel{def}{=}\big\{c\in[W_+(-1), W_-(-1)]\big\},\\
&\ & D_{\ep_0}\stackrel{def}{=}\big\{c=c_r+i\ep, \ c_r\in[W_+(-1), W_-(-1)], \ 0<|\ep|<\ep_0\big\},\\
&\ & B_{\ep_0}^{l}\stackrel{def}{=}\big\{c=W_+(-1)+\ep e^{i\theta},\  0<\ep<\ep_0, \ \frac{\pi}{2}\leq\theta\leq\frac{3\pi}{2}\big\},\\
&\ & B_{\ep_0}^{r}\stackrel{def}{=}\big\{c=W_-(-1)-\ep e^{i\theta}, \ 0<\ep<\ep_0, \ \frac{\pi}{2}\leq\theta\leq\frac{3\pi}{2}\big\},
\eeno
for some $\ep_0\in(0,1).$ We denote
$
\Omega_{\ep_0}\stackrel{def}{=}D_0\cup D_{\ep_0}\cup B_{\ep_0}^{l}\cup B_{\ep_0}^{r}.
$
We also define
\beqno
c_r=Re c\ for\  c\in D_0\cup D_{\ep_0}, \ c_r=W_+(-1)
\ for \ c\in B_{\ep_0}^{l}, \ c_r=W_-(-1) \ for \ c\in B_{\ep_0}^{r}.
\eeqno
By Lemma \ref{lem: extend}, we can take a $C^5$ extension of  $W_+$ to be $\wt{W}_+$ for $y\in[-1,{\rm{a}}_+]$ such that $\wt{W}_+({\rm{a}}_+)=W_-(-1)$, $\wt{W}'_+(y)>0$.
\begin{itemize}
\item For $c\in D_0\cup D_{\ep_0}$ and $c_r\geq0$, we denote $y_{c_+}\in[0,{\rm{a}}_+]$ with $y_{c_+}=(\wt{W}_+)^{-1}(c_r)$, so that $\wt{W}_+(y_{c_+})-c_r=0$ and $y_{c_-}\in[-1,0]$ with $y_{c_-}=(W_-)^{-1}(c_r)$, so that $W_-(y_{c_-})-c_r=0$.
\item For $c\in D_0\cup D_{\ep_0}$ and $c_r\leq0$, we denote $y_{c_+}\in[0,1]$ with $y_{c_+}=(W_-)^{-1}(c_r)$, so that $W_-(y_{c_+})-c_r=0$ and $y_{c_-}\in[-1,0]$ with $y_{c_-}=(W_+)^{-1}(c_r)$, so that $W_+(y_{c_-})-c_r=0$.
\item For $c\in B_{\ep_0}^l$, then $c_r=W_+(-1)=W_-(-1)$, we denote $y_{c_+}=1$ and $y_{c_-}=-1$.
\item For $c\in B_{\ep_0}^r$, then $c_r=W_-(-1)=\wt{W}_+({\rm{a}}_+)$, we denote $y_{c_+}={\rm{a}}_+$ and $y_{c_-}=-1$.
\end{itemize}

We show the relationship between $y_{c_+}, y_{c_-}$ and $c_r$ by the following picture for the above case.

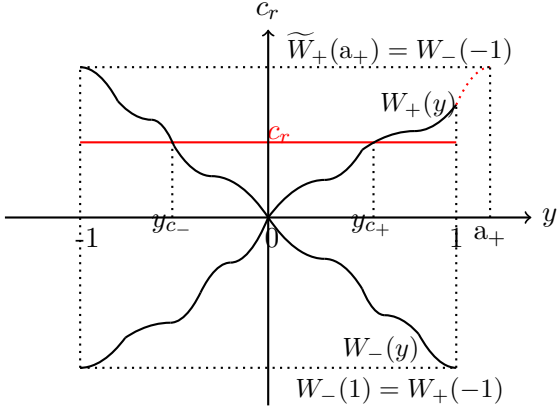
\begin{tikzpicture}[thick, scale=0.5]
    \draw[very thin,color=gray];
    \draw[->] (-7,0) -- (7,0) node[right] {$y$};
    \draw	(0.1,-0.01) node[anchor=north] {0}
		(5,0) node[anchor=north] {1}
        (-4.8,0) node[anchor=north] {-1};
    \draw[dotted] (-5,4) -- (5.9,4);
    \draw[dotted] (-5,-4) -- (5,-4);
    \draw[dotted] (-5,-4) -- (-5,4);
    \draw[dotted] (5,-4) -- (5,3);

    \draw[red,thick] (-5,2) -- (5,2);
    \draw[red] (0.3,2.2) node {$c_r$};
    \draw[->] (0,-5) -- (0,5) node[above] {$c_r$};

    \draw[thick] (0,0) parabola[bend at end] (1.5,1.0);
    \draw[thick] (2.5,1.8) parabola[bend at end] (1.5,1.0);
    \draw[thick] (2.5,1.8) parabola[bend at end] (3.9,2.3);
     \draw[thick] (5,3) parabola[bend at end] (3.9,2.3);
      \draw[red,dotted] (5,3) parabola[bend at end] (5.9,4);

      \draw[thick] (0,0) parabola[bend at end] (-1,-1.2);
    \draw[thick] (-1.8,-2) parabola[bend at end] (-1,-1.2);
    \draw[thick] (-1.8,-2) parabola[bend at end] (-2.6,-2.8);
     \draw[thick] (-3.8,-3.2) parabola[bend at end] (-2.6,-2.8);
     \draw[thick] (-3.8,-3.2) parabola[bend at end] (-5,-4);

     \draw[thick] (0,0) parabola[bend at end] (-1.5,1.1);
    \draw[thick] (-2.5,1.9) parabola[bend at end] (-1.5,1.1);
    \draw[thick] (-2.5,1.9) parabola[bend at end] (-3.1,2.6);
     \draw[thick] (-4,3.1) parabola[bend at end] (-3.1,2.6);
     \draw[thick] (-4,3.1) parabola[bend at end] (-5,4);

    \draw[thick] (0,0) parabola[bend at end] (1.5,-1.1);
    \draw[thick] (2.5,-1.9) parabola[bend at end] (1.5,-1.1);
    \draw[thick] (2.5,-1.9) parabola[bend at end] (3.1,-2.6);
     \draw[thick] (4,-3.1) parabola[bend at end] (3.1,-2.6);
     \draw[thick] (4,-3.1) parabola[bend at end] (5,-4);

\draw (4.0,3.0) node {\small $W_+(y)$};
\draw (3.0,-3.5) node {\small $W_-(y)$};
 \draw[dotted] (-2.55,0) -- (-2.55,2);
 \draw (-2.55,-0.2) node {$y_{c_-}$};
  \draw[dotted] (2.8,0) -- (2.8,2);
  \draw (2.75,-0.2) node {$y_{c_+}$};
  \draw[dotted] (5.9,0)-- (5.9,4); \draw (5.9,-0.5) node {${\rm{a}}_+$};
  \draw (3.5,4.5) node {\small $\wt{W}_+({\rm{a}}_+)=W_-(-1)$};
\draw (3.5,-4.6) node {\small $W_-(1)=W_+(-1)$};
\end{tikzpicture}

We also take a $C^5$ extension of $W_-$ to be $\wt{W}_-$ for $y\in [-1, {\rm{a}}_+]$, so that $\wt{W}_-'(y)<0$.

For {\bf{Case 9}}, we let
\beno
&\ & D_0\stackrel{def}{=}\big\{c\in[W_-(1), W_-(-1)]\big\},\\
&\ & D_{\ep_0}\stackrel{def}{=}\big\{c=c_r+i\ep, \ c_r\in[W_-(1), W_-(-1)], \ 0<|\ep|<\ep_0\big\},\\
&\ & B_{\ep_0}^{l}\stackrel{def}{=}\big\{c=W_-(1)+\ep e^{i\theta},\  0<\ep<\ep_0, \ \frac{\pi}{2}\leq\theta\leq\frac{3\pi}{2}\big\},\\
&\ & B_{\ep_0}^{r}\stackrel{def}{=}\big\{c=W_-(-1)-\ep e^{i\theta}, \ 0<\ep<\ep_0, \ \frac{\pi}{2}\leq\theta\leq\frac{3\pi}{2}\big\},
\eeno
for some $\ep_0\in(0,1).$ We denote
$
\Omega_{\ep_0}\stackrel{def}{=}D_0\cup D_{\ep_0}\cup B_{\ep_0}^{l}\cup B_{\ep_0}^{r}.
$
We also define
\beqno
c_r=Re c\ for\  c\in D_0\cup D_{\ep_0}, \ c_r=W_-(1)
\ for \ c\in B_{\ep_0}^{l}, \ c_r=W_-(-1) \ for \ c\in B_{\ep_0}^{r}.
\eeqno
By Lemma \ref{lem: extend}, we can take a $C^5$ extension of $W_+$ to be $\wt{W}_+$ for $y\in[{\rm{a}}_-,{\rm{a}}_+]$ such that $\wt{W}_+({\rm{a}}_-)=W_-(1)$, $\wt{W}_+({\rm{a}}_+)=W_-(-1)$ and $\wt{W}'_+(y)>0$.
\begin{itemize}
\item For $c\in D_0\cup D_{\ep_0}$ and $c_r\geq0$, we denote $y_{c_+}\in[0,{\rm{a}}_+]$ with $y_{c_+}=(\wt{W}_+)^{-1}(c_r)$, so that $\wt{W}_+(y_{c_+})-c_r=0$ and $y_{c_-}\in[-1,0]$ with $y_{c_-}=(W_-)^{-1}(c_r)$, so that $W_-(y_{c_-})-c_r=0$.
\item For $c\in D_0\cup D_{\ep_0}$ and $c_r\leq0$, we denote $y_{c_+}\in[0,1]$ with $y_{c_+}=(W_-)^{-1}(c_r)$, so that $W_-(y_{c_+})-c_r=0$ and $y_{c_-}\in[{\rm{a}}_-,0]$ with $y_{c_-}=(\wt{W}_+)^{-1}(c_r)$, so that $\wt{W}_+(y_{c_-})-c_r=0$.
\item For $c\in B_{\ep_0}^l$, then $c_r=W_-(1)=\wt{W}_+(a_-)$, we denote $y_{c_+}=1$ and $y_{c_-}={\rm{a}}_-$.
\item For $c\in B_{\ep_0}^r$, then $c_r=W_-(-1)=\wt{W}_+({\rm{a}}_+)$, we denote $y_{c_+}={\rm{a}}_+$ and $y_{c_-}=-1$.
\end{itemize}

We show the relationship between $y_{c_+}, y_{c_-}$ and $c_r$ by the following picture for the above case.

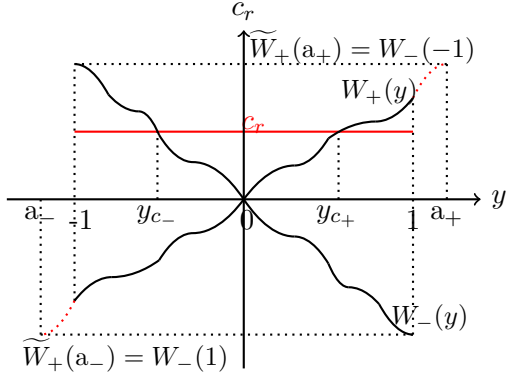
\begin{tikzpicture}[thick, scale=0.45]
    \draw[very thin,color=gray];
    \draw[->] (-7,0) -- (7,0) node[right] {$y$};
    \draw	(0.1,-0.01) node[anchor=north] {0}
		(5,0) node[anchor=north] {1}
        (-4.8,0) node[anchor=north] {-1};
    \draw[dotted] (-5,4) -- (6,4);
    \draw[dotted] (-6,-4) -- (5,-4);
    \draw[dotted] (-5,-3) -- (-5,4);
    \draw[dotted] (5,-4) -- (5,3);

    \draw[red,thick] (-5,2) -- (5,2);
    \draw[red] (0.3,2.2) node {$c_r$};
    \draw[->] (0,-5) -- (0,5) node[above] {$c_r$};
    \draw[thick] (0,0) parabola[bend at end] (1.5,1.0);
    \draw[thick] (2.5,1.8) parabola[bend at end] (1.5,1.0);
    \draw[thick] (2.5,1.8) parabola[bend at end] (3.9,2.3);
     \draw[thick] (5,3) parabola[bend at end] (3.9,2.3);
      \draw[red,dotted] (5,3) parabola[bend at end] (6,4);

     \draw[thick] (0,0) parabola[bend at end] (-1.5,-1.0);
    \draw[thick] (-2.5,-1.8) parabola[bend at end] (-1.5,-1.0);
    \draw[thick] (-2.5,-1.8) parabola[bend at end] (-3.9,-2.3);
     \draw[thick] (-5,-3) parabola[bend at end] (-3.9,-2.3);
      \draw[red,dotted] (-5,-3) parabola[bend at end] (-6,-4);

     \draw[thick] (0,0) parabola[bend at end] (-1.5,1.1);
    \draw[thick] (-2.5,1.9) parabola[bend at end] (-1.5,1.1);
    \draw[thick] (-2.5,1.9) parabola[bend at end] (-3.1,2.6);
     \draw[thick] (-4,3.1) parabola[bend at end] (-3.1,2.6);
     \draw[thick] (-4,3.1) parabola[bend at end] (-5,4);

     \draw[thick] (0,0) parabola[bend at end] (1.5,-1.1);
    \draw[thick] (2.5,-1.9) parabola[bend at end] (1.5,-1.1);
    \draw[thick] (2.5,-1.9) parabola[bend at end] (3.1,-2.6);
     \draw[thick] (4,-3.1) parabola[bend at end] (3.1,-2.6);
     \draw[thick] (4,-3.1) parabola[bend at end] (5,-4);

\draw (4.0,3.2) node {\small $W_+(y)$};
\draw (5.5,-3.5) node {\small $W_-(y)$};
 \draw[dotted] (-2.55,0) -- (-2.55,2);
 \draw (-2.55,-0.5) node {$y_{c_-}$};
  \draw[dotted] (2.8,0) -- (2.8,2);
  \draw (2.75,-0.5) node {$y_{c_+}$};
   \draw[dotted] (6,4)-- (6,0); \draw (6,-0.5) node {${\rm{a}}_+$};
  \draw[dotted] (-6,-4)-- (-6,0); \draw (-6,-0.5) node {${\rm{a}}_-$};
  \draw (3.5,4.5) node {\small $\wt{W}_+({\rm{a}}_+)=W_-(-1)$};
\draw (-3.5,-4.6) node {\small $\wt{W}_+({\rm{a}}_-)=W_-(1)$};
\end{tikzpicture}

We also take a $C^5$ extension of $W_-$ to be $\wt{W}_-$ for $y\in [{\rm{a}_-}, {\rm{a}}_+]$, so that $\wt{W}_-'(y)<0$.

In the last step of Case 1, 2, 3, 7, 8, 9, we only restrict the regularity and monotonicity of the extension. 

\no{\bf Notations:}
We summarize the above nine cases and make the following notations. 

We denote $a_+={\rm{a}}_+$ if we need to extend the definition of $W_{+}(y)$ or $W_{-}(y)$ for $y\geq 1$ and $a_+=1$ if we do not need to extend the definition of $W_{+}(y)$ nor $W_{-}(y)$ for $y\geq 1$. Similarly we also denote $a_-={\rm{a}}_-$ if we need to extend the definition of $W_{+}(y)$ or $W_{-}(y)$ for $y\leq -1$ and $a_-=-1$ if we do not need to extend the definition of $W_{+}(y)$ nor $W_{-}(y)$ for $y\leq -1$.  

By letting $u(y)=\f{\wt{W}_+(y)+\wt{W}_-(y)}{2}$ and $b(y)=\f{\wt{W}_+(y)-\wt{W}_-(y)}{2}$, we extend $u$ and $b$. We also take a $C^3$ extension of $\widehat{\phi}_0(\al,y)$ for $y\in (1,a_+]\cup [a_-,-1)$ and extend $\widehat{\om}_0(\al,y)=0$ for $y\in (1,a_+] \cup [a_-,-1)$ then $cG(\al, y,c)\in L^{\infty}([a_-,a_+]\times \Om_{\ep_0})$ and $\|cG\|_{L^{\infty}}\leq C(\|\widehat{\om}_0\|_{H^1_y}+C\|\widehat{j}_0\|_{H^3_y})$. 

We also use $W_{\pm}$ to represent $\wt{W}_{\pm}$, whether they are extended or not. For $\ep_0>0$, we let 
\beno
&&D_0=[\min\{W_-(1),W_+(-1)\}, \max\{W_+(1),W_-(-1)\}], \\
&&D_{\ep_0}=\{z=c+i\ep:~c\in D_0, 0<|\ep|<\ep_0\},\\
&&B_{\ep_0}^l=\{z=\min\{W_-(1),W_+(-1)\}+\ep e^{i\th},0<\ep<\ep_0, \f{\pi}{2}\leq \th\leq \f{3\pi}{2}\},\\
&&B_{\ep_0}^r=\{z=\max\{W_+(1),W_-(-1)\}-\ep e^{i\th}, 0<\ep<\ep_0, \f{\pi}{2}\leq \th\leq \f{3\pi}{2}\},
\eeno 
and $\Omega_{\ep_0}\stackrel{def}{=}D_0\cup D_{\ep_0}\cup B_{\ep_0}^{l}\cup B_{\ep_0}^{r}$. 
We let $\mathcal{H}(y,c)=({W}_+(y)-c)({W}_-(y)-c)$ be well defined on $[a_-, a_+]\times \Om_{\ep_0}$. Let $d_+=[0,a_+]$ and $d_-=[a_-,0]$. 

\subsection{Sturmian integral operator.}
Given $|\al|\geq 1$, let $A$ be a constant larger
than $C|\al|$ with $C\geq1$ independent of $\al$.
\begin{definition}\label{def: X+ norm}
For a function $f(y,c)$ defined on $d_{+}\times \Omega_{\ep_0}$ or $d_{-}\times \Omega_{\ep_0}$, we define
\beqno
\begin{split}
&\|f\|_{X_0^{\pm}}\stackrel{def}{=}\sup_{(y,c)\in d_{\pm}
\times D_0}\Big|\frac{f(y,c)}{\cosh(A(y-y_{c_{\pm}}))}\Big|,\\
&\|f\|_{X^{\pm}}\stackrel{def}{=}\sup_{(y,c)\in d_{\pm}
\times D_{\ep_0}\cup D_0}\Big|\frac{f(y,c)}{\cosh(A(y-y_{c_+}))}\Big|,\\
&\|f\|_{X_l^{\pm}}\stackrel{def}{=}\sup_{(y,c)\in d_{\pm}
\times B_{\ep_0}^{l}}\Big|\frac{f(y,c)}{\cosh(A(y-y_{c_{\pm}}))}\Big|,\\
&\|f\|_{X_r^{\pm}}\stackrel{def}{=}\sup_{(y,c)\in d_{\pm}\times B_{\ep_0}^{r}}\Big|\frac{f(y,c)}{\cosh(A(y-y_{c_{\pm}}))}\Big|,
\end{split}
\eeqno
where $c_r$ and $y_{c_{\pm}}$ were defined in the previous section satisfies $\mathcal{H}(y,c_r)=0$.
\end{definition}

\begin{definition}
For a function $f(y,c)$ defined on $d_+\times \Omega_{\ep_0}$ or $d_-\times \Omega_{\ep_0}$, we define
\beqno
\begin{split}
&\|f\|_{Y^{\pm}}\stackrel{def}{=}\|f\|_{X^{\pm}}
+\frac{1}{A}\big(\|\pa_yf\|_{X^{\pm}}+\|\pa_{c_r}f\|_{X^{\pm}}+\|\pa_{\ep}f\|_{X^{\pm}}\big),\\
&\|f\|_{Y_l^{\pm}}\stackrel{def}{=}\|f\|_{X_l^{\pm}}
+\frac{1}{A}\big(\|\pa_yf\|_{X_l^{\pm}}+\|\pa_{\ep}f\|_{X_l^{\pm}}+\|\pa_{\theta}f\|_{X_l^{\pm}}\big),\\
&\|f\|_{Y_r^{\pm}}\stackrel{def}{=}\|f\|_{X_r^{\pm}}
+\frac{1}{A}\big(\|\pa_yf\|_{X_r^{\pm}}+\|\pa_{\ep}f\|_{X_r^{\pm}}+\|\pa_{\theta}f\|_{X_r^{\pm}}\big),\\
&\|f\|_{Y_0^{\pm}}\stackrel{def}{=}\|f\|_{X_0^{\pm}}+\frac{1}{A}\big(\|\pa_yf\|_{X_0^{\pm}}
+\|\pa_cf\|_{X_0^{\pm}}\big)+\frac{1}{A^2}\|\pa_y\pa_cf\|_{X_0^{\pm}}.
\end{split}
\eeqno
\end{definition}

Now, we introduce the Sturmian integral operator, which will be used to give the solution formula of the homogeneous  Sturmian equation.
\begin{definition}
Let $y\in d_+$ or $y\in d_-$, the Sturmian integral operator $S^{\pm}$ is defined by
\beqno
S^{\pm}f(y,c)\stackrel{def}{=}S_0^{\pm}\circ S_1^{\pm}f(y,c)=\int_{y_{c_{\pm}}}^y
\frac{\int_{y_{c_{\pm}}}^{y'}\big(W_+(z)-c\big)\big(W_-(z)-c\big)f(z,c)dz}
{\big(W_+(y')-c\big)\big(W_-(y')-c\big)}dy',
\eeqno
where
\beqno
\begin{split}
&S_0^{\pm}f(y,c)\stackrel{def}{=}\int_{y_{c_{\pm}}}^yf(y',c)dy',\\
&S_1^{\pm}f(y,c)\stackrel{def}{=}
\frac{\int_{y_{c_{\pm}}}^{y}\big(W_+(z)-c\big)\big(W_-(z)-c\big)f(z,c)dz}
{\big(W_+(y)-c\big)\big(W_-(y)-c\big)}.
\end{split}
\eeqno
\end{definition}


\begin{proposition}\label{prop: S+ estimate}
For $y\in d_+$, there exists a constant $C_1$ independent of $A$ so that
\beqno
\begin{split}
&\|S^+f\|_{Y_0^+}\leq \frac{C_1}{A^2}\|f\|_{Y_0^+}, \
\|S^+f\|_{Y^+}\leq \frac{C_1}{A^2}\|f\|_{Y^+},\\
&\|S^+f\|_{Y_l^+}\leq \frac{C_1}{A^2}\|f\|_{Y_l^+},\
\|S^+f\|_{Y_r^+}\leq \frac{C_1}{A^2}\|f\|_{Y_r^+}.
\end{split}
\eeqno
Moreover, if $f\in C\big(d_+\times\Om_{\ep_0}\big)$, then
\beqno
S_0^+f,\  S_1^+f,\  S^+f\in C\big(d_+\times \Om_{\ep_0}\big).
\eeqno
\end{proposition}

\begin{proposition}\label{prop: S- estimate}
For $y\in d_-$, there exists a constant $C_1$ independent of $A$ so that
\beqno
\begin{split}
&\|S^-f\|_{Y_0^-}\leq \frac{C_1}{A^2}\|f\|_{Y_0^-},\
\|S^-f\|_{Y^-}\leq \frac{C_1}{A^2}\|f\|_{Y^-},\\
&
\|S^-f\|_{Y_l^-}\leq \frac{C_1}{A^2}\|f\|_{Y_l^-},\
 \|S^-f\|_{Y_r^-}\leq \frac{C_1}{A^2}\|f\|_{Y_r^-}.
 \end{split}
 \eeqno
Moreover, if $f\in C\big(d_-\times\Om_{\ep_0}\big)$, then
\beqno
S_0^-f, \ S_1^-f, \ S^-f\in C\big(d_-\times \Om_{\ep_0}\big).
\eeqno
\end{proposition}

In the following, we only give the proof of the Proposition \ref{prop: S+ estimate} and Proposition \ref{prop: S- estimate} can be obtained by the same method.
\begin{proof}
By the fact that $W'_+(y)>0$, $W'_-(y)<0$ for $0\leq y\leq a_+$ and $W_+(0)=W_-(0)=0$, we have for $c_r\geq 0$, then
$W_+(y_{c_+})=c_r$, we can conclude that for $c\in \Omega_{\ep_0}$,
\beqno
\textrm{if} \ 0<y_{c_+}<y'<y<a_+ \ \textrm{or}\  0<y<y'<y_{c_+}<a_+,  \ \textrm{then} \
\Big|\frac{W_+(y')-c}{W_+(y)-c}\Big|\leq 1,
\eeqno
and
\begin{align*}
\textrm{if} \ 0<y_{c_+}<y'<y<a_+, \  \textrm{then}\
\Big|\frac{W_-(y')-c}{W_-(y)-c}\Big|\leq 1,\\
\textrm{if} \  0<y<y'<y_{c_+}<a_+, \  \textrm{then}\
\Big|\frac{W_-(y')-c}{W_-(y)-c}\Big|
\leq C.
\end{align*}
And for the case $c_r\leq 0$, then $W_-(y_{c_+})=c_r$, we can conclude that for $c\in \Omega_{\ep_0}$,
\beqno
\textrm{if} \ 0<y_{c_+}<y'<y<a_+ \ or\  0<y<y'<y_{c_+}<a_+,  \ \textrm{then} \
\Big|\frac{W_-(y')-c}{W_-(y)-c}\Big|\leq 1,
\eeqno
and
\begin{align*}
\textrm{if} \ 0<y_{c_+}<y'<y<a_+, \ \textrm{then} \
\Big|\frac{W_+(y')-c}{W_+(y)-c}\Big|\leq 1,\\
\textrm{if} \  0<y<y'<y_{c_+}<a_+,  \ \textrm{then}\
\Big|\frac{W_+(y')-c}{W_+(y)-c}\Big|
\leq C.
\end{align*}
Thus we have that for $(y,c)\in[0,a_+]\times \Om_{\ep_0}$ and $ 0<y_{c_+}<y'<y<a_+$ or $0<y<y'<y_{c_+}<a_+$,
\beq\label{eq: W mono}
\Big|\frac{W_+(y')-c}{W_+(y)-c}\Big|\leq C,\quad
\Big|\frac{W_-(y')-c}{W_-(y)-c}\Big|\leq C.
\eeq
A direct calculation shows that for $(y,c)\in[0,a_+]\times D_0$
\beq\label{eq: S_0f}
\begin{split}
\|S_0^+f\|_{X_0^+}
&=\sup_{(y,c)\in[0,a_+]\times D_0}
\Big|\frac{1}{\cosh A(y-y_{c_+})}\int_{y_{c_+}}^y
\frac{f(z,c)}{\cosh A(y-y_{c_+})}\cosh A(z-y_{c_+})dz\Big|\\
&\leq \sup_{(y,c)\in[0,a_+]\times D_0}\Big|\frac{1}{\cosh A(y-y_{c_+})}\int_{y_{c_+}}^y\cosh A(z-y_{c_+})dz\Big|\|f\|_{X_0^+}\leq \frac{1}{A}\|f\|_{X_0^+}.
\end{split}
\eeq
Then we deduce
\beq\label{eq: S_11f}
\begin{split}
\Big\|S_1^+f(y,c)\Big\|_{X_0^+}
&\leq C\sup_{(y,c)\in[0,a_+]\times D_0}
\Big|\frac{y-y_{c_+}}{\cosh A(y-y_{c_+})}\int_0^1\cosh tA(y-y_{c_+})dt
\Big|\|f\|_{X_0^+}\\
&\leq \frac{C}{A}\|f\|_{X_0^+},
\end{split}
\eeq
which along with (\ref{eq: S_0f}) shows that
\beq\label{eq: Sf}
\|S^+f\|_{X_0^+}\leq \frac{C}{A^2}\|f\|_{X_0^+}.
\eeq
By using (\ref{eq: W mono}), we obtain
\ben\label{0}
&\ & \left\|\frac{\int_{y_{c_+}}^{y}(W_-(y')-c)f(y',c)dy'}
{(W_+(y)-c)(W_-(y)-c)}\right\|_{X_0^+}
+\left\|\frac{\int_{y_{c_+}}^{y}(W_+(y')-c)(W_-(y')-c)f(y',c)dy'}
{(W_+(y)-c)^2(W_-(y)-c)}\right\|_{X_0^+}\nonumber\\
&\ &\leq C\sup_{(y,c)\in[0,1]\times D_0}\Big|\frac{\int_0^1\cosh tA(y-y_{c_+})dt}
{\cosh A(y-y_{c_+})\int_0^1W_+^{'}(y_{c_+}+t(y-y_{c_+}))dt}
\Big|\|f\|_{X_0^+}\nonumber\\
&\ & \leq C\|f\|_{X_0^+},
\een
and
\ben\label{00}
&\ & \left\|\frac{\int_{y_{c_+}}^{y}(W_+(y)-c)f(y',c)dy'}
{(W_+(y)-c)(W_-(y)-c)}\right\|_{X_0^+}
+\left\|\frac{\int_{y_{c_+}}^{y}(W_+(y')-c)(W_-(y')-c)f(y',c)dy'}
{(W_+(y)-c)(W_-(y)-c)^2}\right\|_{X_0^+}\nonumber\\
&\ & \leq C\sup_{(y,c)\in[0,a_+]\times D_0}\Big|\frac{1}{\cosh A(y-y_{c_+})}
\int_0^1\cosh tA(y-y_{c_+})dt
\Big|\|f\|_{X_0^+}
\nonumber\\
&\ & \leq C\|f\|_{X_0^+}.
\een
Similarly, we also obtain that
\ben\label{8}
\|S_0^+f\|_{Z^+}\leq \frac{C}{A}\|f\|_{Z^+}, \ \ \|S_1^+f\|_{Z^+}\leq \frac{C}{A}\|f\|_{Z^+},
\een
\beq\label{88}
\begin{split}
&\left\|\frac{\int_{y_{c_+}}^{y}(W_-(y')-c)f(y',c)dy'}
{(W_+(y)-c)(W_-(y)-c)}\right\|_{Z^+}
+\left\|\frac{\int_{y_{c_+}}^{y}(W_+(y')-c)(W_-(y')-c)f(y',c)dy'}
{(W_+(y)-c)^2(W_-(y)-c)}\right\|_{Z^+}\\
&\leq C\|f\|_{Z^+}
\end{split}
\eeq
and
\beq\label{888}
\begin{split}
&\left\|\frac{\int_{y_{c_+}}^{y}(W_+(y)-c)f(y',c)dy'}
{(W_+(y)-c)(W_-(y)-c)}\right\|_{Z^+}
+\left\|\frac{\int_{y_{c_+}}^{y}(W_+(y')-c)(W_-(y')-c)f(y',c)dy'}
{(W_+(y)-c)(W_-(y)-c)^2}\right\|_{Z^+}\\
&\leq C\|f\|_{Z^+},
\end{split}
\eeq
here $Z^+$ can be taken as $X^+, X_l^+, X_r^+$. \\

A direct calculation shows that for $c\in D_0$,
\beqno
\partial_yS^+f(y,c)=S_1^+f(y,c),
\eeqno
and
\beq\label{Sc+}
\begin{split}
\partial_cS^+f(y,c)&=-\int_{y_{c_+}}^y
\frac{\int_{y_{c_+}}^{y'}\big(W_-(z)-c)f(z,c)dz}
{\big(W_+(y')-c\big)\big(W_-(y')-c\big)}dy'
-\int_{y_{c_+}}^y
\frac{\int_{y_{c_+}}^{y'}\big(W_+(z)-c\big)f(z,c)dz}
{\big(W_+(y')-c\big)\big(W_-(y')-c\big)}dy'\nonumber\\
&\ \
+\int_{y_{c_+}}^y
\frac{\int_{y_{c_+}}^{y'}\big(W_+(z)-c\big)\big(W_-(z)-c\big)f(z,c)dz}
{\big(W_+(y')-c\big)^2\big(W_-(y')-c\big)}dy'\nonumber\\
&\ \
+\int_{y_{c_+}}^y
\frac{\int_{y_{c_+}}^{y'}\big(W_+(z)-c\big)\big(W_-(z)-c\big)f(z,c)dz}
{\big(W_+(y')-c\big)\big(W_-(y')-c\big)^2}dy'\nonumber\\
&\ \
+\int_{y_{c_+}}^y
\frac{\int_{y_{c_+}}^{y'}\big(W_+(z)-c\big)\big(W_-(z)-c\big)\pa_cf(z,c)dz}
{\big(W_+(y')-c\big)\big(W_-(y')-c\big)}dy',
\end{split}
\eeq
and
\beqno
\begin{split}
\partial_{cy}S^+f(y,c)
&=-\frac{\int_{y_{c_+}}^{y}\big(W_-(y')-c\big)f(y',c)dy'}
{\big(W_+(y)-c\big)\big(W_-(y)-c\big)}
-\frac{\int_{y_{c_+}}^{y}\big(W_+(y')-c\big)f(y',c)dy'}
{\big(W_+(y)-c\big)\big(W_-(y)-c\big)}\\
&\ \
+\frac{\int_{y_{c_+}}^{y}\big(W_+(y')-c\big)\big(W_-(y')-c\big)f(y',c)dy'}
{\big(W_+(y)-c\big)^2\big(W_-(y)-c\big)}\\
&\ \
+\frac{\int_{y_{c_+}}^{y}\big(W_+(y')-c\big)\big(W_-(y')-c\big)f(y',c)dy'}
{\big(W_+(y)-c\big)\big(W_-(y)-c\big)^2}\\
&\ \
+\frac{\int_{y_{c_+}}^{y}\big(W_+(y')-c\big)\big(W_-(y')-c\big)\pa_cf(y',c)dy'}
{\big(W_+(y)-c\big)\big(W_-(y)-c\big)}.
\end{split}
\eeqno
Thus,  from (\ref{eq: Sf}), (\ref{0}) and (\ref{00}), we obtain
\begin{align*}
\|S^+f\|_{Y_0^+}&=\|S^+f\|_{X_0^+}+\f{1}{A}\|\pa_yS^+f\|_{X_0^+}
+\f{1}{A}\|\pa_cS^+f\|_{X_0^+}+\f{1}{A^2}
\|\pa_{yc}S^+f\|_{X_0^+}\\
&\leq C\f{1}{A^2}\|f\|_{X_0^+}+\frac{C}{A^3}\|\pa_cf\|_{X_0^+}
\leq \f{C_1}{A^2}\|f\|_{Y_0^+}.
\end{align*}

For $c\in D_{\ep_0}$, we have
\beno
\pa_yS^+f(y,c)=S_1^+f(y,c),
\eeno
and
\beqno
\begin{split}
\pa_{\ep}S^+f(y,c)&
=S^+\pa_{\ep}f(y,c)+
i\int_{y_{c_+}}^y\frac{\int_{y_{c_+}}^{y'}\big(W_+(z)-c\big)\big(W_-(z)-c\big)f(z,c)dz}
{\big(W_+(y')-c\big)^2\big(W_-(y')-c\big)}dy'\\
&\quad
+i\int_{y_{c_+}}^y\frac{\int_{y_{c_+}}^{y'}\big(W_+(z)-c\big)\big(W_-(z)-c\big)f(z,c)dz}
{\big(W_+(y')-c\big)\big(W_-(y')-c\big)^2}dy'\\
&\quad
-i\int_{y_{c_+}}^y\frac{\int_{y_{c_+}}^{y'}\big(W_+(z)-c\big)f(z,c)dz}
{\big(W_+(y')-c\big)\big(W_-(y')-c\big)}dy'
-i\int_{y_{c_+}}^y\frac{\int_{y_{c_+}}^{y'}\big(W_-(z)-c\big)f(z,c)dz}
{\big(W_+(y')-c\big)\big(W_-(y')-c\big)}dy'.
\end{split}
\eeqno
If $c_r\geq 0$, then $W_+(y_{c_+})=c_r$, we obtain
\beqno
\begin{split}
\pa_{c_r}S^+f(y,c)
&=S^+\pa_{c_r}f(y,c)+
\int_{y_{c_+}}^y\frac{\int_{y_{c_+}}^{y'}\big(W_+(z)-c\big)\big(W_-(z)-c\big)f(z,c)dz}
{\big(W_+(y')-c\big)^2\big(W_-(y')-c\big)}dy'\\
&\quad
+\int_{y_{c_+}}^y\frac{\int_{y_{c_+}}^{y'}\big(W_+(z)-c\big)\big(W_-(z)-c\big)f(z,c)dz}
{\big(W_+(y')-c\big)\big(W_-(y')-c\big)^2}dy'\\
&\quad
-\int_{y_{c_+}}^y\frac{\int_{y_{c_+}}^{y'}\big(W_+(z)-c\big)f(z,c)dz}
{\big(W_+(y')-c\big)\big(W_-(y')-c\big)}dy'
-\int_{y_{c_+}}^y\frac{\int_{y_{c_+}}^{y'}\big(W_-(z)-c\big)f(z,c)dz}
{\big(W_+(y')-c\big)\big(W_-(y')-c\big)}dy'\\
&\quad
+i(W_+^{-1})'(c_r)f(y_{c_+},c)\big(W_-(y_{c_+})-c\big)
\int_{y_{c_+}}^y\frac{\ep}{\big(W_+(y')-c\big)\big(W_-(y')-c\big)}dy',
\end{split}
\eeqno
and if $c_r\leq 0$, then $W_-(y_{c_+})=c_r$ and we have
\beqno
\begin{split}
\pa_{c_r}S^+f(y,c)
&=S^+\pa_{c_r}f(y,c)+
\int_{y_{c_+}}^y\frac{\int_{y_{c_+}}^{y'}\big(W_+(z)-c\big)\big(W_-(z)-c\big)f(z,c)dz}
{\big(W_+(y')-c\big)^2\big(W_-(y')-c\big)}dy'\\
& \quad
+\int_{y_{c_+}}^y\frac{\int_{y_{c_+}}^{y'}\big(W_+(z)-c\big)\big(W_-(z)-c\big)f(z,c)dz}
{\big(W_+(y')-c\big)\big(W_-(y')-c\big)^2}dy'\\
& \quad
-\int_{y_{c_+}}^y\frac{\int_{y_{c_+}}^{y'}\big(W_+(z)-c\big)f(z,c)dz}
{\big(W_+(y')-c\big)\big(W_-(y')-c\big)}dy'
-\int_{y_{c_+}}^y\frac{\int_{y_{c_+}}^{y'}\big(W_-(z)-c\big)f(z,c)dz}
{\big(W_+(y')-c\big)\big(W_-(y')-c\big)}dy'\\
& \quad
+i(W_-^{-1})'(c_r)f(y_{c_+},c)\big(W_+(y_{c_+})-c\big)
\int_{y_{c_+}}^y\frac{\ep}{\big(W_+(y')-c\big)\big(W_-(y')-c\big)}dy'.
\end{split}
\eeqno
Due to (\ref{eq: W mono}), we have
\beqno
\left|\frac{\big(W_+(y_{c_+})-c\big)}{\cosh A(y-y_{c_+})}
\int_{y_{c_+}}^y\frac{\ep}{\big(W_+(y')-c\big)\big(W_-(y')-c\big)}dy'\right|
\leq \frac{C_1|y-y_{c_+}|}{\cosh A(y-y_{c_+})}\leq \frac{C_1}{A},
\eeqno
\beqno
\left|\frac{\big(W_-(y_{c_+})-c\big)}{\cosh A(y-y_{c_+})}
\int_{y_{c_+}}^y\frac{\ep}{\big(W_+(y')-c\big)\big(W_-(y')-c\big)}dy'\right|
\leq \frac{C_1|y-y_{c_+}|}{\cosh A(y-y_{c_+})}\leq \frac{C_1}{A},
\eeqno
Then from which and (\ref{8})-(\ref{888}), we get
\beq
\|S^+f\|_{Y^+}\leq \frac{C_1}{A^2}\|f\|_{Y^+}.
\eeq

For $c\in B_{\ep_0}^l$, we have
\beqno
\pa_yS^+f(y,c)=S_1^+f(y,c),
\eeqno
and
\beqno
\begin{split}
\pa_{\ep}S^+f(y,c)&=S^+\pa_{\ep}f(y,c)+
e^{i\theta}\Big\{\int_{y_{c_+}}^y\frac{\int_{y_{c_+}}^{y'}
\big(W_+(z)-c\big)\big(W_-(z)-c\big)f(z,c)dz}
{\big(W_+(y')-c\big)^2\big(W_-(y')-c\big)}dy'\\
& \ \
+\int_{y_{c_+}}^y\frac{\int_{y_{c_+}}^{y'}\big(W_+(z)-c\big)\big(W_-(z)-c\big)f(z,c)dz}
{\big(W_+(y')-c\big)\big(W_-(y')-c\big)^2}dy'\\
& \ \
-\int_{y_{c_+}}^y\frac{\int_{y_{c_+}}^{y'}\big(W_+(z)-c\big)f(z,c)dz}
{\big(W_+(y')-c\big)\big(W_-(y')-c\big)}dy'
-\int_{y_{c_+}}^y\frac{\int_{y_{c_+}}^{y'}\big(W_-(z)-c\big)f(z,c)dz}
{\big(W_+(y')-c\big)\big(W_-(y')-c\big)}dy'\Big\},
\end{split}
\eeqno
and
\beqno
\begin{split}
\pa_{\theta}S^+f(y,c)&=S^+\pa_{\theta}f(y,c)+
i\ep e^{i\theta}\Big\{\int_{y_{c_+}}^y
\frac{\int_{y_{c_+}}^{y'}\big(W_+(z)-c\big)\big(W_-(z)-c\big)f(z,c)dz}
{\big(W_+(y')-c\big)^2\big(W_-(y')-c\big)}dy'\\
& \ \
+\int_{y_{c_+}}^y\frac{\int_{y_{c_+}}^{y'}\big(W_+(z)-c\big)\big(W_-(z)-c\big)f(z,c)dz}
{\big(W_+(y')-c\big)\big(W_-(y')-c\big)^2}dy'\\
& \ \
-\int_{y_{c_+}}^y\frac{\int_{y_{c_+}}^{y'}\big(W_+(z)-c\big)f(z,c)dz}
{\big(W_+(y')-c\big)\big(W_-(y')-c\big)}dy'
-\int_{y_{c_+}}^y\frac{\int_{y_{c_+}}^{y'}\big(W_-(z)-c\big)f(z,c)dz}
{\big(W_+(y')-c\big)\big(W_-(y')-c\big)}dy'\Big\}.
\end{split}
\eeqno
Then we can get that by (\ref{8})-(\ref{888})
\beq
\|S^+f\|_{Y_l^+}\leq \frac{C_1}{A^2}\|f\|_{Y_l^+}.
\eeq

For $c\in B_{\ep_0}^r$, the proof is similar, we omit it for the sake of brevity.

Now we check the continuity. We rewrite $S_1^+f$ as
\beno
S_1^+f=\int_0^1K(t,y,c)f(y_{c_+}+t(y-y_{c_+}),c)dt,
\eeno
with $K(t,y,c)=\frac{(y-y_{c_+})\big(W_+(y_{c_+}+t(y-y_{c_+}))-c\big)\big(W_-(y_{c_+}+t(y-y_{c_+}))-c\big)}{(W_+(y)-c)(W_-(y)-c)}$.

Using the fact that for $(y,c)\in d_+\times\Om_{\ep_0}$ with  $0<y<y'<y_{c_+}$ or $y_{c_+}<y'<y<a_+$, $\Big|\frac{\big(W_+(y')-c\big)\big(W_-(y')-c\big)}{\big(W_+(y)-c\big)
\big(W_-(y)-c\big)}\Big|\leq C$,
the continuity of $K(t,y,c)$ and the Lebesgue's dominated convergence theorem, we conclude the continuity of $S_1^+f$. The continuity of $S^+f$ follows from $S^+f=S_0^+\circ S_1^+f$.
\end{proof}

\subsection{Existence of the solution}

In the following, the constant $C$ may depend on $\al$.

The homogeneous Sturmian equation on $[0,a_+]$ is
\beq\label{eq: homo eq}
\Big\{\begin{array}{l}
 \pa_y\Big(\mathcal{H}(y,c)\pa_y\va_+(y,c)\Big)=\al^2\mathcal{H}(y,c)\va_+(y,c),\\
\va_+(y_{c_+},c)=1, \ \pa_y\va_+(y_{c_+},c)=0.
\end{array}\Big.
\eeq

\begin{proposition}\label{prop: sol. hom. [0,1]}
1. For $c\in \Omega_{\ep_0}$, there exists a solution
$\va_+(y,c)\in C([0,a_+]\times \Omega_{\ep_0})$ of the
Sturmian equation \eqref{eq: homo eq} and $\pa_y\va_+(y,c)\in C([0,a_+]\times \Om_{\ep_0})$.
 Moreover, there exists $\ep_1>0$ such that for any $\ep_0\in[0,\ep_1)$
  and $(y,c)\in [0,a_+]\times \Omega_{\ep_0}$,
\beno
|\va_+(y,c)|\geq \frac{1}{2}, \ \ |\va_+(y,c)-1|\leq C|y-y_{c_+}|^2,
\eeno
where the constants $\ep_1, C$ may depend on $\al$.\\
2. For $c\in D_0$, for any $y\in[0,a_+]$, there is a constant $C$(depends on $\al$) such that,
\beno
\va_+(y,c)\geq \va_+(y',c)\geq 1, \ \ for \ \ 0\leq y_{c_+}\leq y'\leq y\leq 1
\ \ or\ \ 0\leq y\leq y'\leq y_{c_+}\leq 1;
\eeno
\beno
&\ &0\leq \va_+(y,c)-1\leq C\min\big\{\al^2(y-y_{c_+})^2, 1\big\}\va_+(y,c),
\eeno
\beno
C^{-1} |y-y_{c_+}|\leq |\pa_y\va_+(y,c)|\leq C |y-y_{c_+}|,
\eeno
and
\beno
|\pa_c\va_+(y,c)|\leq C |y-y_{c_+}|, \ \
|\pa_y\pa_c\va_+(y,c)|\leq C .
\eeno
\end{proposition}

The proof is based on the following lemmas.
\begin{lemma}\label{lem: Y+}
Let $c\in D_{\ep_0}$.
Then there exists a solution $\va_+(y,c)\in Y^+$ to the Sturmian equation \eqref{eq: homo eq}. Moreover, it holds
\beno
\|\va_+\|_{Y^+}\leq C.
\eeno
\end{lemma}
\begin{proof}
$\va_+$ satisfies
\beqno
\pa_y\Big(\mathcal{H}(y,c)\pa_y\va_+(y,c)\Big)=\al^2\mathcal{H}(y,c)\va_+(y,c),
\eeqno
from which, we infer that
\beqno
\va_+(y,c)=1+\int_{y_{c_+}}^y\frac{\al^2}{\mathcal{H}(y',c)}
\int_{y_{c_+}}^{y'}\mathcal{H}(z,c)\va_+(z,c)dzdy'.
\eeqno
This means that $\va_+$ satisfies
\beqno
\va_+(y,c)=1+\al^2S^+\va_+(y,c).
\eeqno
In follows from Proposition \ref{prop: S+ estimate} that the
operator $I-\al^2 S^+$  is invertible in the space $Y^+$, if
\beqno
\frac{\al^2C_1}{A^2}\leq \frac{1}{2}<1,
\eeqno
where $C_1$ is the constant in Proposition \ref{prop: S+ estimate}. Thus
\beqno
\va_+(y,c)=(I-\al^2S^+)^{-1}1.
\eeqno
Hence, $\|\va_+\|_{Y^+}\leq \|1\|_{Y^+}+\|\al^2S^+\va_+\|_{Y^+}
\leq C+\frac{1}{2}\|\va_+\|_{Y^+}$ implies $\|\va_+\|_{Y^+}\leq C$,
here $C$ is a constant independent of $A$ and $\al$.
\end{proof}
In a similar way as in Lemma \ref{lem: Y+}, we can show that
\begin{lemma}\label{lem: Y+l}
Let $c\in B_{\ep_0}^l$. Then there exists a solution
$\va_+(y,c)\in Y^+_l$ to the Sturmian equation \eqref{eq: homo eq}.
Moreover, it holds
\beqno
\|\va_+\|_{Y^+_l}\leq C.
\eeqno
\end{lemma}

\begin{lemma}\label{lem: Y+r}
Let $c\in B_{\ep_0}^r$. Then there exists a solution $\va_+(y,c)\in Y^+_r$
to the Sturmian equation \eqref{eq: homo eq}. Moreover, it holds
\beqno
\|\va_+\|_{Y^+_r}\leq C.
\eeqno
\end{lemma}
\begin{lemma}\label{lem: Y+0}
Let $c\in D_0$.
 Then there exists a solution $\va_+(y,c)\in Y^+_0$ to the Sturmian equation \eqref{eq: homo eq}. Moreover, it holds
\beqno
\|\va_+\|_{Y^+_0}\leq C.
\eeqno
\end{lemma}
Now we are in a position to prove the Proposition \ref{prop: sol. hom. [0,1]}.
\begin{proof}
{\bf Proof of 1. }
Let us define
\beqno
 \va_+(y,c)\stackrel{def}{=}\left\{\begin{array}{l}
 \va_+^0(y,c) \ \ for \ \ c\in D_0,\\
\va_+^{\pm}(y,c) \ \ for \ \ c\in D_{\ep_0},\\
\va_+^l(y,c) \ \ for \ \ c\in B_{\ep_0}^l,\\
\va_+^r(y,c) \ \ for \ \ c\in B_{\ep_0}^r,\\
\end{array}\right.
\eeqno
where $\va_+^{\pm}, \va_+^l, \va_+^r, \va_+^0$ are given
by Lemma \ref{lem: Y+},  Lemma \ref{lem: Y+l}, Lemma \ref{lem: Y+r} and
Lemma \ref{lem: Y+0} respectively. Then $\va_+(y,c)$ is our desired solution.

By Proposition \ref{prop: S+ estimate} and using the formula
$\va_+(y,c)=\sum\limits_{k=0}^{+\infty}\al^{2k}(S^+)^k1$ for $\Omega_{\ep_0}$, we can conclude
that $\va_+(y,c)\in C([0,a_+]\times \Omega_{\ep_0})$.
Moreover, for $c\in D_0$, we have
\beno
(S^+)^k1(y,c)\geq 0, \ \ \va_+(y,c)\geq 1,
\eeno
which ensures that there exists $\ep_1>0$ so that for
any $\ep_0\in[0,\ep_1)$ and $(y,c)\in d_+\times \Om_{\ep_0}$,
\beno
|\va_+(y,c)|\geq \frac{1}{2}, \ \ |\va_+(y,c)|\leq C.
\eeno
 Thanks to $\va_+=1+\al^2S^+\va_+$, we have
 \ben\label{va+}
 \va_+(y,c)=1+\int_{y_{c_+}}^y
 \frac{\al^2}{\mathcal{H}(y',c)}\int_{y_{c_+}}^{y'}\mathcal{H}(z,c)\va_+(z,c)dzdy',
 \een
 from which, it follows that
 \beno
 \pa_y\Big(\mathcal{H}(y,c)\pa_y\va_+(y,c)\Big)=\al^2\mathcal{H}(y,c)\va_+(y,c),\ \ \va_+(y_{c_+},c)=1.
 \eeno
 Then $\va_+(y,c)$ satisfies the Sturmian equation \eqref{eq: homo eq}.

By the fact that $\pa_y\va_+(y,c)=\al^2S_1^+\va_+(y,c)$ and  Proposition \ref{prop: S+ estimate}, we have $\pa_y\va_+(y,c)\in C(d_+\times \Om_{\ep_0})$.

From (\ref{va+}), we have
 \beqno
 |\va_+(y,c)-1|\leq \al^2\int_{y_{c_+}}^y\int_{y_{c_+}}^{y'}|\va_+(z,c)|
 \Big|\frac{(W_+(z)-c)(W_-(z)-c)}{(W_+(y')-c)(W_-(y')-c)}\Big|dzdy'\leq C|y-y_{c_+}|^2.
 \eeqno

\no{\bf Proof of 2. } Since we have $\pa_y\va_+(y,c)=\al^2S_1^+\va_+(y,c)$, thus for $y\geq z\geq y_{c_+}$ or $y\leq z\leq y_{c_+}$, $\frac{\mathcal{H}(z,c)}
{\mathcal{H}(y,c)}\geq 0$ and then $\pa_y\va_+(y,c)\geq 0$,
for $y\geq y_{c_+}$ and $\pa_y\va_+(y,c)\leq 0$, for $y\leq y_{c_+}$.

Since $\va_+(y,c)-1=\al^2S^+\va_+(y,c)$, using (\ref{eq: W mono}), we have
\beqno
0\leq S^+\va_+(y,c)\leq \Big(\int_{y_{c_+}}^y\int_{y_{c_+}}^{y'}\Big|\frac{\mathcal{H}(z,c)}
{\mathcal{H}(y',c)}\Big|dzdy'\Big)\va_+(y,c)\leq C|y-y_{c_+}|^2\va_+(y,c),
\eeqno
then we obtain
\beno
0\leq \va_+(y,c)-1\leq C|y-y_{c_+}|^2.
\eeno
By the fact that $|S_1^+\va_+(y,c)|\leq C |y-y_{c_+}|\va_+(y,c)$, we have
\beno
|\pa_y\va_+(y,c)|\leq C|y-y_{c_+}|\va_+(y,c).
\eeno
On the other hand, $\va_+\geq 1$ and $\mathcal{H}(y,c)\geq C^{-1}|y-y_{c_+}||y+y_{c_+}|$, thus
\begin{align*}
|S_1^+\va_+(y,c)|\geq \left|\int_{y_{c_+}}^y\Big|\frac{\mathcal{H}(z,c)}{\mathcal{H}(y,c)}\Big|dz\right|\geq C^{-1}|y-y_{c_+}|.
\end{align*}
By Lemma \ref{lem: Y+0}, we get $\va_+(y,c)\in Y_0^+$, which implies
\beno
|\va_+(y,c)|+|\pa_c\va_+(y,c)|+|\pa_c\pa_y\va_+(y,c)|\leq C,
\eeno
with $C$ depending on $\al$. Then by the fact
that $\pa_c\va_+(y_{c_+},c)=0$, we have
\beno
|\partial_c\va_+(y,c)|=\left|\int_{y_{c_+}}^y\pa_y\pa_c\va_+(y',c)dy'\right|\leq C|y-y_{c_+}|,
\eeno
with $C$ depending on $\al$.

This completes the proof of the proposition.
\end{proof}

Similarly, for the case $y\in [a_-,0]$, we solve the Sturmian  equation:
\beq\label{eq: homo eq1}
 \left\{\begin{array}{l}
 \pa_y\Big(\mathcal{H}(y,c)\pa_y\va_-(y,c)\Big)=\al^2\mathcal{H}(y,c)\va_-(y,c),\\
\va_-(y_{c_-},c)=1, \va'_-(y_{c_-},c)=0.\\
\end{array}\right.
\eeq
Here we only give the conclusion, and omit the details of the proof for brevity.
\begin{proposition}\label{prop: sol. hom. [-1,0]}
1. For $c\in \Omega_{\ep_0}$, there exists a solution
$\va_-(y,c)\in C(d_-\times \Omega_{\ep_0})$ of the Sturmian
equation \eqref{eq: homo eq1} and  $\pa_y\va_-(y,c)\in C(d_-\times \Om_{\ep_0})$.
 Moreover,  there exists $\ep_1>0$ such that for any $\ep_0\in[0,\ep_1)$ and $(y,c)\in d_-\times \Om_{\ep_0}$,
\beqno
|\va_-(y,c)|\geq \frac{1}{2}, \ \ |\va_-(y,c)-1|\leq C|y-y_{c_-}|^2,
\eeqno
where the constants $\ep_1, C$ may depend on $\al$.\\
2. For $c\in D_0$,  we have that for any $y\in d_-$, there is a constant $C$(depends on $\al$) such that,
\beno
\va_-(y,c)\geq \va_-(y',c)\geq1, \  for \  -1\leq y_{c_-}
\leq y'\leq y\leq 0 \ or\  -1\leq y\leq y'\leq y_{c_-}\leq 0;
\eeno
and
\beno
&\ &0\leq \va_-(y,c)-1\leq C\min\big\{\al^2(y-y_{c_-})^2, 1\big\}\va_-(y,c)
\eeno
\beno
C^{-1}|y-y_{c_-}|\leq |\pa_y\va_-(y,c)|\leq C|y-y_{c_-}|,
\eeno
and
\beno
|\pa_c\va_-(y,c)|\leq C|y-y_{c_-}|,
 \ \ |\pa_y\pa_c\va_-(y,c)|\leq C.
\eeno
\end{proposition}
The proposition can be proved mainly by the following lemmas.
\begin{lemma}
Let $c\in D_{\ep_0}$.
Then there exists a solution $\va_-(y,c)\in Y^-$ to the Sturmian
equation \eqref{eq: homo eq1}. Moreover, it holds
\beno
\|\va_-\|_{Y^-}\leq C.
\eeno
\end{lemma}
\begin{lemma}
Let $c\in B_{\ep_0}^l$. Then there exists a solution
$\va_-(y,c)\in Y^-_l$ to the Sturmian equation \eqref{eq: homo eq1}).
Moreover, it holds
\beqno
\|\va_-\|_{Y^-_l}\leq C.
\eeqno
\end{lemma}

\begin{lemma}
Let $c\in B_{\ep_0}^r$. Then there exists a solution $\va_-(y,c)\in Y^-_r$
to the Sturmian equation \eqref{eq: homo eq1}. Moreover, it holds
\beqno
\|\va_-\|_{Y^-_r}\leq C.
\eeqno
\end{lemma}
\begin{lemma}
Let $c\in D_0$.
 Then there exists a solution $\va_-(y,c)\in Y^-_0$ to the Sturmian equation \eqref{eq: homo eq1}. Moreover, there holds
\beqno
\|\va_-\|_{Y^-_0}\leq C.
\eeqno
\end{lemma}
\begin{remark}
From the above construction, we  note that $\va_+(y,c)$, $\va_-(y,c)$ may be not equal at the point $y=0$.
\end{remark}
\begin{remark}\label{rmk: epsilon}
By the definition of $Y^{\pm}$ and $Y_r^{\pm}$, $Y_l^{\pm}$ and Proposition \ref{prop: sol. hom. [0,1]}, Proposition \ref{prop: sol. hom. [-1,0]}, we have for $c\in D_0$ and $c_{\ep}=c+i\ep\in D_{\ep_0}\cup D_0$ with $0\leq |\ep|\leq \ep_0$, 
\beno
|\va_{\pm}(y,c_{\ep})-\va_{\pm}(y,c)|\leq C|\ep|,
\eeno
and for $c_{\ep}=c+\ep e^{i\th}\in B_{\ep_0}^l$ or $c_{\ep}=c+\ep e^{i\th}\in B_{\ep_0}^r$ with $0\leq |\ep|\leq \ep_0$, 
\beno
|\va_{\pm}(y,c_{\ep})-\va_{\pm}(y,c)|\leq C|\ep|.
\eeno
\end{remark}

\section{The inhomogeneous Sturmian equations}
\subsection{The Wronskian and its estimate}
In the following, we introduce  for $c\in \Omega_{\ep_0}\setminus D_0$,
\beq
I_+(c)=\int_0^1\frac{1}{\mathcal{H}(y,c)\va_+(y,c)^2}dy,\ \
I_-(c)=\int_{-1}^0\frac{1}{\mathcal{H}(y,c)\va_-(y,c)^2}dy,
\eeq
\beq
P(c)=\va_-(0,c)^2(\va_+\pa_y\va_+)(0,c)-\va_+(0,c)^2(\va_-\pa_y\va_-)(0,c),
\eeq
\ben\label{wronskian}
\mathcal{D}(c)=c^2P(c)I_+(c)I_-(c)-\va_+(0,c)^2I_+(c)-\va_-(0,c)^2I_-(c).
\een
Here and as what follows, $\va_+(y,c),\va_-(y,c)$ are the solutions of the homogeneous Sturmian
equation constructed in Proposition \ref{prop: sol. hom. [0,1]} and
Proposition \ref{prop: sol. hom. [-1,0]}.

The Stern stability condition ({\bf{SS}}) was proved in \cite{Stern}. Here we recall the lemma and show the relationship between the Stern stability condition ({\bf{SS}}) and the Wronskian $\mathcal{D}(c)$.
\begin{lemma}\label{stern lemma}
 If $|u(y)|\leq|b(y)|$ for $y\in[-1,1]$, $M_{\al}$ has no $H^1$  eigenvalue.
  Thus for any $c\notin D_0$, the  Sturmian equation
 \beq\label{homo Sturmian}
 \left\{\begin{array}{l}
\pa_y\Big(\mathcal{H}(y,c)\pa_y\Psi(y,c)\Big)-\al^2\mathcal{H}(y,c)\Psi(y,c)=0\\
\Psi(-1,c)=\Psi(1,c)=0.\\
\end{array}\right.
\eeq
has no $H^1(-1,1)$ solution.   And then  we have $\mathcal{D}(c)\neq0$ for $c\in \Om_{\ep_0}\setminus D_0$.
 \end{lemma}
 \begin{proof}
For $c=c_r+ic_i\notin D_0$, let $\Psi(y,c)\in H^1_0(-1,1)$ be a nontrivial solution of the Sturmian equation
 \beq\label{sturmian eq}
\pa_y\Big(\mathcal{H}(y,c)\pa_y\Psi(y,c)\Big)-\al^2\mathcal{H}(y,c)\Psi(y,c)=0.
\eeq
Taking the inner product with $\overline{\Psi}(y,c)$ on both sides of (\ref{sturmian eq})
 and by integration by parts, we obtain
\ben\label{integral}
\int_{-1}^1\mathcal{H}(y,c)\Big(|\pa_y\Psi(y,c)|^2+\al^2|\Psi(y,c)|^2\Big)dy=0.
\een
Due to
\beqno
\begin{split}
\mathcal{H}(y,c)&=\big(u(y)+b(y)-c\big)\big(u(y)-b(y)-c\big)\\
&=\big(u(y)+b(y)-c_r\big)\big(u(y)-b(y)-c_r\big)
-c_i^2-2ic_i\big(u(y)-c_r\big),
\end{split}
\eeqno
taking the real part of (\ref{integral}) gives
\beq\label{re}
\int_{-1}^1\Big[\big(u(y)+b(y)-c_r\big)\big(u(y)-b(y)-c_r\big)-c_i^2\Big]
\Big(|\pa_y\Psi(y,c)|^2+\al^2|\Psi(y,c)|^2\Big)dy=0,
\eeq
and taking the imagine part of (\ref{integral}) gives
\ben\label{im}
\int_{-1}^1(u(y)-c_r)\Big(|\pa_y\Psi(y,c)|^2+\al^2|\Psi(y,c)|^2\Big)dy=0.
\een
Then multiplying $2c_r$ on both sides of (\ref{im}) and adding (\ref{re}), we get
\beno
\int_{-1}^1\Big[u(y)^2-b(y)^2-c_r^2-c_i^2\Big]
\Big(|\pa_y\Psi(y,c)|^2+\al^2|\Psi(y,c)|^2\Big)dy=0.
\eeno
Thus if $|u(y)|\leq |b(y)|$ for $y\in[-1,1]$, we have $\Psi(y,c)\equiv0$, which leads to a contradiction.\\

Let $\Psi(y,c)\in H^1(-1,1)$ is a solution of (\ref{homo Sturmian}) and assume $\Psi(y,c)=\left\{\begin{array}{l}
\Psi_+(y,c), \ \ y\in[0,1], \\
\Psi_-(y,c), \ \ y\in[-1,0],\\
\end{array}\right.$ then we have for $y\in[0,1]$, $\Psi_+$ satisfies
\beq\label{Psi+}
 \left\{\begin{array}{l}
\pa_y\Big(\mathcal{H}(y,c)\pa_y\Psi_+(y,c)\Big)-\al^2\mathcal{H}(y,c)\Psi_+(y,c)=0\\
\Psi_+(1,c)=0.\\
\end{array}\right.
\eeq
and for $y\in[-1,0]$, $\Psi_-$ satisfies
\beq\label{Psi-}
 \left\{\begin{array}{l}
\pa_y\Big(\mathcal{H}(y,c)\pa_y\Psi_-(y,c)\Big)-\al^2\mathcal{H}(y,c)\Psi_-(y,c)=0\\
\Psi_-(-1,c)=0.\\
\end{array}\right.
\eeq

By Proposition \ref{prop: sol. hom. [0,1]} and Proposition \ref{prop: sol. hom. [-1,0]}, $\va_+,\va_-\neq 0$, then it holds that the equations (\ref{Psi+}) and (\ref{Psi-}) are equivalent to
\beqno
 \left\{\begin{array}{l}
\pa_y\Big(\mathcal{H}\va_+^2\pa_y\Big(\frac{\Psi_+}{\va_+}\Big)\Big)=0\\
\Psi_+(1,c)=0,\\
\end{array}\right.
\quad\textrm{and}\quad
 \left\{\begin{array}{l}
\pa_y\Big(\mathcal{H}\va_-^2\pa_y\Big(\frac{\Psi_-}{\va_-}\Big)\Big)=0\\
\Psi_-(1,c)=0.\\
\end{array}\right.
\eeqno
 As $\mathcal{H}\neq0$ for $c\notin D_0$, by integration twice, we obtain that $\va_+(y,c)$ and
$\va_+(y,c)\int_1^y\frac{1}{\mathcal{H}(y,c)\va_+(y,c)^2}dy'$ are two independent solutions
of the homogeneous Sturmian equation for $y\in[0,1]$, and $\va_-(y,c)$
and $\va_-(y,c)\int_{-1}^y\frac{1}{\mathcal{H}(y,c)\va_-(y,c)^2}dy'$ are two independent
solutions of the homogeneous Sturmian equation for $y\in[-1,0]$.
Thus for $y\in[0,1]$ and $c\notin D_0$, we have
\beqno
\begin{split}
\Psi_+(y,c)&=\wt{\mu}_+(c) \va_+(y,c)\int_0^y\frac{1}{\mathcal{H}(y',c)\va_+(y',c)^2}dy'
+\nu_+(c)\va_+(y,c):= \Psi_+^0(y,c)\\
&=\mu_+(c) \va_+(y,c)\int_1^y\frac{1}{\mathcal{H}(y',c)\va_+(y',c)^2}dy':= \Psi_+^1(y,c),
\end{split}
\eeqno
and for $y\in[-1,0]$ and $c\notin D_0$,
\beqno
\begin{split}
\Psi_-(y,c)&=\mu_-(c) \va_-(y,c)\int_{-1}^y\frac{1}{\mathcal{H}(y',c)\va_-(y',c)^2}dy'
:=\Psi_-^{-1}(y,c)\\
&=\wt{\mu}_-(c) \va_-(y,c)\int_0^y\frac{1}{\mathcal{H}(y',c)\va_-(y',c)^2}dy'
+\nu_-(c)\va_-(y,c):=\Psi_-^0(y,c).
\end{split}
\eeqno
By the boundary conditions and the fact that $\Psi(y,c)\in H_0^{1}(-1,1)$, we get,
\beno
&\ & \Psi_+^0(1,c)=0, \  \Psi_-^0(-1,c)=0, \
\Psi_+^1(0,c)=\Psi_+^0(0,c), \\
&\ &  \Psi_-^{-1}(0,c)=\Psi_-^0(0,c),\
\Psi_+^0(0,c)=\Psi_-^0(0,c), \  \pa_y\Psi_+^0(0,c)=\pa_y\Psi_-^0(0,c),
\eeno
which gives
\beqno
 \left\{\begin{array}{l}
\mu_+(c)=\wt{\mu}_+(c), \ \
 \mu_-(c)=\wt{\mu}_-(c),\\
I_+(c)\mu_+(c)+\nu_+(c)=0,\\
I_-(c)\mu_-(c)-\nu_-(c)=0,\\
\va_+(0,c)\nu_+(c)-\va_-(0,c)\nu_-(c)=0,\\
\va_-(0,c)\mu_+(c)-\va_+(0,c)\mu_-(c)
+c^2\big(\va_-\va_+\pa_y\va_+\big)(0,c)\nu_+(c)\\
\ \  -c^2\big(\va_+\va_-\pa_y\va_-\big)(0,c)\nu_-(c)=0.
\end{array}\right.
\eeqno
Thus we have
\begin{equation}\label{matrix}
W
 \begin{bmatrix}
 \mu_+(c)\\
 \mu_-(c)\\
 \nu_+(c)\\
 \nu_-(c)
  \end{bmatrix}
=0,
\end{equation}
where
$
W=\begin{bmatrix}
 I_+(c)& 0 & 1&  0\\
 0&  I_-(c)& 0 &  -1\\
 0& 0&  \va_+(0,c)&  -\va_-(0,c)\\
 \va_-(0,c)&  -\va_+(0,c)& c^2\big(\va_-\va_+\pa_y\va_+\big)(0,c)
 & -c^2\big(\va_+\va_-\pa_y\va_-\big)(0,c)
 \end{bmatrix},
 $
 and the fact that \eqref{homo Sturmian} has no nontrivial $H^1$ solution, we obtain the Wronskian $\det(W)\neq 0$ holds for $c\in \Om_{\ep_0}\setminus D_0$, which gives
\beno
 \det(W)=\mathcal{D}(c)=c^2P(c)I_+(c)I_-(c)-\va_+(0,c)^2I_+(c)-\va_-(0,c)^2I_-(c)\neq0.
\eeno
Thus we proved the lemma.
\end{proof}
 \begin{remark}
 Thanks to the continuity of $\va_+$ and $\va_-$, we have
 that $\mathcal{D}(c)$ is continuous for $c\in \Om_{\ep_0}\setminus D_0$.
 \end{remark}

Next we show that $\mathcal{D}(c)$ is continuous to the boundary.

For $c\in D_{\ep_0}\cup D_0$, let
\beq\label{sigma+}
\sigma_+(c)=W'_+(y_{c_+})\big(W_-(y_{c_+})-c\big)
-W'_-(y_{c_+})\big(W_+(y_{c_+})-c\big),
\eeq
\beq\label{sigma-}
\sigma_-(c)=W'_+(y_{c_-})\big(W_-(y_{c_-})-c\big)
-W'_-(y_{c_-})\big(W_+(y_{c_-})-c\big),
\eeq
\beq\label{Pi}
\Pi_+(c)=\int_0^1\frac{1}{\mathcal{H}(y,c)}\Big(\frac{1}{\va_+(y,c)^2}-1\Big)dy, \ \
\Pi_-(c)=\int_{-1}^0\frac{1}{\mathcal{H}(y,c)}\Big(\frac{1}{\va_-(y,c)^2}-1\Big)dy,
\eeq

For $c\in D_0$, let
\beq\label{R+1}
R_+^1(c)=\int_0^1\frac{W'_+(y_{c_+})\big[W_-(y_{c_+})-W_-(y)\big]
-W'_-(y_{c_+})\big[W_+(y_{c_+})-W_+(y)\big]}
{\big(W_+(y)-c\big)\big(W_-(y)-c\big)}dy,
\eeq
\beq\label{R+2}
R_+^2(c)=\int_0^1\frac{W'_+(y_{c_+})-W'_+(y)}{W_+(y)-c}
-\frac{W'_-(y_{c_+})-W'_-(y)}{W_-(y)-c}dy,
\eeq
\beq\label{R-1}
R_-^1(c)=\int_{-1}^0\frac{W'_+(y_{c_-})\big[W_-(y_{c_-})-W_-(y)\big]
-W'_-(y_{c_-})\big[W_+(y_{c_-})-W_+(y)\big]}
{\big(W_+(y)-c\big)\big(W_-(y)-c\big)}dy,
\eeq
\beq\label{R-2}
R_-^2(c)=\int_{-1}^0\frac{W'_+(y_{c_-})-W'_+(y)}{W_+(y)-c}
-\frac{W'_-(y_{c_-})-W'_-(y)}{W_-(y)-c}dy,
\eeq

For $c\in D_0\setminus \big\{0,W_+(1),W_-(1),W_+(-1),W_-(-1)\big\}$, let
\beno
\chi_+(c)=\left\{\begin{aligned}
1,\ c\in (W_-(1),W_+(1)),\\
0,\ c\notin [W_-(1),W_+(1)];
\end{aligned}
\right.
\quad
\chi_-(c)=\left\{\begin{aligned}
1,\ c\in (W_+(-1),W_-(-1)),\\
0,\ c\notin [W_+(-1),W_-(-1)].
\end{aligned}
\right.
\eeno
Then $\chi_+(c)^2+\chi_-(c)^2\geq 1$ for $c\in D_0\setminus \big\{0,W_+(1),W_-(1),W_+(-1),W_-(-1)\big\}$. 

We denote for $c\in D_0\setminus \big\{0,W_+(1),W_-(1),W_+(-1),W_-(-1)\big\}$, 
\beq\label{I+ re def}
I^{re}_+(c)=\Pi_+(c)+\frac{1}{\sigma_+(c)}
\left(R_+^1(c)+R_+^2(c)+\f12\ln \Big(\frac{W_+(1)-c}{c-W_-(1)}\Big)^2\right),
\eeq
\beq\label{I- re def}
I^{re}_-(c)=\Pi_-(c)+\frac{1}{\sigma_-(c)}
\left(R_-^1(c)+R_-^2(c)+\f12\ln\Big(\frac{W_-(-1)-c}{c-W_+(-1)}\Big)^2\right),
\eeq
\beq
\mathcal{D}^{re}(c)=c^2P(c)\Big(I^{re}_+(c)I^{re}_-(c)
-\frac{\pi^2\chi_+(c)\chi_-(c)}{\sigma_+(c)\sigma_-(c)}\Big)
-\va_+(0,c)^2I^{re}_+(c)-\va_-(0,c)^2I^{re}_-(c),
\eeq
\beq
\mathcal{D}^{im}(c)=c^2P(c)\Big(\frac{\pi I^{re}_+(c)\chi_-(c)}{\sigma_-(c)}
+\frac{\pi I^{re}_-(c)\chi_+(c)}{\sigma_+(c)}\Big)
-\frac{\pi\va_+(0,c)^2\chi_+(c)}{\sigma_+(c)}
-\frac{\pi\va_-(0,c)^2\chi_-(c)}{\sigma_-(c)}.
\eeq

We also define $l(x)=\ln(e+|x|^{-1})$ for $x\in \mathbb{C}$, so that $C(M)^{-1}(1+|\ln|x||)\leq l(x)\leq C(M)(1+|\ln|x||)$ for $|x|\leq M$.

\begin{remark}\label{rmk: sigma bdd}
By the definition of $\sigma_+(c)$ and $\sigma_-(c)$, for $c\in D_{\ep_0}\cup D_0$, we can easily get that there exists a positive constant $C$ such that
\beno
C^{-1} |c|\leq|\sigma_+(c)|\leq C |c|, \ C^{-1} |c|\leq|\sigma_-(c)|\leq C |c|.
\eeno
\end{remark}

\begin{proposition}\label{prop: D est}
There exists $\ep_0>0$, such that for $c_{\ep}\in \Om_{\ep_0}$, the following properties hold. \\
1. For $c\in  D_0\setminus \big\{0,W_+(1),W_-(1),W_+(-1),W_-(-1)\big\}$, $c_{\ep}=c+i\ep$. It holds
 \beno
 \lim_{\ep\rightarrow0^{\pm}} \mathcal{D}(c_{\ep})=\mathcal{D}^{re}(c)\pm i\mathcal{D}^{im}(c).
 \eeno
 Moreover, there exists a constant $C\geq 1$ such that
 \beno
\mathcal{D}^{re}(c)^2+\mathcal{D}^{im}(c)^2\geq C^{-1} >0.
 \eeno
2. For $c_{\ep}=c+i\ep\in \Om_{\ep_0}\setminus D_0$, $0<\ep<\ep_0$, there exists a constant $\delta_0>0$ such that for $|c|<\delta_0$,
 \beno
 |\mathcal{D}(c_{\ep})|\geq \frac{C^{-1}}{|c_{\ep}|};
 \eeno
3. For $c_{\ep}\in \Omega_{\ep_0}\setminus D_0$, it holds that
 \beno
 |\mathcal{D}(c_{\ep})|\geq \frac{
 l\big(c_{\ep}-W_-(-1)\big)l\big(W_+(-1)-c_{\ep}\big)l\big(c_{\ep}-W_-(1)\big)l\big(W_+(1)-c_{\ep}\big)}{C|c_{\ep}|}.
 \eeno
Here we recall $l(x)=\ln(e+|x|^{-1})$.
 \end{proposition}

The proof of the proposition is mainly dependent on the following lemmas.

\begin{lemma}\label{lem: I limit+}
For $c_{\ep}= c+i\ep \in D_{\ep_0}$, $c\in D_0\setminus \big\{W_+(1),W_-(1)\big\}$, $\ep\in(0,\ep_0)$. It holds that
\beq\label{I+ lim}
\lim_{\ep\rightarrow 0{\pm}}\sigma_+(c_{\ep})I_+(c_{\ep})
=\sigma_+(c)I_+^{re}(c)\pm i\pi\chi_+(c),
\eeq
where
$\sigma_+(c_{\ep})$ is defined as \eqref{sigma+}.
\end{lemma}
\begin{proof}
Let $c_{\ep}=c+i\ep$.
Due to the fact that
\beq\label{cI+}
\sigma_+(c_{\ep})I_+(c_{\ep})
=\sigma_+(c_{\ep})\int_0^1\frac{1}{\mathcal{H}(y,c_{\ep})}
\Big(\frac{1}{\va_+(y,c_{\ep})^2}-1\Big)dy
+\sigma_+(c_{\ep})\int_0^1\frac{1}{\mathcal{H}(y,c_{\ep})}dy,
\eeq
and by Proposition \ref{prop: sol. hom. [0,1]}, $\va_+(y,c_{\ep})$ is continuous for $(y,c_{\ep})\in [0,a_+]\times \Omega_{\ep_0}$ and for $\ep_0$ small enough,
\beqno
|\va_+(y,c_{\ep})|\geq \frac{1}{2}, \ \ |\va_+(y,c_{\ep})-1|\leq C|y-y_{c_+}|^2.
\eeqno
By using the fact that $y_{c_-}\leq 0\leq y_{c_+}$, we have
\beq\label{mono}
|y-y_{c_+}|\leq |y-y_{c_-}|,
\eeq
and for $c_{\ep}\in\Omega_{\ep_0}$,
\beq\label{1}
\Big|\frac{1}{\mathcal{H}(y,c_{\ep})}
\Big(\frac{1}{\va_+(y,c_{\ep})^2}-1\Big)\Big|\leq C\frac{|y-y_{c_+}|^2}{|y-y_{c_+}||y-y_{c_-}|}\leq C.
\eeq
Thus, by the Lebesgue's dominated convergence theorem, we have that as $\ep\rightarrow0$,
\beqno
\sigma_+(c_{\ep})\int_0^1\frac{1}{\mathcal{H}(y,c_{\ep})}
\Big(\frac{1}{\va_+(y,c_{\ep})^2}-1\Big)dy\rightarrow\sigma_+(c)
\int_0^1\frac{1}{\mathcal{H}(y,c)}
\Big(\frac{1}{\va_+(y,c)^2}-1\Big)dy.
\eeqno

On the other hand, we have
\beq\label{sigma I+}
\begin{split} &\sigma_+(c_{\ep})\int_0^1\frac{1}{\big(W_+(y)-c_{\ep}\big)\big(W_-(y)-c_{\ep}\big)}dy\\
&=
\int_0^1\frac{W'_+(y_{c_+})\big(W_-(y_{c_+})-W_-(y)\big)
-W'_-(y_{c_+})\big(W_+(y_{c_+})-W_+(y)\big)}
{\big(W_+(y)-c_{\ep}\big)\big(W_-(y)-c_{\ep}\big)}dy\\
&\quad+
\int_0^1\frac{W'_+(y_{c_+})-W'_+(y)}{W_+(y)-c_{\ep}}
-\frac{W'_-(y_{c_+})-W'_-(y)}{W_-(y)-c_{\ep}}dy\\
&\quad
+
\int_0^1\frac{W'_+(y)}{W_+(y)-c_{\ep}}-\frac{W'_-(y)}{W_-(y)-c_{\ep}}dy
=I_1(c_{\ep})+I_2(c_{\ep})+I_3(c_{\ep}).
\end{split}
\eeq
By denoting that $h(y,y_{c_+})=W'_+(y_{c_+})\big(W_-(y_{c_+})-W_-(y)\big)-W'_-(y_{c_+})\big(W_+(y_{c_+})-W_+(y)\big)$, we have $h(y_{c_+},y_{c_+})=0$
and $(\pa_yh)(y_{c_+},y_{c_+})=0$. Thus we obtain that
\beqno
h(y,y_{c_+})=(y-y_{c_+})^2\int_0^1\int_0^1\pa_{yy}h\big(y_{c_+}+ts(y-y_{c_+})\big)dtds,
\eeqno
and then
\beq\label{h bdd}
|h(y,y_{c_+})|\leq C|y-y_{c_+}|^2.
\eeq
Due to $|(W_+(y)-c_{\ep})(W_-(y)-c_{\ep})|\geq C|y-y_{c_+}||y-y_{c_-}|$, we get by (\ref{mono}) and (\ref{h bdd}),
\beq\label{A+1}
\Big|\frac{h(y,y_c)}
{(W_+(y)-c_{\ep})(W_-(y)-c_{\ep})}\Big|\leq \frac{C|y-y_{c_+}|^2}{|y-y_{c_+}||y-y_{c_-}|}\leq C,
\eeq
and
\beq\label{A+2}
\Big|\frac{W'_+(y_{c_+})-W'_+(y)}{W_+(y)-c_{\ep}}
-\frac{W'_-(y_{c_+})-W'_-(y)}{W_-(y)-c_{\ep}}\Big|
\leq \frac{C|y-y_{c_+}|}{|y-y_{c_+}|}
+\frac{C|y-y_{c_+}|}{|y-y_{c_-}|}\leq C.
\eeq
Therefore $|I_1(c_{\ep})|+|I_2(c_{\ep})|\leq C$ and by the Lebesgue's dominated convergence theorem, we have as $\ep$ tends to $0$,
\beqno
\begin{split}
I_1(c_{\ep})&\rightarrow
\int_0^1\frac{h(y,y_{c_+})}
{\big(W_+(y)-c\big)\big(W_-(y)-c\big)}dy,\\
I_2(c_{\ep})&\rightarrow
\int_0^1\frac{W'_+(y_{c_+})-W'_+(y)}{W_+(y)-c}
-\frac{W'_-(y_{c_+})-W'_-(y)}{W_-(y)-c}dy.
\end{split}
\eeqno
Due to $W_+(0)=W_-(0)=0$, we also have for $c_{\ep}\in \Omega_{\ep_0}\setminus D_0$,
\beq\label{eq: I3}
I_3(c_{\ep})=\ln\frac{W_+(1)-c_{\ep}}{W_-(1)-c_{\ep}}.
\eeq
And for $c_{\ep}=c+i\ep, c\in D_{\ep_0}\setminus D_0$, we have
\beq
I_3(c_{\ep})=
\frac{1}{2}\ln\frac{\big(W_+(1)-c\big)^2+\ep^2}{\big(W_-(1)-c\big)^2+\ep^2}
+i\arctan\frac{W_+(1)-c}{\ep}
-i\arctan\frac{W_-(1)-c}{\ep},
\eeq
and then
\begin{align*}
\lim_{\ep\rightarrow0+}I_3(c_{\ep})=\f12\ln\Big(\frac{W_+(1)-c}{c-W_-(1)}\Big)^2+i\pi\chi_+(c), \\
\lim_{\ep\rightarrow0-}I_3(c_{\ep})=
\f12\ln\Big(\frac{W_+(1)-c}{c-W_-(1)}\Big)^2-i\pi\chi_+(c).
\end{align*}

Thus, from (\ref{cI+}), we get
\beqno
\lim_{\ep\rightarrow 0^{\pm}}\sigma_+(c_{\ep})I_+(c_{\ep})
=\sigma_+(c)\Pi_+(c)+R_+^1(c)+R_+^2(c)+\f12\ln \Big(\frac{W_+(1)-c}{c-W_-(1)}\Big)^2\pm i\pi\chi_+(c).
\eeqno
This complete the proof of the lemma.
\end{proof}
\begin{lemma}\label{lem: I limit-}
For $c_{\ep}= c+i\ep \in D_{\ep_0}$ and $c\in D_0\setminus\big\{W_-(-1), W_+(-1)\big\}$. It holds that
\beq\label{I- lim}
\lim_{\ep\rightarrow 0{\pm}}\sigma_-(c_{\ep})I_-(c_{\ep})
=\sigma_-(c)I_-^{re}(c)\pm i\pi\chi_-(c),
\eeq
where $\sigma_-(c_{\ep})$ is defined by \eqref{sigma-}.
\end{lemma}
\begin{proof}
The proof of the lemma is same as the proof of Lemma \ref{lem: I limit+}. A direct calculation gives
\begin{align*}
&\sigma_-(c_{\ep})I_-(c_{\ep})\\
&=\sigma_-(c_{\ep})\int_{-1}^0\frac{1}{\mathcal{H}(y,c_{\ep})}
\left(\frac{1}{\va_-(y,c_{\ep})^2}-1\right)dy
+\sigma_-(c_{\ep})\int_{-1}^0\frac{1}{\mathcal{H}(y,c_{\ep})}dy\\
&=\sigma_-(c_{\ep})\int_{-1}^0\frac{1}{\mathcal{H}(y,c_{\ep})}
\left(\frac{1}{\va_-(y,c_{\ep})^2}-1\right)dy
+\int_{-1}^0\frac{g(y,y_{c_+})}
{\big(W_+(y)-c_{\ep}\big)\big(W_-(y)-c_{\ep}\big)}dy\\
&\quad+
\int_{-1}^0\frac{W'_+(y_{c_-})-W'_+(y)}{W_+(y)-c_{\ep}}
-\frac{W'_-(y_{c_-})-W_-{'}(y)}{W_-(y)-c_{\ep}}dy\\
&\quad
+\int_{-1}^0\frac{W'_+(y)}{W_+(y)-c_{\ep}}-\frac{W'_-(y)}{W_-(y)-c_{\ep}}dy\eqdef\sigma(c_{\ep})\Pi_-(c_{\ep})+J_1(c_{\ep})+J_2(c_{\ep})+J_3(c_{\ep}),
\end{align*}
with $g(y,y_{c_-})=W'_+(y_{c_-})\big[W_-(y_{c_-})-W_-(y)\big]
-W'_-(y_{c_-})\big[W_+(y_{c_-})-W_+(y)\big]$.

Proposition \ref{prop: sol. hom. [-1,0]} implies
\beno
\Big|\frac{1}{\mathcal{H}(y,c_{\ep})}
\Big(\frac{1}{\va_-(y,c_{\ep})^2}-1\Big)\Big|\leq C,
\eeno
and thus $|\Pi_-(c_{\ep})|\leq C$ and as $\ep\rightarrow0$,
\beqno
\sigma_-(c_{\ep})\Pi_-(c_{\ep})\rightarrow\sigma_-(c)
\int_{-1}^0\frac{1}{\mathcal{H}(y,c)}
\Big(\frac{1}{\va_-(y,c)^2}-1\Big)dy.
\eeqno
By the fact that
\beq\label{eq: g est}
|g(y,y_{c_+})|\leq C|y-y_{c_+}|^2,
\eeq
and then
\beno
\left|\frac{W'_+(y_{c_-})\big[W_-(y_{c_-})-W_-(y)\big]
-W'_-(y_{c_-})\big[W_+(y_{c_-})-W_+(y)\big]}
{\big(W_+(y)-c_{\ep}\big)\big(W_-(y)-c_{\ep}\big)}\right|\leq C,
\eeno
and
\beno
\Big|\frac{W'_+(y_{c_-})-W'_+(y)}{W_+(y)-c_{\ep}}
-\frac{W'_-(y_{c_-})-W'_-(y)}{W_-(y)-c_{\ep}}\Big|
\leq C,
\eeno
and the Lebesgue's dominated convergence theorem, we have $|J_1(c_{\ep})|+|J_2(c_{\ep})|\leq C$ and as $\ep$ tends to $0$,
\beqno
\begin{split}
J_1(c_{\ep})&\rightarrow
\int_{-1}^0\frac{W'_+(y_{c_-})\big[W_-(y_{c_-})-W_-(y)\big]
-W'_-(y_{c_-})\big[W_+(y_{c_-})-W_+(y)\big]}
{\big(W_+(y)-c\big)\big(W_-(y)-c\big)}dy,\\
J_2(c_{\ep})&\rightarrow
\int_{-1}^0\frac{W'_+(y_{c_-})-W'_+(y)}{W_+(y)-c}
-\frac{W'_-(y_{c_-})-W'_-(y)}{W_-(y)-c}dy.
\end{split}
\eeqno
Due to $W_+(0)=W_-(0)=0$, we also have for $c_{\ep}\in \Omega_{\ep_0}\setminus D_0$,
\beq\label{eq: J3}
J_3(c_{\ep})=\ln\frac{W_-(-1)-c_{\ep}}{W_+(-1)-c_{\ep}}.
\eeq
And for $c_{\ep}=c+i\ep, c\in D_{\ep_0}\setminus D_0$, we have
\beqno
J_3(c_{\ep})=
\frac{1}{2}\ln\frac{\big(W_-(-1)-c\big)^2+\ep^2}{\big(W_+(-1)-c\big)^2+\ep^2}
-i\arctan\frac{W_+(-1)-c}{\ep}
+i\arctan\frac{W_-(-1)-c}{\ep}.
\eeqno
And then, we get
\beno
&&\lim_{\ep\rightarrow0+}J_3(c_{\ep})=\f12\ln\Big(\frac{W_-(-1)-c}{c-W_+(-1)}\Big)^2+i\pi\chi_-(c),\\
&&\lim_{\ep\rightarrow0-}J_3(c_{\ep})=\f12
\ln\Big(\frac{W_-(-1)-c}{c-W_+(-1)}\Big)^2-i\pi\chi_-(c).
\eeno

Thus we complete the proof of lemma.
\end{proof}
\begin{lemma}\label{lem: I bdd}
There is $\ep_0>0$ such that for $c_{\ep}\in \Omega_{\ep_0}\setminus D_0$, there exists a  constant $C\geq1$ such that
\beno
\frac{l\big(W_+(\pm1)-c_{\ep}\big)l
\big(W_-(\pm1)-c_{\ep}\big)}{C|c_{\ep}|}\leq |I_{\pm}(c_{\ep})|
\leq \frac{Cl\big(W_+(\pm1)-c_{\ep}\big)l
\big(W_-(\pm1)-c_{\ep}\big)}{|c_{\ep}|}.
\eeno
\end{lemma}
\begin{proof}
We only give the estimate of $I_+(c_{\ep})$, the estimate of $I_-(c_{\ep})$ can be obtained by the same way and we omit the details here.

Due to (\ref{cI+}), we have
\beno
I_+(c_{\ep})=\Pi_+(c_{\ep})
+\frac{I_1(c_{\ep})+I_2(c_{\ep})+I_3(c_{\ep})}{\sigma_+(c_{\ep})}.
\eeno
The upper bound is directly follows from the fact that $|\Pi_+(c_{\ep})|+|I_1(c_{\ep})|+|I_2(c_{\ep})|\leq C$ and
\begin{align*}
|I_3(c_{\ep})|\leq
\left|\frac{1}{2}\ln\frac{\big(W_+(1)-c\big)^2+\ep^2}{\big(W_-(1)-c\big)^2+\ep^2}\right|+C
\leq Cl\big(W_+(1)-c_{\ep}\big)l
\big(W_-(1)-c_{\ep}\big).
\end{align*}

For the lower bounded. At first, we consider the case of $|c|<\delta_0$ for some $\delta_0>0$ small enough.
By Remark \ref{rmk: sigma bdd} and the fact that $|\Pi_+(c_{\ep})|\leq C$, we have
\begin{align*}
|I_+(c_{\ep})|&\geq \frac{\big|I_1(c_{\ep})+I_2(c_{\ep})+I_3(c_{\ep})\big|}{|\sigma_+(c_{\ep})|}
-|\Pi_+(c_{\ep})|\\
&\geq \frac{\big|Im(I_1(c_{\ep})+I_2(c_{\ep})+I_3(c_{\ep}))\big|}{C|c_{\ep}|}
-C\\
&\geq \frac{\big|Im(I_3(c_{\ep}))\big|}{C|c_{\ep}|}
-\frac{\big|Im(I_1(c_{\ep})+I_2(c_{\ep}))\big|}{C|c_{\ep}|}
-C.
\end{align*}

We have for $|c|<\delta_0$.
\beqno
\begin{split}
Im(I_1(c_{\ep}))&=\int_0^1\frac{\ep h(y,y_{c_+})(W_+(y)+W_-(y)-2c)}
{\big((W_+(y)-c)^2+\ep^2\big)\big((W_-(y)-c)^2+\ep^2\big)}dy,\\
Im(I_2(c_{\ep}))&=\int_0^1\frac{\ep\big(W'_+(y_{c_+})-W'_+(y)\big)}
{(W_+(y)-c)^2+\ep^2}dy-\int_0^1\frac{\ep\big(W'_-(y_{c_+})-W'_-(y)\big)}
{(W_-(y)-c)^2+\ep^2}dy.
\end{split}
\eeqno
By (\ref{h bdd}) and $|W_+(y)+W_-(y)-2c|\leq C|y-y_{c_+}|+C|y-y_{c_-}|\leq C|y-y_{c_-}|$, we have
\beq\label{I1 im}
|Im(I_1(c_{\ep}))|\leq C\ep\int_0^1\frac{y-y_{c_-}}{(y-y_{c_-})^2+\ep^2}dy\leq C\ep|\ln\ep|,
\eeq
and
\beq\label{I2 im}
\begin{split}
|Im(I_2(c_{\ep}))|&\leq C\ep\int_0^1\frac{|y-y_{c_+}|}{(y-y_{c_+})^2+\ep^2}dy
+C\ep\int_0^1\frac{|y-y_{c_+}|}{(y-y_{c_-})^2+\ep^2}dy\\
&\leq C\ep\int_0^1\frac{|y-y_{c_+}|}{(y-y_{c_+})^2+\ep^2}dy
+C\ep\int_0^1\frac{|y-y_{c_-}|}{(y-y_{c_-})^2+\ep^2}dy\\
&\leq C\ep|\ln\ep|.
\end{split}
\eeq

On the other hand, under the assumption $|c|<\delta_0$,
\beq\label{I3 im}
|Im(I_3(c_{\ep}))|=\Big|\arctan\frac{W_+(1)-c}{\ep}-\arctan\frac{W_-(1)-c}{\ep}\Big|
\geq\frac{3\pi}{4}.
\eeq

By choosing $\delta_0\leq \f{W_+(1)-W_-(1)}{1000}\ep_0$ small enough with $\ep_0$ in Proposition \ref{prop: sol. hom. [0,1]} also small enough, then (\ref{I1 im}), (\ref{I2 im}) and (\ref{I3 im}) imply for $|c|<\delta_0$,
\beq\label{eq: I+0 low}
|I_+(c_{\ep})|\geq \frac{C^{-1}}{|c_{\ep}|}-\frac{C\ep|\ln\ep|}{|c_{\ep}|}-C\geq\frac{C^{-1}}{|c_{\ep}|}.
\eeq

For the case of $|c_{\ep}-W_+(1)|<\delta_0$, we have
\beq
|\Pi_+(c_{\ep})|+\Big|\frac{I_1(c_{\ep})+I_2(c_{\ep})}{\sigma_+(c_{\ep})}\Big|\leq C.
\eeq
Thus for $|c_{\ep}-W_+(1)|<\delta_0$,  by \eqref{eq: I3},
\beno
\begin{split}
|I_+(c_{\ep})|&\geq \frac{|I_3(c_{\ep})|}{|\sigma_+(c_{\ep})|}
-\frac{|I_1(c_{\ep})+I_2(c_{\ep})|}{|\sigma_+(c_{\ep})|}
-|\Pi_+(c_{\ep})|\\
&\geq \frac{\Big|\ln\big|W_+(1)-c_{\ep}\big|\Big|}{C}-C\\
&\geq C^{-1}\big|\ln|W_+(1)-c_{\ep}|\big|
\end{split}
\eeno

Similarly, we have for the case of $\big|c_{\ep}-W_-(1)\big|<\delta_0$,
\beno
|I_+(c_{\ep})|\geq C^{-1}|\ln\big|W_-(1)-c_{\ep}|\big|
\eeno

For $c_{\ep}\in \Om_{\ep_0}\setminus D_0$ with $|c_{\ep}|\geq \d_0$, $\big|c_{\ep}-W_-(1)\big|\geq \delta_0$ and $\big|c_{\ep}-W_-(1)\big|\geq \delta_0$, Lemma \ref{lem: I limit+} implies $\sigma_+(c_{\ep})I_+(c_{\ep})$ is continuous to the boundary with its boundary value $\sigma_+(c)I_+^{re}\pm i\pi\chi_+(c)$. Let $c_{\ep}=c+i\ep$ with $c\in D_0$ and $|c-W_+(1)|\geq \delta_0$ and $|c-W_-(1)|\geq \delta_0$, then if $\chi_+(c)=0$($c<W_-(1)$ in Case 1, 2, 3, or $c>W_+(1)$ in Case 3, 8, 9,) then $\mathcal{H}(y,c)>0$ which implies $I_+^{re}(c)=\int_0^1\f{1}{\mathcal{H}(y,c)\va_+(y,c)^2}dy>0$ and if $\chi_+(c)\neq 0$, then $|\sigma_+(c)I_+^{re}\pm i\pi\chi_+(c)|\geq \pi$. 

Therefore there is $\ep_0$ such that for any $0\leq |\ep|\leq \ep_0$ and $c_{\ep}=c+i\ep$
\beno
|\sigma_+(c_{\ep})I_+(c_{\ep})|\geq \f{\pi}{2}.
\eeno
Thus we conclude that
\beno
|I_+(c_{\ep})|\geq \frac{l\big(W_+(1)-c_{\ep}\big)l
\big(W_-(1)-c_{\ep}\big)}{C|c_{\ep}|}.
\eeno
This completes the proof of the lemma.
\end{proof}

\begin{lemma}\label{lem: D>0}
Let $c\in  D_0\setminus \big\{0,W_+(1),W_-(1),W_+(-1),W_-(-1)\big\}$, $c_{\ep}=c+i\ep$. It holds
 \beqno
 \lim_{\ep\rightarrow0^{\pm}} \mathcal{D}(c_{\ep})=\mathcal{D}^{re}(c)\pm i\mathcal{D}^{im}(c).
 \eeqno
 Moreover, we have that there exists a constant $C\geq 1$ such that
 \beqno
 \mathcal{D}^{re}(c)^2+\mathcal{D}^{im}(c)^2\geq C^{-1} >0.
 \eeqno
\end{lemma}
\begin{proof}
By Lemma \ref{lem: I limit+} and Lemma \ref{lem: I limit-}, we get for $D_0\setminus \big\{0,W_+(1),W_-(1),W_+(-1),W_-(-1)\big\}$,
\beno
\lim_{\ep\rightarrow 0{\pm}}I_+(c)=I^{re}_+(c)\pm \frac{i\pi\chi_+(c)}{\sigma_+(c)},\quad
\lim_{\ep\rightarrow 0{\pm}}I_-(c)
=I^{re}_-(c)\pm \frac{i\pi\chi_-(c)}{\sigma_-(c)}.
\eeno
Thus for $c_{\ep}=c+i\ep$ and $c\in  D_0\setminus \big\{0,W_+(1),W_-(1),W_+(-1),W_-(-1)\big\}$, we get $P(c_{\ep})\to P(c)$ as $\ep\to 0$ and
\begin{align*}
&\lim_{\ep\rightarrow0\pm}\mathcal{D}(c_{\ep})\\
&=c^2P(c)\Big(I^{re}_+(c)\pm\frac{i\pi\chi_+(c)}{\sigma_+(c)}\Big)
\Big(I^{re}_-(c)\pm\frac{i\pi\chi_-(c)}{\sigma_-(c)}\Big)\\
&\quad-\va_+(0,c)^2\Big(I^{re}_+(c)\pm\frac{i\pi\chi_+(c)}{\sigma_+(c)}\Big)
-\va_-(0,c)^2\Big(I^{re}_-(c)\pm\frac{i\pi\chi_-(c)}{\sigma_-(c)}\Big)\\
&=c^2P(c)\Big(I^{re}_+(c)I^{re}_-(c)
-\frac{\pi^2\chi_+(c)\chi_-(c)}{\sigma_+(c)\sigma_-(c)}\Big)-\va_+(0,c)^2I^{re}_+(c)
-\va_-(0,c)^2I^{re}_-(c)\\
&\quad
\pm i\left(c^2P(c)\Big(\frac{\pi I^{re}_+(c)\chi_-(c)}{\sigma_-(c)}+\frac{\pi\ I^{re}_-(c)\chi_+(c)}
{\sigma_+(c)}\Big)-\frac{\pi\va_+(0,c)^2\chi_+(c)}{\sigma_+(c)}
-\frac{\pi\va_-(0,c)^2\chi_-(c)}{\sigma_-(c)}\right)\\
&=\mathcal{D}^{re}(c)\pm i\mathcal{D}^{im}(c),
\end{align*}
And then we have, for $c\in  D_0\setminus \big\{0,W_+(1),W_-(1),W_+(-1),W_-(-1)\big\}$ with $\chi_+(c)\chi_-(c)=1$, 
\begin{align*}
&\mathcal{D}^{re}(c)^2+\mathcal{D}^{im}(c)^2\\
&=c^4P(c)^2I^{re}_+(c)^2I^{re}_-(c)^2
+\frac{\pi^4c^4P(c)^2}{\sigma_+(c)^2\sigma_-(c)^2}
+\va_+(0,c)^4I^{re}_+(c)^2+\va_-(0,c)^4I^{re}_-(c)^2\\
&\quad
-2c^2P(c)\va_+(0,c)^2I^{re}_+(c)^2I^{re}_-(c)
-2c^2P(c)\va_-(0,c)^2I^{re}_+(c)I^{re}_-(c)^2\\
&\quad
+2\va_+(0,c)^2\va_-(0,c)^2I^{re}_+(c)I^{re}_-(c)
+\frac{\pi^2}{\sigma_-(c)^2}c^4P(c)^2I^{re}_+(c)^2\\
&\quad
+\frac{\pi^2}{\sigma_+(c)^2}c^4P(c)^2I^{re}_-(c)^2
+\frac{\pi^2}{\sigma_-(c)^2}\va_-(0,c)^4
+\frac{\pi^2}{\sigma_+(c)^2}\va_+(0,c)^4\\
&\quad
+\frac{2\pi^2\va_+(0,c)^2\va_-(0,c)^2}{\sigma_+(c)\sigma_-(c)}
-\frac{2\pi^2}{\sigma_-(c)^2}c^2P(c)\va_-(0,c)^2I^{re}_+(c)\\
&\quad
-\frac{2\pi^2}{\sigma_+(c)^2}c^2P(c)\va_+(0,c)^2I^{re}_-(c)\\
&=\Big[c^2P(c)I^{re}_+(c)I^{re}_-(c)
-\va_+(0,c)^2I^{re}_+(c)-\va_-(0,c)^2I^{re}_-(c)\Big]^2\\
&\quad +
\frac{\pi^2}{\sigma_-(c)^2}\Big[c^2P(c)I^{re}_+(c)-\va_-(0,c)^2\Big]^2
+\frac{\pi^2}{\sigma_+(c)^2}\Big[c^2P(c)I^{re}_-(c)-\va_+(0,c)^2\Big]^2\\
&\quad
+\frac{\pi^4c^4P(c)^2}{\sigma_+(c)^2\sigma_-(c)^2}
+\frac{2\pi^2\va_+(0,c)^2\va_-(0,c)^2}{\sigma_+(c)\sigma_-(c)}.
\end{align*}

On one hand, for $c\in  D_0\setminus \big\{0,W_+(1),W_-(1),W_+(-1),W_-(-1)\big\}$, we have 
$$\sigma_+(c)\sigma_-(c)\geq C^{-1}|c|^2>0.$$
 Indeed, by the fact that $W'_+(y)\geq C^{-1}>0, W'_-(y)\leq -C^{-1}<0$, we have
\beno
&&0>\sigma_+(c)=W'_+(y_{c_+})\big(W_-(y_{c_+})-W_-(y_{c_-})\big)\geq -C^{-1}(y_{c_+}-y_{c_-})\geq -C^{-1}|c|, \\
&&0>\sigma_-(c)=-W'_-(y_{c_+})\big(W_+(y_{c_-})-W_+(y_{c_+})\big)\geq -C^{-1}(y_{c_+}-y_{c_-})\geq -C^{-1}|c|,
 \eeno

On the other hand, by Proposition \ref{prop: sol. hom. [0,1]} and Proposition \ref{prop: sol. hom. [-1,0]}, we have
\beq\label{eq: P est}
\begin{split}
C|c|\geq |P(c)|&=|\va_-(0,c)^2(\va_+\pa_y\va_+)(0,c)-\va_+(0,c)^2(\va_-\pa_y\va_-)(0,c)|\\
&=|\va_-(0,c)^2(\va_+\pa_y\va_+)(0,c)|+|\va_+(0,c)^2(\va_-\pa_y\va_-)(0,c)|\\
&\geq C^{-1}(|y_{c_+}|+|y_{c_-}|)\geq C^{-1}|c|,
\end{split}
\eeq
and then
\beq
\mathcal{D}^{re}(c)^2+\mathcal{D}^{im}(c)^2\geq \frac{\pi^4c^4P(c)^2}{\sigma_+(c)^2\sigma_-(c)^2}
+\frac{2\pi^2\va_+(0,c)^2\va_-(0,c)^2}{\sigma_+(c)\sigma_-(c)}\geq C^{-1}\big(c^2+\frac{1}{c^2}\big)>C^{-1}.
\eeq

For $c\in  D_0\setminus \big\{0,W_+(1),W_-(1),W_+(-1),W_-(-1)\big\}$ with $\chi_+(c)=0$ then $\chi_-(c)=1$, we have
\begin{align*}
\mathcal{D}^{im}(c)=\f{\pi}{\sigma_-(c)}(I_+^{re}(c)c^2P(c)-\va_-(0,c)^2)
\end{align*}
In this case $c>W_+(1)$(in Case 3, 8, 9,) or $c<W_-(1)$(in Case 1, 2, 3), then $\mathcal{H}(y,c)>0$ for $y\in [0,1]$, thus $I_+^{re}(c)>0$ and by the fact that $P(c)<-C^{-1}|c|\leq 0$, we obtain that $|\mathcal{D}^{im}(c)|>C^{-1}$ for $c\in (W_+(1), W_-(-1)]$ and $c\in [W_+(-1), W_-(1))$. 

For $c\in  D_0\setminus \big\{0,W_+(1),W_-(1),W_+(-1),W_-(-1)\big\}$ with $\chi_-(c)=0$ then $\chi_+(c)=1$(in Case 1, 4, 5, 7, 9), we also have  for $c\in (W_-(-1),W_+(1)]$ and $c\in [W_-(1),W_+(-1))$
\begin{align*}
|\mathcal{D}^{im}(c)|=\Big|\f{\pi}{\sigma_+(c)}(I_-^{re}(c)c^2P(c)-\va_+(0,c)^2)\Big|\geq C^{-1}.
\end{align*}

Thus we complete the proof of the lemma.
\end{proof}



\begin{lemma}\label{lem: Re sigma}
Let $c_{\ep}\in \Omega_{\ep_0}\setminus\{0\}$. It holds that
\beno
\Big|\frac{Re(\sigma_+(c_{\ep}))}{Re(\sigma_-(c_{\ep}))}-1\Big|\leq C|c|.
\eeno
\end{lemma}
\begin{proof}
We have
\beq\label{Re sigma}
\begin{split}
&Re(\sigma_+(c_{\ep}))=W'_+(y_{c_+})\big(W_-(y_{c_+})-c\big)
-W'_-(y_{c_+})\big(W_+(y_{c_+})-c\big),\\
&Re(\sigma_-(c_{\ep}))=W'_+(y_{c_-})\big(W_-(y_{c_-})-c\big)
-W'_-(y_{c_-})\big(W_+(y_{c_-})-c\big),
\end{split}
\eeq
For the case $c\geq 0$, then $W_+(y_{c_+})=c=W_-(y_{c_-})$, we have
\begin{align*}
Re(\sigma_+(c_{\ep}))=W'_+(y_{c_+})\big(W_-(y_{c_+})-c\big),\ Re(\sigma_-(c_{\ep}))=-W'_-(y_{c_-})\big(W_+(y_{c_-})-c\big),
\end{align*}
and then due to $W_+(0)=W_-(0)=0$,
\begin{align*}
Re(\sigma_+(c_{\ep}))&=W'_+(y_{c_+})\big(W_-(y_{c_+})-W_+(y_{c_+})\big)\\
&=-W'_+(y_{c_+})y_{c_+}\lambda_+(y_{c_+}),\\
 Re(\sigma_-(c_{\ep}))&=-W'_-(y_{c_-})\big(W_+(y_{c_-})-W_-(y_{c_-})\big)\\
 &=-W'_-(y_{c_-})y_{c_-}\lambda_-(y_{c_-}),
\end{align*}
where
\begin{align*}
&\lambda_+(y_{c_+})=-\int_0^1\big[W'_-(sy_{c_+})-W'_+(sy_{c_+})\big]ds, \\
&\lambda_-(y_{c_-})=\int_0^1\big[W'_+(sy_{c_-})-W'_-(sy_{c_-})\big]ds,
\end{align*}
Thus we obtain $\lambda_+(y_{c_+})\geq C^{-1}$, $\lambda_-(y_{c_-})\geq C^{-1}$ and $|Re(\sigma_{\pm})(c_{\ep})|\geq C^{-1}|c|$.

We also have
\begin{align*}
\frac{Re(\sigma_+(c_{\ep}))}{Re(\sigma_-(c_{\ep}))}-1
&
=\frac{\lambda_+(y_{c_+})\big(W'_+(y_{c_+})y_{c_+}-W'_-(y_{c_-})y_{c_-}\big)}
{\lambda_-(y_{c_-})W'_-(y_{c_-})y_{c_-}}+\frac{\lambda_+(y_{c_+})-\lambda_-(y_{c_-})}{\lambda_-(y_{c_-})}.
\end{align*}
Due to $y_{c_-}=W_-^{-1}(c)=W_-^{-1}\big(W_+(y_{c_+})\big)$ and
\begin{align*}
&|W'_+(y_{c_+})y_{c_+}-W'_-(y_{c_-})y_{c_-}|\\
&=\left|\int_0^{y_{c_+}} sW''_+(s)
-(W_-^{-1}\circ W_+)(s)(W''_-\circ W_-^{-1}\circ W_+)(s)
\big((W_-^{-1})'\circ W_+\big)(s)W'_+(s)ds\right|\\
&\leq C \int_0^{y_{c_+}}  \left|s\right|
+\left|(W_-^{-1}\circ W_+)(s)\right|ds\leq C|c|^2,
\end{align*}
then we get
\begin{align*}
\Big|\frac{\lambda_+(y_{c_+})\big(W'_+(y_{c_+})y_{c_+}-W'_-(y_{c_-})y_{c_-}\big)}
{\lambda_-(y_{c_-})W'_-(y_{c_-})y_{c_-}}\Big|\leq C|c|.
\end{align*}
On the other hand,
\begin{align*}
&|\lambda_+(y_{c_+})-\lambda_-(y_{c_-})|\\
&=\left|\int_0^1\big[W'_-(sy_{c_+})-W'_-(sy_{c_-})
-W'_+(sy_{c_+})+W'_+(sy_{c_-})\big]ds\right|\\
&=|(y_{c_+}-y_{c_-})|\left|\int_0^1s\int_0^1\big[\big(W''_--W''_+\big)
\big(sy_{c_-}+ts(y_{c_+}-y_{c_-})\big)\big]dtds\right|\leq C|c|.
\end{align*}

For $c\leq 0$, then $W_-(y_{c_+})=c=W_+(y_{c_-})$, we have
\begin{align*}
Re(\sigma_+(c_{\ep}))&=-W'_-(y_{c_+})\big(W_+(y_{c_+})-W_-(y_{c_+})\big)
=-W'_-(y_{c_+})y_{c_+}\wt{\lambda}_+(y_{c_+}),\\
 Re(\sigma_-(c_{\ep}))&=W'_+(y_{c_-})\big(W_-(y_{c_-})-W_+(y_{c_-})\big)
 =-W'_+(y_{c_-})y_{c_-}\wt{\lambda}_-(y_{c_-}),
\end{align*}
where
\begin{align*}
&\wt{\lambda}_+(y_{c_+})=\int_0^1\big[W'_+(sy_{c_+})-W'_-(sy_{c_+})\big]ds, \\
&\wt{\lambda}_-(y_{c_-})=-\int_0^1\big[W'_-(sy_{c_-})-W'_+(sy_{c_-})\big]ds,
\end{align*}
Thus we obtain $\wt{\lambda}_+(y_{c_+})\geq C^{-1}$, $\wt{\lambda}_-(y_{c_-})\geq C^{-1}$ and $|Re(\sigma_{\pm}(c_{\ep}))|\geq C^{-1}|c|$.

We also have
\begin{align*}
\frac{Re(\sigma_+(c_{\ep}))}{Re(\sigma_-(c_{\ep}))}-1
&
=\frac{\wt{\lambda}_+(y_{c_+})\big(W'_-(y_{c_+})y_{c_+}-W'_+(y_{c_-})y_{c_-}\big)}
{\wt{\lambda}_-(y_{c_-})W'_+(y_{c_-})y_{c_-}}
+\frac{\wt{\lambda}_+(y_{c_+})-\wt{\lambda}_-(y_{c_-})}{\wt{\lambda}_-(y_{c_-})}.
\end{align*}
Due to $y_{c_-}=W_+^{-1}(c)=W_+^{-1}\big(W_-(y_{c_+})\big)$ and
\beqno
|W'_-(y_{c_+})y_{c_+}-W'_+(y_{c_-})y_{c_-}|\leq C|c|^2,
\eeqno
then we get
\beqno
\Big|\frac{\wt{\lambda}_+(y_{c_+})\big(W'_-(y_{c_+})y_{c_+}-W'_+(y_{c_-})y_{c_-}\big)}
{\wt{\lambda}_-(y_{c_-})W'_+(y_{c_-})y_{c_-}}\Big|\leq C|c|
\eeqno
and
\beqno
|\wt{\lambda}_+(y_{c_+})-\wt{\lambda}_-(y_{c_-})|\leq C|c|.
\eeqno

Therefore,
\begin{align*}
\Big|\frac{Re(\sigma_+(c_{\ep}))}{Re(\sigma_-(c_{\ep}))}-1\Big|\leq C|c|.
\end{align*}
This completes the proof of the Lemma.
\end{proof}

Now, we present the proof of Proposition \ref{prop: D est}.
\begin{proof}
The first part of the Proposition follows directly from Lemma \ref{lem: D>0}.

For the case of $|c|<\delta_0$, we have
\begin{align}
\mathcal{D}(c_{\ep})&=c_{\ep}^2P(c_{\ep})I_+(c_{\ep})I_-(c_{\ep})
-\big(\va_+(0,c_{\ep})^2-1\big)I_+(c_{\ep})-\Pi_+(c_{\ep})\nonumber\\
&\ \
-\big(\va_-(0,c_{\ep})^2-1\big)I_-(c_{\ep})
-\Pi_-(c_{\ep})\nonumber\\
&\ \
-\frac{I_1(c_{\ep})+I_2(c_{\ep})+I_3(c_{\ep})}{\sigma_+(c_{\ep})}
-\frac{J_1(c_{\ep})+J_2(c_{\ep})+J_3(c_{\ep})}{\sigma_-(c_{\ep})}\nonumber\\
&=c_{\ep}^2P(c_{\ep})I_+(c_{\ep})I_-(c_{\ep})
-\big(\va_+(0,c_{\ep})^2-1\big)I_+(c_{\ep})-\Pi_+(c_{\ep})\nonumber\\
&\ \
-\big(\va_-(0,c_{\ep})^2-1\big)I_-(c_{\ep})
-\Pi_-(c_{\ep})\nonumber\\
&\ \ -\frac{I_1(c_{\ep})+I_2(c_{\ep})+I_3(c_{\ep})}{\sigma_+(c_{\ep})}
-\frac{J_1(c_{\ep})+J_2(c_{\ep})+J_3(c_{\ep})}{\sigma_+(c_{\ep})}\nonumber\\
&\ \
-\frac{\frac{\sigma_+(c_{\ep})}{\sigma_-(c_{\ep})}-1}{\sigma_+(c_{\ep})}
\big(J_1(c_{\ep})+J_2(c_{\ep})+J_3(c_{\ep})\big),
\end{align}
where we recall in the proof of Lemma \ref{lem: I limit+} and Lemma \ref{lem: I limit-},
\beno
\sigma_+(c_{\ep})I_+(c_{\ep})=\sigma_+(c_{\ep})\Pi_+(c_{\ep})+I_1(c_{\ep})+I_2(c_{\ep})+I_3(c_{\ep}),\\
\sigma_-(c_{\ep})I_-(c_{\ep})=\sigma_-(c_{\ep})\Pi_-(c_{\ep})+J_1(c_{\ep})+J_2(c_{\ep})+J_3(c_{\ep}).
\eeno
By Proposition \ref{prop: sol. hom. [0,1]}, Proposition \ref{prop: sol. hom. [-1,0]},  Lemma \ref{lem: I bdd} and \eqref{eq: P est}, we obtain
\begin{align}\label{eq: up bdd1}
&\Big|c_{\ep}^2P(c_{\ep})I_+(c_{\ep})I_-(c_{\ep})-\big(\va_+(0,c_{\ep})^2-1\big)I_+(c_{\ep})
-\big(\va_-(0,c_{\ep})^2-1\big)I_-(c_{\ep})\Big|\nonumber\\
&\leq |c_{\ep}|^2|P(c_{\ep})||I_+(c_{\ep})||I_-(c_{\ep})|+C\big|\va_+(0,c_{\ep})-1\big|
|I_+(c_{\ep})|+C\big|\va_-(0,c_{\ep})-1\big||I_-(c_{\ep})|\nonumber\\
& \leq C|c_{\ep}|.
\end{align}
We also have
\beqno
Im(I_1(c_{\ep}))=\int_0^1\frac{\ep h(y,y_{c_+})\big(W_+(y)+W_-(y)-2c\big)}
{\big((W_+(y)-c)^2+\ep^2\big)\big((W_-(y)-c)^2+\ep^2\big)}dy,
\eeqno
by (\ref{h bdd}) and $|W_+(y)+W_-(y)-2c|\leq C|y-y_{c_-}|$, we have
\beqno
\Big|\frac{h(y,y_{c_+})\big(W_+(y)+W_-(y)-2c\big)}
{\big((W_+(y)-c)^2+\ep^2\big)\big((W_-(y)-c)^2+\ep^2\big)}\Big|\leq C\frac{|y-y_{c_-}|}{|(y-y_{c_-})^2+\ep^2|}.
\eeqno
Thus we can get that for $|c|<\delta_0$ with $\delta_0$ small enough,
\beqno
|Im(I_1(c_{\ep}))|\leq C\ep|\ln \ep|.
\eeqno
And similarly, we have
\beqno
Im(I_2(c_{\ep}))=\ep\int_0^1\frac{W'_+(y_{c_+})-W'_+(y)}{(W_+(y)-c)^2+\ep^2}dy
-\ep\int_0^1\frac{W'_-(y_{c_+})-W'_-(y)}{(W_-(y)-c)^2+\ep^2}dy,
\eeqno
which implies
\beqno
\big|Im(I_2(c_{\ep}))\big|\leq C\ep|\ln\ep|.
\eeqno
By the same argument, we can deduce that for $|c|<\delta_0$,
\beqno
\big|Im(J_1(c_{\ep})+J_2(c_{\ep}))\big|\leq C\ep|\ln\ep|.
\eeqno
Thus from the above, we have
\beq\label{eq: IJ12 bdd}
\big|Im(I_1(c_{\ep})+I_2(c_{\ep})+J_1(c_{\ep})+J_2(c_{\ep}))\big|\leq C\ep|\ln\ep|.
\eeq
Due to $
Im(I_3(c_{\ep})+J_3(c_{\ep}))=2\arctan\frac{W_+(1)-c}{\ep}-2\arctan\frac{W_-(1)-c}{\ep},
$
we get for $|c|\leq \delta_0$,
\beq\label{eq: IJ3 low bdd}
\big|Im(I_3(c_{\ep})+J_3(c_{\ep}))\big|\geq\frac{3\pi}{2}.
\eeq

 On the other hand, we have $C^{-1}|c|\leq \big|Re(\sigma_-(c_{\ep}))\big|\leq C|c|$ and
\beno
Im(\sigma_+(c_{\ep}))=-\ep \big(W'_+(y_{c_+})-W'_-(y_{c_+})\big),\\
Im(\sigma_-(c_{\ep}))=-\ep \big(W'_+(y_{c_-})-W'_-(y_{c_-})\big),
\eeno
and then
\beqno
C^{-1}\ep\leq \big|Im(\sigma_-(c_{\ep}))\big|\leq C\ep,
\eeqno
and
\begin{align*}
\big|Im(\sigma_+(c_{\ep}))-Im(\sigma_-(c_{\ep}))\big|
&=\ep\left|W'_+(y_{c_+})-W'_+(y_{c_-})-W'_-(y_{c_+})+W'_-(y_{c_-})\right|
\leq C\ep|c|.
\end{align*}
Then from which and by Lemma \ref{lem: Re sigma}, we get
\begin{align}\label{eq: sig bdd}
\Big|\frac{\sigma_+(c_{\ep})}{\sigma_-(c_{\ep})}-1\Big|
&=\left|\frac{Re(\sigma_+(c_{\ep}))-Re(\sigma_-(c_{\ep}))
+i\big(Im(\sigma_+(c_{\ep}))-Im(\sigma_-(c_{\ep}))\big)}
{Re(\sigma_-(c_{\ep}))+iIm(\sigma_-(c_{\ep}))}\right|\nonumber\\
&
\leq\Big|\frac{Re(\sigma_+(c_{\ep}))}
{Re(\sigma_-(c_{\ep}))}-1\Big|
+\Big|\frac{Im(\sigma_+(c_{\ep}))-Im(\sigma_-(c_{\ep}))}
{Re(\sigma_-(c_{\ep}))+iIm(\sigma_-(c_{\ep}))}\Big|\nonumber\\
&\leq C|c|+\frac{C\ep |c|}{C^{-1}(\ep+|c|)}
\leq C|c|.
\end{align}

Therefore by the fact that $|\Pi_{\pm}(c_{\ep})|+|J_1(c_{\ep})|+|J_2(c_{\ep})|\leq C$ and that for $|c|\leq \delta_0$, $|J_3(c_{\ep})|\leq C$ and \eqref{eq: P est}, \eqref{eq: up bdd1}, \eqref{eq: IJ12 bdd}, \eqref{eq: IJ3 low bdd} and \eqref{eq: sig bdd} we obtain for $|c|<\delta_0$
\begin{align}\label{D 0}
|\mathcal{D}(c_{\ep})|&\geq \frac{|I_1(c_{\ep})+I_2(c_{\ep})+I_3(c_{\ep})+J_1(c_{\ep})+J_2(c_{\ep})+J_3(c_{\ep})|}{|\sigma_+(c_{\ep})|}
-\big|\Pi_+(c_{\ep})\big|-\big|\Pi_-(c_{\ep})\big|\nonumber\\
&\ \
-\Big|c_{\ep}^2P(c_{\ep})I_+(c_{\ep})I_-(c_{\ep})
-\big(\va_+(0,c_{\ep})^2-1\big)I_+(c_{\ep})
-\big(\va_-(0,c_{\ep})^2-1\big)I_-(c_{\ep})\Big|\nonumber\\
&\ \ -\Big|\frac{\frac{\sigma_+(c_{\ep})}{\sigma_-(c_{\ep})}-1}{\sigma_+(c_{\ep})}
\big(J_1(c_{\ep})+J_2(c_{\ep})+J_3(c_{\ep})\big)\Big|\nonumber\\
&\geq \frac{\big|Im(I_3(c_{\ep})+J_3(c_{\ep}))\big|}{\big|\sigma_+(c_{\ep})\big|}
-\frac{\big|Im(I_1(c_{\ep})+I_2(c_{\ep})+J_1(c_{\ep})+J_2(c_{\ep}))\big|}
{\big|\sigma_+(c_{\ep})\big|}-C-C|c_{\ep}|\nonumber\\
&\ \  -\frac{\big|\frac{\sigma_1(c_{\ep})}{\sigma_2(c_{\ep})}-1\big|}
{\big|\sigma_+(c_{\ep})\big|}\Big|J_1(c_{\ep})+J_2(c_{\ep})+J_3(c_{\ep})\Big|\nonumber\\
&\geq \frac{C^{-1}}{|c_{\ep}|}-\frac{C\ep|\ln\ep|}{|c_{\ep}|}-C-C|c_{\ep}|
\geq\frac{C^{-1}}{|c_{\ep}|}.
\end{align}

\no{\bf Case 1. $W_+(1)\neq W_-(-1)$ and $W_+(-1)\neq W_-(1)$. }\\
For $|c_{\ep}-W_{\pm}(1)|\leq \delta_0$, by Lemma \ref{lem: I limit+}, we get
\begin{align*}
\Big|\frac{c_{\ep}^2P(c_{\ep})I_-(c_{\ep})\big(I_1(c_{\ep})+I_2(c_{\ep})\big)}{\sigma_+(c_{\ep})}\Big|+\Big|\frac{\va_+(0,c_{\ep})^2\big(I_1(c_{\ep})+I_2(c_{\ep})\big)}
{\sigma_+(c_{\ep})}\Big|+\Big|\va_-(0,c_{\ep})^2I_-(c_{\ep})\Big|\leq C,
\end{align*}
and thus, we have
\begin{align*}
|\mathcal{D}(c_{\ep})|&\geq \left|Im\, \left(\frac{c_{\ep}^2P(c_{\ep})I_-(c_{\ep})I_3(c_{\ep})}
{\sigma_+(c_{\ep})}-\frac{\va_+(0,c_{\ep})^2I_3(c_{\ep})}{\sigma_+(c_{\ep})}\right)\right|\\
&\quad
-\Big|\frac{c_{\ep}^2P(c_{\ep})I_-(c_{\ep})\big(I_1(c_{\ep})+I_2(c_{\ep})\big)}{\sigma_+(c_{\ep})}\Big|-\Big|\frac{\va_+(0,c_{\ep})^2\big(I_1(c_{\ep})+I_2(c_{\ep})\big)}
{\sigma_+(c_{\ep})}\Big|-\Big|\va_-(0,c_{\ep})^2I_-(c_{\ep})\Big|\\
&\geq \left|Im\,\Big(\frac{c_{\ep}^2P(c_{\ep})I_-(c_{\ep})I_3(c_{\ep})}
{\sigma_+(c_{\ep})}\Big)
-Im\,\Big(\frac{\va_+(0,c_{\ep})^2I_3(c_{\ep})}{\sigma_+(c_{\ep})}\Big)\right|-C.
\end{align*}
Due to 
\begin{align*}
&Im\Big(\frac{c_{\ep}^2P(c_{\ep})I_-(c_{\ep})I_3(c_{\ep})}
{\sigma_+(c_{\ep})}\Big)
-Im\Big(\frac{\va_+(0,c_{\ep})^2I_3(c_{\ep})}{\sigma_+(c_{\ep})}\Big)\\
&=Im\Big(\frac{c_{\ep}^2P(c_{\ep})I_-(c_{\ep})}
{\sigma_+(c_{\ep})}\Big)Re(I_3(c_{\ep}))+Re\Big(\frac{c_{\ep}^2P(c_{\ep})I_-(c_{\ep})}
{\sigma_+(c_{\ep})}\Big)Im(I_3(c_{\ep}))\\
&\ \ 
-Im\Big(\frac{\va_+(0,c_{\ep})^2}{\sigma_+(c_{\ep})}\Big)Re(I_3(c_{\ep}))
-Re\Big(\frac{\va_+(0,c_{\ep})^2}{\sigma_+(c_{\ep})}\Big)Im(I_3(c_{\ep}))\\
&=Re\Big(\frac{c_{\ep}^2P(c_{\ep})}{\sigma_+(c_{\ep})}\Big)Im(I_-(c_{\ep}))Re(I_3(c_{\ep}))
+Im\Big(\frac{c_{\ep}^2P(c_{\ep})}{\sigma_+(c_{\ep})}\Big)Re(I_-(c_{\ep}))Re(I_3(c_{\ep}))\\
&\ \ 
-Im\Big(\frac{\va_+(0,c_{\ep})^2}{\sigma_+(c_{\ep})}\Big)Re(I_3(c_{\ep}))
+Re\Big(\frac{c_{\ep}^2P(c_{\ep})I_-(c_{\ep})}
{\sigma_+(c_{\ep})}\Big)Im(I_3(c_{\ep}))\\
&\ \
-Re\Big(\frac{\va_+(0,c_{\ep})^2}{\sigma_+(c_{\ep})}\Big)Im(I_3(c_{\ep}))\\
&=\frac{c^2P(c)}{\sigma_+(c)}Im(I_-(c_{\ep}))Re(I_3(c_{\ep}))
-\frac{\va_+(0,c)^2}{\sigma_+(c)}Im(I_3(c_{\ep}))\\
&\ \
+Re\Big(\frac{c_{\ep}^2P(c_{\ep})}{\sigma_+(c_{\ep})}
-\frac{c^2P(c)}{\sigma_+(c)}\Big)Im(I_-(c_{\ep}))Re(I_3(c_{\ep}))
+Im\Big(\frac{c_{\ep}^2P(c_{\ep})}{\sigma_+(c_{\ep})}\Big)Re(I_-(c_{\ep}))Re(I_3(c_{\ep}))\\
&\ \
-Im\Big(\frac{\va_+(0,c_{\ep})^2}{\sigma_+(c_{\ep})}\Big)Re(I_3(c_{\ep}))
+Re\Big(\frac{c_{\ep}^2P(c_{\ep})I_-(c_{\ep})}
{\sigma_+(c_{\ep})}\Big)Im(I_3(c_{\ep}))\\
&\ \
-Re\Big(\frac{\va_+(0,c_{\ep})^2}{\sigma_+(c_{\ep})}
-\frac{\va_+(0,c)^2}{\sigma_+(c)}\Big)Im(I_3(c_{\ep})),
\end{align*}
and by Remark \ref{rmk: epsilon},
\begin{align*}
&\Big|Re\Big(\frac{c_{\ep}^2P(c_{\ep})}{\sigma_+(c_{\ep})}
-\frac{c^2P(c)}{\sigma_+(c)}\Big)\Big|\leq C|\ep|,\\
&\Big|Re\Big(\frac{\va_+(0,c_{\ep})^2}{\sigma_+(c_{\ep})}
-\frac{\va_+(0,c)^2}{\sigma_+(c)}\Big)\Big|\leq C|\ep|,\\
&\Big|Im\Big(\frac{c_{\ep}^2P(c_{\ep})}{\sigma_+(c_{\ep})}\Big)\Big|
=\Big|Im\Big(\frac{c_{\ep}^2P(c_{\ep})}{\sigma_+(c_{\ep})}
-\frac{c^2P(c)}{\sigma_+(c)}\Big)\Big|\leq C|\ep|,\\
&\Big|Im\Big(\frac{\va_+(0,c_{\ep})^2}{\sigma_+(c_{\ep})}\Big)\Big|
=\Big|Im\Big(\frac{\va_+(0,c_{\ep})^2}{\sigma_+(c_{\ep})}
-\frac{\va_+(0,c)^2}{\sigma_+(c)}\Big)\Big|\leq C|\ep|.
\end{align*}
Thus by taking $\ep_0$ small enough, we have for any $|\ep|\leq \ep_0$, $|Im(I_-(c_{\ep}))|\geq |Im(J_3(c_{\ep}))|\geq\frac{3\pi}{4}$, $\frac{c^2P(c)}{\sigma_+(c)}\geq C^{-1}$, $\frac{\va_+(0,c)^2}{\sigma_+(c)}\geq C^{-1}$.
Thus by taking $\delta_0$ small enough, we have for $|c_{\ep}-W_{\pm}(1)|\leq \delta_0$,
\begin{align*}
|\mathcal{D}(c_{\ep})|
&\geq\Big|\frac{c^2P(c)}{\sigma_+(c)}Im(I_-(c_{\ep}))Re(I_3(c_{\ep}))\Big|
-\Big|\frac{\va_+(0,c)^2}{\sigma_+(c)}Im(I_3(c_{\ep}))\Big|\\
&\ \ 
-\Big|Re\Big(\frac{c_{\ep}^2P(c_{\ep})}{\sigma_+(c_{\ep})}
-\frac{c^2P(c)}{\sigma_+(c)}\Big)Im(I_-(c_{\ep}))Re(I_3(c_{\ep}))\Big|\\
&\ \ 
-\Big|Im\Big(\frac{c_{\ep}^2P(c_{\ep})}{\sigma_+(c_{\ep})}\Big)
Re(I_-(c_{\ep}))Re(I_3(c_{\ep}))\Big|
-\Big|Im\Big(\frac{\va_+(0,c_{\ep})^2}{\sigma_+(c_{\ep})}\Big)Re(I_3(c_{\ep}))\Big|\\
&\ \ -\Big|Re\Big(\frac{c_{\ep}^2P(c_{\ep})I_-(c_{\ep})}
{\sigma_+(c_{\ep})}\Big)Im(I_3(c_{\ep}))\Big|
+\Big|
-Re\Big(\frac{\va_+(0,c_{\ep})^2}{\sigma_+(c_{\ep})}
-\frac{\va_+(0,c)^2}{\sigma_+(c)}\Big)Im(I_3(c_{\ep}))\Big|\\
&\geq C^{-1}|Re(I_3(c_{\ep}))|-C\ep|Re(I_3(c_{\ep}))|-C\\
&\geq C^{-1} l(c_{\ep}-W_{\pm}(1)).
\end{align*}
For $|c_{\ep}-W_{\pm}(-1)|\leq \delta_0$, by Lemma \ref{lem: I limit-}, we get
\begin{align*}
\Big|\frac{c_{\ep}^2P(c_{\ep})I_+(c_{\ep})\big(J_1(c_{\ep})+J_2(c_{\ep})\big)}
{\sigma_-(c_{\ep})}\Big|+\Big|\frac{\va_-(0,c_{\ep})^2\big(J_1(c_{\ep})+J_2(c_{\ep})\big)}
{\sigma_-(c_{\ep})}\Big|+\Big|\va_+(0,c_{\ep})^2I_+(c_{\ep})\Big|\leq C.
\end{align*}
By the same method, we have
\begin{align*}
|\mathcal{D}(c_{\ep})|&\geq\Big|Im\Big(\frac{c_{\ep}^2P(c_{\ep})I_+(c_{\ep})J_3(c_{\ep})}
{\sigma_-(c_{\ep})}\Big)
-Im\Big(\frac{\va_-(0,c_{\ep})^2J_3(c_{\ep})}{\sigma_-(c_{\ep})}\Big)\Big|\\
&\quad
-\Big|\frac{c_{\ep}^2P(c_{\ep})I_+(c_{\ep})\big(J_1(c_{\ep})+J_2(c_{\ep})\big)}
{\sigma_-(c_{\ep})}-\frac{\va_-(0,c_{\ep})^2\big(J_1(c_{\ep})+J_2(c_{\ep})\big)}
{\sigma_-(c_{\ep})}-\va_+(0,c_{\ep})^2I_+(c_{\ep})\Big|\\
&\geq \Big|\frac{c^2P(c)}{\sigma_-(c)}Im(I_+(c_{\ep}))Re(J_3(c_{\ep}))\Big|
-\Big|\frac{\va_-(0,c)^2}{\sigma_-(c)}Im(J_3(c_{\ep}))\Big|\\
&\quad
-\Big|Re\Big(\frac{c_{\ep}^2P(c_{\ep})}{\sigma_-(c_{\ep})}
-\frac{c^2P(c)}{\sigma_-(c)}\Big)Im(I_+(c_{\ep}))Re(J_3(c_{\ep}))\Big|\\
&\quad
-\Big|Im\Big(\frac{c_{\ep}^2P(c_{\ep})}{\sigma_-(c_{\ep})}\Big)
Re(I_+(c_{\ep}))Re(J_3(c_{\ep}))\Big|
-\Big|Im\Big(\frac{\va_-(0,c_{\ep})^2}{\sigma_-(c_{\ep})}\Big)Re(J_3(c_{\ep}))\Big|\\
&\quad -\Big|Re\Big(\frac{c_{\ep}^2P(c_{\ep})I_+(c_{\ep})}
{\sigma_-(c_{\ep})}\Big)Im(J_3(c_{\ep}))\Big|
-\Big|
Re\Big(\frac{\va_-(0,c_{\ep})^2}{\sigma_-(c_{\ep})}
-\frac{\va_-(0,c)^2}{\sigma_-(c)}\Big)Im(J_3(c_{\ep}))\Big|-C\\
&\geq C^{-1}|Re(J_3(c_{\ep}))|-C\ep|Re(J_3(c_{\ep}))|-C\ep\\
&\geq C^{-1} l(c_{\ep}-W_{\pm}(-1)).
\end{align*}

\no{\bf{Case 2. $W_+(1)=W_-(-1)$ and $W_+(-1)\neq W_-(1)$.}}\\ 
For $|c_{\ep}-W_+(1)|=|c_{\ep}-W_-(-1)|< \delta_0$, by Lemma \ref{lem: I bdd} and \eqref{eq: P est}, we get
\begin{align*}
\big|\mathcal{D}(c_{\ep})\big|
&=\Big|c_{\ep}^2P(c_{\ep})I_+(c_{\ep})I_-(c_{\ep})-\va_+(0,c_{\ep})^2I_+(c_{\ep})
-\va_-(0,c_{\ep})^2I_-(c_{\ep})\Big|\nonumber\\
&\geq |c_{\ep}|^2|P(c_{\ep})||I_+(c_{\ep})||I_-(c_{\ep})|-C
|I_+(c_{\ep})|-C|I_-(c_{\ep})|\nonumber\\
& \geq C^{-1}l(c_{\ep}-W_+(1))l(c_{\ep}-W_-(-1))-Cl(c_{\ep}-W_+(1))
-Cl(c_{\ep}-W_-(-1))\\
&\geq C^{-1}l(c_{\ep}-W_+(1))l(c_{\ep}-W_-(-1)).
\end{align*}
And for $|c_{\ep}-W_-(1)|<\delta_0$, we can obtain by the same argument as in Case 1 that
\begin{align*}
\big|\mathcal{D}(c_{\ep})\big|\geq C^{-1}l(c_{\ep}-W_-(1))
\end{align*}
and for $|c_{\ep}-W_+(-1)|<\delta_0$,
\begin{align*}
\big|\mathcal{D}(c_{\ep})\big|\geq C^{-1}l(c_{\ep}-W_+(-1)).
\end{align*}

\no{\bf{Case 3. $W_+(1)\neq W_-(-1)$ and $W_+(-1)=W_-(1)$.}}\\
The proof is similar to Case 2 and we have for $|c_{\ep}-W_+(-1)|=|c_{\ep}-W_-(1)|< \delta_0$,
\begin{align*}
\big|\mathcal{D}(c_{\ep})\big|
&=\Big|c_{\ep}^2P(c_{\ep})I_+(c_{\ep})I_-(c_{\ep})-\va_+(0,c_{\ep})^2I_+(c_{\ep})
-\va_-(0,c_{\ep})^2I_-(c_{\ep})\Big|\nonumber\\
&\geq |c_{\ep}|^2|P(c_{\ep})||I_+(c_{\ep})||I_-(c_{\ep})|-C
|I_+(c_{\ep})|-C|I_-(c_{\ep})|\nonumber\\
& \geq C^{-1}l(c_{\ep}-W_+(-1))l(c_{\ep}-W_-(1))-Cl(c_{\ep}-W_+(-1))
-Cl(c_{\ep}-W_-(1))\\
&\geq C^{-1}l(c_{\ep}-W_+(-1))l(c_{\ep}-W_-(1)).
\end{align*}
And we have, for $|c_{\ep}-W_+(1)|<\delta_0$,
\begin{align*}
\big|\mathcal{D}(c_{\ep})\big|\geq C^{-1}l(c_{\ep}-W_+(1))
\end{align*}
and for $|c_{\ep}-W_-(-1)|<\delta_0$,
\begin{align*}
\big|\mathcal{D}(c_{\ep})\big|\geq C^{-1}l(c_{\ep}-W_-(-1)).
\end{align*} 

\no{\bf Case 4. $W_+(1)=W_-(-1)$ and $W_-(1)=W_+(-1)$.} \\
For the case $|c_{\ep}-W_+(1)|<\delta_0$, by Lemma \ref{lem: I bdd} and \eqref{eq: P est}, we get
\begin{align}
\big|\mathcal{D}(c_{\ep})\big|
&=\Big|c_{\ep}^2P(c_{\ep})I_+(c_{\ep})I_-(c_{\ep})-\va_+(0,c_{\ep})^2I_+(c_{\ep})
-\va_-(0,c_{\ep})^2I_-(c_{\ep})\Big|\nonumber\\
&\geq |c_{\ep}|^2|P(c_{\ep})||I_+(c_{\ep})||I_-(c_{\ep})|-C
|I_+(c_{\ep})|-C|I_-(c_{\ep})|\nonumber\\
& \geq C^{-1}\Big(1+\Big|\ln\big|W_+(1)-c_{\ep}\big|\Big|\Big)^2
-C\Big(1+\Big|\ln\big|W_+(1)-c_{\ep}\big|\Big|\Big)\nonumber\\
&\geq C^{-1}\Big(1+\Big|\ln\big|W_+(1)-c_{\ep}\big|\Big|\Big)^2.
\end{align}

Similarly, we can deduce that for  the case of $|c_{\ep}-W_-(1)|<\delta_0$,
\beqno
|\mathcal{D}(c_{\ep})|\geq C^{-1}\Big(1+\Big|\ln\big|W_-(1)-c_{\ep}\big|\Big|\Big)^2.
\eeqno

For the case $c_{\ep}\in\Omega_{\ep_0}\setminus D_0$ with $|c|\geq \delta_0$, $|c_{\ep}-W_+(1)|\geq \delta_0$ and $|c_{\ep}-W_-(1)|\geq \delta_0$, by the fact that $\mathcal{D}(c_{\ep})$ is continuous to the boundary and Lemma \ref{lem: D>0}, we have $|\mathcal{D}(c_{\ep})|\geq C^{-1}$ in this case.

Thus, from the above argument, we can deduce that for $c_{\ep}\in\Omega_{\ep_0}\setminus D_0$,
\beno
|\mathcal{D}(c_{\ep})|\geq \frac{l\big(W_+(1)-c_{\ep}\big)l
\big(W_-(1)-c_{\ep}\big)l\big(W_+(-1)-c_{\ep}\big)l
\big(W_-(-1)-c_{\ep}\big)}{C|c_{\ep}|}.
\eeno
This completes the proof of the proposition.
\end{proof}

\begin{remark}\label{rmk: D}
The above proposition implies for $\ep_0$ small enough
and $c_{\ep}\in \Omega_{\ep_0}\setminus D_0$, $|\mathcal{D}(c_{\ep})|$
is lower bounded and $\frac{1}{\mathcal{D}}$ is well-defined
in $\Omega_{\ep_0}\setminus D_0$ and
\beno
\Big|\frac{1}{\mathcal{D}(c_{\ep})}\Big|\leq \frac{C|c_{\ep}|}{l\big(W_+(1)-c_{\ep}\big)l
\big(W_-(1)-c_{\ep}\big)l\big(W_+(-1)-c_{\ep}\big)l
\big(W_-(-1)-c_{\ep}\big)}.
\eeno
 And we have for $c_{\ep}=c+i\ep$ with
 $c\in  D_0\setminus \big\{0,W_+(1),W_-(1),W_+(-1),W_-(-1)\big\}$,
\beno
\lim_{\ep\rightarrow0{\pm}}\frac{1}{\mathcal{D}(c_{\ep})}
=\frac{1}{\mathcal{D}^{re}(c)\pm i\mathcal{D}^{im}(c)}.
\eeno
Moreover, $\frac{1}{\mathcal{D}}$ can be continuous extend
 to the boundary with $\frac{1}{\mathcal{D}(0)}=\frac{1}{\mathcal{D}(W_-(1))}=\frac{1}{\mathcal{D}(W_-(-1))}=\frac{1}{\mathcal{D}(W_+(-1))}
 =\frac{1}{\mathcal{D}(W_+(1))}=0$.
And then we have for $c\in D_0$,
\beno
\Big|\frac{1}{\mathcal{D}^{re}(c)\pm i\mathcal{D}^{im}(c)}\Big|\leq \frac{C |c|}
{l\big(W_+(1)-c\big)l
\big(W_-(1)-c\big)l\big(W_+(-1)-c\big)l
\big(W_-(-1)-c\big)}.
\eeno
\end{remark}
The upper bound of $\frac{1}{\mathcal{D}^{re}(c)\pm i\mathcal{D}^{im}(c)}$ follows from Lemma \ref{lem: I bdd} and Lemma \ref{lem: D>0}. We omit the proof of this remark.

\begin{lemma}\label{lem: limit up bdd}
For $c\in D_0\setminus\{0, W_+(1), W_+(-1), W_-(1),W_-(-1)\}$, there exists a positive constant $C$ such that
\beno
|I^{re}_{\pm}(c)|\leq \frac{Cl\big(W_+(\pm1)-c\big)
l\big(W_-(\pm1)-c\big)}{|c|},
\eeno
and
\beno
|\mathcal{D}^{re}(c)|
\leq\frac{Cl\big(W_+(1)-c\big)
l\big(W_-(1)-c\big)l\big(W_+(-1)-c\big)
l\big(W_-(-1)-c\big)}{|c|},
\eeno
and
\beno
|\mathcal{D}^{im}(c)|
\leq\frac{Cl\big(W_+(1)-c\big)
l\big(W_-(1)-c\big)}{|c|}+\frac{Cl\big(W_+(-1)-c\big)
l\big(W_-(-1)-c\big)}{|c|}.
\eeno
\end{lemma}
\begin{proof}
From Remark \ref{rmk: sigma bdd} and the fact that $|\Pi_{\pm}(c)|+|R^1_{\pm}(c)|+|R^2_{\pm}(c)|\leq C$ for $c\in D_0\setminus\{0, W_+(1), W_+(-1), W_-(1),W_-(-1)\}$, we can easily get the estimate of $I^{re}_{\pm}$.

The estimate of $\mathcal{D}^{re}(c)$ and $\mathcal{D}^{im}(c)$ for $c\in D_0\setminus\{0, W_+(1), W_+(-1), W_-(1),W_-(-1)\}$ follows from the estimate of $I^{re}_{\pm}$ and Proposition \ref{prop: sol. hom. [0,1]}, Proposition \ref{prop: sol. hom. [-1,0]} and \eqref{eq: P est}.
\end{proof}

%

\subsection{Determine the coefficients}

Now we solve the inhomogeneous Sturmian equation,
for  $c\in \Omega_{\ep_0}\setminus D_0$,
\beq\label{eq: inhomo Sturmian}
 \left\{\begin{array}{l}
\pa_y\Big(\mathcal{H}(y,c)\pa_y\Th(y,c)\Big)-\al^2\mathcal{H}(y,c)\Th(y,c)=F(y,c)\\
\Th(-1,c)=\Th(1,c)=0,\\
\end{array}\right.
\eeq
here $F(y,c)=cG(\al,y,c)$ and recall $G(\al, y,c)=G_1(\al, y,c)-\f{\widehat{\phi}_0(\al,0)f(\al,y,c)}{b(y)^3}$ defined as in \eqref{eq: G_1} and \eqref{eq: f}.

In particular, we can get that for $c=0$,
\beno
F(y,0)=-\widehat{\phi}_0(0)\left(\Big(\big(u(y)^2-b(y)^2\big)\Big(\frac{\chi}{b}(y)\Big)'\Big)'
-\al^2\big(u(y)^2-b(y)^2\big)\frac{\chi}{b}(y)\right).
\eeno

For $c\in \Omega_{\ep_0}\setminus\{0\}$, let
\beno
T_{\pm}(F)(c)=\int_0^{\pm1}\frac{\int_{y_{c_{\pm}}}^{y}F(z,c_{\ep})\va_{\pm}(z,c_{\ep})dz}
{\mathcal{H}(y,c_{\ep})\va_{\pm}(y,c_{\ep})^2}dy,
\eeno
and
\beno
L(F)(c)=\va_-(0,c)\int^{y_{c_+}}_0F(y,c)\va_+(y,c)dy
-\va_+(0,c)\int^{y_{c_-}}_0F(y,c)\va_-(y,c)dy.
\eeno

For $y\in[0,1]$ and $c\in \Om_{\ep_0}\setminus D_0$, let
\beq\label{eq: Theta_0+}
\begin{split}
\Th_+^0(y,c)&=\va_+(y,c)\int_0^y\frac{\int_{y_{c_+}}^{y'}(F\va_+)(z,c)dz}
{\mathcal{H}(y',c)\va_+(y',c)^2}dy'\\
&\quad+\wt{\mu}_+(F)(c) \va_+(y,c)\int_0^y\frac{1}{\mathcal{H}(y',c)\va_+(y',c)^2}dy'
+\nu_+(F)(c)\va_+(y,c)
\end{split}
\eeq
and
\beq\label{eq: Theta_1+}
\begin{split}
\Th_+^1(y,c)&=\va_+(y,c)\int_1^y\frac{\int_{y_{c_+}}^{y'}(F\va_+)(z,c)dz}
{\mathcal{H}(y',c)\va_+(y',c)^2}dy'\\
&\quad+\mu_+(F)(c) \va_+(y,c)\int_1^y\frac{1}{\mathcal{H}(y',c)\va_+(y',c)^2}dy',
\end{split}
\eeq
with
\beq\label{mu+}
\begin{split}
 \mu_+(F)(c)=\wt{\mu}_+(F)(c)&=\frac{1}{\mathcal{D}(c)}\big[-c^2P(c)T_+(F)(c)I_-(c)
 -\va_-(0,c)L(F)(c)I_-(c)\\
&\quad
+\va_+(0,c)^2T_+(F)(c)-(\va_+\va_-)(0,c)T_-(F)(c)\big],
\end{split}
\eeq
and
\beq\label{nu+}
\begin{split}
\nu_+(F)(c)&=\frac{1}{\mathcal{D}(c)}\big[\va_-(0,c)L(F)(c)I_+(c)I_-(c)
-(\va_+\va_-)(0,c)T_-(F)(c)I_+(c)\\
 &\quad +\va_-(0,c)^2T_+(F)(c)I_-(c)\big].
\end{split}
\eeq

For $y\in[-1,0]$ and $c\in \Om_{\ep_0}\setminus D_0$, let
\beq\label{eq: Theta_-1-}
\begin{split}
\Th_-^{-1}(y,c)&=\va_-(y,c)\int_{-1}^y\frac{\int_{y_{c_-}}^{y'}(F\va_-)(y'',c)dy''}{\mathcal{H}(y',c)\va_-(y',c)^2}dy'\\
&\quad+\mu_-(F)(c) \va_-(y,c)\int_{-1}^y\frac{1}{\mathcal{H}(y',c)\va_-(y',c)^2}dy',
\end{split}
\eeq
and
\beq\label{eq: Theta_0-}
\begin{split}
\Th_-^0(y,c)&=\va_-(y,c)\int_0^y\frac{\int_{y_{c_-}}^{y'}(F\va_-)(y'',c)dy''}
{\mathcal{H}(y',c)\va_-(y',c)^2}dy'\\
&\quad +\wt{\mu}_-(F)(c) \va_-(y,c)\int_0^y\frac{1}{\mathcal{H}(y',c)\va_-(y',c)^2}dy'
+\nu_-(F)(c)\va_-(y,c),
\end{split}
\eeq
with
\beq\label{mu-}
\begin{split}
\mu_-(F)(c)=\wt{\mu}_-(F)(c)&=\frac{1}{\mathcal{D}(c)}\big[c^2P(c)T_-(F)(c)I_+(c)
+\va_+(0,c)L(F)(c)I_+(c)\\
&\ \
+(\va_+\va_-)(0,c)T_+(F)(c)-\va_-(0,c)^2T_-(F)(c)\big],
\end{split}
\eeq
\beq\label{nu-}
\begin{split}
\nu_-(F)(c)&=\frac{1}{\mathcal{D}(c)}\big[\va_+(0,c)L(F)(c)I_+(c)I_-(c)
+\va_+(0,c)^2T_-(F)(c)I_+(c)\\
 &\quad+(\va_+\va_-)T_+(F)(c)I_-(c)\big].
 \end{split}
\eeq
\begin{proposition}\label{prop: inhom solu}
Let $c\in \Omega_{\ep_0}\setminus D_0$. Then for $y\in[0,1]$, $\Th_{+}^0(y,c)\equiv \Th_{+}^1(y,c)\eqdef \Th_+(y,c)$ and for $y\in[-1,0]$, $\Th_{-}^0(y,c)\equiv \Th_{-}^1(y,c)\eqdef \Th_-(y,c)$. Moreover,
\beno
\Th(y,c)=\left\{\begin{array}{l}
\Th_+(y,c), \ \ y\in[0,1],\\
\Th_-(y,c), \ \ y\in[-1,0],\\
\end{array}\right.
\eeno
is the unique $C^1([-1,1])$ solution to \eqref{eq: inhomo Sturmian}.
\end{proposition}
\begin{proof}
Recall the solution $\va_{\pm}(y,c)$ of \eqref{eq: homo eq} obtained in Proposition \ref{prop: sol. hom. [0,1]} and Proposition \ref{prop: sol. hom. [-1,0]}. Then it is easy to check that the solution of \eqref{eq: inhomo Sturmian} satisfies
\beno
\pa_y\Big(\mathcal{H}(y,c)\va_{\pm}(y,c)^2\pa_y\big(\frac{\Th}{\va_{\pm}}\big)(y,c)\Big)
=\va_{\pm}(y,c)F(y,c).
\eeno
Then the solution of \eqref{eq: inhomo Sturmian} must have the following forms: \\
For $y\in[0,1]$ and $c\in \Om_{\ep_0}\setminus D_0$,
\beno
\begin{split}
\Th(y,c)&=\va_+(y,c)\int_0^y\frac{\int_{y_{c_+}}^{y'}(F\va_+)(y'',c)dy''}{\mathcal{H}(y',c)\va_+(y',c)^2}dy'\\
&\quad+\wt{\mu}_+(F)(c) \va_+(y,c)\int_0^y\frac{1}{\mathcal{H}(y',c)\va_+(y',c)^2}dy'
+\nu_+(F)(c)\va_+(y,c)
\end{split}
\eeno
and
\beno
\begin{split}
\Th(y,c)&=\va_+(y,c)\int_1^y\frac{\int_{y_{c_+}}^{y'}(F\va_+)(y'',c)dy''}{\mathcal{H}(y',c)\va_+(y',c)^2}dy'\\
&\quad+\mu_+(F)(c) \va_+(y,c)\int_1^y\frac{1}{\mathcal{H}(y',c)\va_+(y',c)^2}dy'.
\end{split}
\eeno
For $y\in[-1,0]$ and $c\in \Om_{\ep_0}\setminus D_0$,
\beno
\begin{split}
\Th(y,c)&=\va_-(y,c)\int_{-1}^y\frac{\int_{y_{c_-}}^{y'}(F\va_-)(y'',c)dy''}{\mathcal{H}(y',c)\va_-(y',c)^2}dy'\\
&\quad+\mu_-(F)(c) \va_-(y,c)\int_{-1}^y\frac{1}{\mathcal{H}(y',c)\va_-(y',c)^2}dy',
\end{split}
\eeno
and
\beno
\begin{split}
\Th(y,c)&=\va_-(y,c)\int_0^y\frac{\int_{y_{c_-}}^{y'}(F\va_-)(y'',c)dy''}{\mathcal{H}(y',c)\va_-(y',c)^2}dy'\\
&\quad +\wt{\mu}_-(F)(c) \va_-(y,c)\int_0^y\frac{1}{\mathcal{H}(y',c)\va_-(y',c)^2}dy'
+\nu_-(F)(c)\va_-(y,c).
\end{split}
\eeno
Using the boundary condition and the fact that $\Th(y,c)$ is a $C^1([-1,1])$ function, we obtain that the coefficients are determined by the following equation:
\beqno
 \left\{\begin{array}{l}
\mu_+(F)(c)=\wt{\mu}_+(F)(c), \ \
 \mu_-(F)(c)=\wt{\mu}_-(F)(c),\\
I_+(c)\mu_+(F)(c)+\nu_+(F)(c)=-T_+(F)(c),\\
I_-(c)\mu_-(F)(c)-\nu_-(F)(c)=T_-(F)(c),\\
\va_+(0,c)\nu_+(F)(c)-\va_-(0,c)\nu_-(F)(c)=0,\\
\va_-(0,c)\mu_+(F)(c)-\va_+(0,c)\mu_-(F)(c)
+c^2(\va_-\va_+\pa_y\va_+)(0,c)\nu_+(F)(c)\\
\ \  -c^2(\va_+\va_-\pa_y\va_-)(0,c)\nu_-(F)(c)=L(F)(c),
\end{array}\right.
\eeqno
which is
\begin{gather*}
W
 \begin{bmatrix}
 \mu_+(F)(c)\\
 \mu_-(F)(c)\\
 \nu_+(F)(c)\\
 \nu_-(F)(c)
  \end{bmatrix}
=
\begin{bmatrix}
 -T_+(F)(c)\\
 T_-(F)(c)\\
 0\\
 L(F)(c)
\end{bmatrix}.
\end{gather*}
Therefore we get from Lemma \ref{stern lemma},
 \beno
 \det(W)=c^2P(c)I_+(c)I_-(c)-\va_+(0,c)^2I_+(c)-\va_-(0,c)^2I_-(c)
 = \mathcal{D}(c)\neq0.
 \eeno
 Thus, by solving the matrix equations (\ref{matrix}),
 we can deduce that  $\mu_{\pm}(F)(c), \wt{\mu}_{\pm}(F)(c),\nu_{\pm}(F)(c)$
 satisfy (\ref{mu+}), (\ref{nu+}),  (\ref{mu-}) and (\ref{nu-}). The fact $\Th_{+}^0(y,c)\equiv \Th_{+}^1(y,c)$ and $\Th_{-}^0(y,c)\equiv \Th_{-}^1(y,c)$ can be obtained by the construction. The uniqueness of the solution can be obtained by Lemma \ref{stern lemma}.
Thus we proved the proposition.
\end{proof}

\subsection{The behavior of inhomogeneous solution}
\begin{lemma}\label{lem: T(F)}
Suppose $c_{\ep}=c+i\ep\in D_{\ep_0}$ with $c\in D_0\setminus\{0\}$ and $F\in C([-a,a]\times\Om_{\ep})$. Then we have
\beqno
\lim_{\ep\rightarrow0} T_{\pm}(F)(c_{\ep})=T_{\pm}(F)(c).
\eeqno
Suppose $c_{\ep}\in\Omega_{\ep_0}\setminus D_0$. Then there exists a constant $C>0$ such that,
\beno
\big|T_{\pm}(F)(c_{\ep})\big|\leq C\|F\|_{L^{\infty}}l\big(c_{\ep}\big).
\eeno
\end{lemma}
\begin{proof}
It is easy to check that $|\mathcal{H}(y,c_{\ep})|\geq C^{-1}\big(\big|y^2-y_{c_{\pm}}^2\big|+\ep^2\big)$.
By Proposition \ref{prop: sol. hom. [0,1]} and Proposition \ref{prop: sol. hom. [-1,0]}, we have $C\geq |\va_{\pm}(y,c_{\ep})|\geq \frac{1}{2}$ and
\beno
\Big|\frac{\int_{y_{c_{\pm}}}^{y}F(z,c_{\ep})\va_{\pm}(z,c_{\ep})dz}{\mathcal{H}(y,c_{\ep})\va_{\pm}(y,c_{\ep})^2}
\Big|
\leq C\|F(y,c_{\ep})\|_{L^{\infty}}\|\va_{\pm}\|_{L^{\infty}}
\frac{|y-y_{c_{\pm}}|}{|\mathcal{H}(y,c_{\ep})|}\leq \f{C}{|y|+|c_{\ep}|},
\eeno
which directly implies
\beno
\big|T_{\pm}(F)(c_{\ep})\big|\leq C\|F\|_{L^{\infty}}\Big(\big|\ln|c_{\ep}|\big|+1\Big).
\eeno
Since $F(y,c_{\ep})$, $\mathcal{H}(y,c_{\ep})$ and $\va_{\pm}(y,c_{\ep})$ are continuous functions, then by Lebesgue's dominated convergence theorem, as $\ep\rightarrow0$, it holds that
\beqno
\lim_{\ep\rightarrow0} T_{\pm}(F)(c_{\ep})=T_{\pm}(F)(c).
\eeqno
Thus we complete the proof of the lemma.
\end{proof}

\begin{remark}\label{rmk: T(F)}
For $c\in D_0\setminus \{0\}$, there is a constant $C>0$ such that,
\beno
\big|T_{\pm}(F)(c)\big|\leq C\|F\|_{L^{\infty}}l\big(c_{\ep}\big).
\eeno
\end{remark}
\begin{proof}
For $c\in D_0\setminus \{0\}$, we have $|\mathcal{H}(y,c)|\geq C^{-1}\big|y^2-y_{c_{\pm}}^2\big|$, then
\beno
\big|T_{\pm}(F)(c)\big|\leq C\|F\|_{L^{\infty}}\left|\int_{0}^{\pm 1}\f{1}{|y|+|c|}dy\right|\leq C\|F\|_{L^{\infty}}l\big(c_{\ep}\big).
\eeno
Thus we complete the proof of the remark.
\end{proof}

In the following paper, for $c\in D_0\setminus\{0, W_+(1), W_+(-1), W_-(1),W_-(-1)\}$, let
\begin{align*}
\mathcal{U}_+^{re}(F)(c)&=-c^2P(c)T_+(F)(c)I^{re}_-(c)-\va_-(0,c)L(F)(c)I^{re}_-(c)\\
&\quad
+\va_+(0,c)^2T_+(F)(c)-(\va_+\va_-)(0,c)T_-(F)(c),
\end{align*}
\beqno
\mathcal{U}_+^{im}(F)(c)=-\frac{\pi c^2P(c)T_+(F)(c)\chi_-(c)}{\sigma_-(c)}
-\frac{\pi\va_-(0,c)L(F)(c)\chi_-(c)}{\sigma_-(c)},
\eeqno
\begin{align*}
\mathcal{U}^{re}_-(F)(c)&=c^2P(c)T_-(F)(c)I^{re}_+(c)+\va_+(0,c)L(F)(c)I^{re}_+(c)\\
&\ \
+(\va_+\va_-)(0,c)T_+(F)(c)-\va_-(0,c)^2T_-(F)(c),
\end{align*}
\begin{align*}
\mathcal{U}^{im}_-(F)(c)=\frac{\pi c^2P(c)T_-(F)(c)\chi_+(c)}{\sigma_+(c)}
+\frac{\pi\va_+(0,c)L(F)(c)\chi_+(c)}{\sigma_+(c)},
\end{align*}
\begin{align*}
\mathcal{V}_+^{re}(F)(c)&=\va_-(0,c)L(F)(c)I^{re}_+(c)I^{re}_-(c)
+(\va_+\va_-)(0,c)T_-(F)(c)I^{re}_+(c)\\
&\quad
+\va_-(0,c)^2T_+(F)(c)I^{re}_-(c)
-\frac{\pi^2\va_-(0,c)L(F)(c)\chi_+(c)\chi_-(c)}{\sigma_+(c)\sigma_-(c)},
\end{align*}
\begin{align*}
\mathcal{V}_+^{im}(F)(c)&=\pi\va_-(0,c)L(F)(c)
\Big(\frac{I^{re}_+(c)\chi_-(c)}{\sigma_-(c)}+\frac{I^{re}_-(c)\chi_+(c)}{\sigma_+(c)}\Big)\\
&\quad
+\frac{\pi(\va_+\va_-)(0,c)T_-(F)(c)\chi_+(c)}{\sigma_+(c)}
+\frac{\pi\va_-(0,c)^2T_+(F)(c)\chi_-(c)}{\sigma_-(c)},
\end{align*}
\begin{align*}
\mathcal{V}_-^{re}(F)(c)&=
\va_+(0,c)L(F)(c)I^{re}_+(c)I^{re}_-(c)
+\va_+(0,c)^2T_-(F)(c)I^{re}_+(c)\\
&\quad
+(\va_+\va_-)(0,c)T_+(F)(c)I^{re}_-(c)-\frac{\pi^2\va_+(0,c)L(F)(c)\chi_+(c)\chi_-(c)}{\sigma_+(c)\sigma_-(c)},
\end{align*}
\begin{align*}
\mathcal{V}_-^{im}(F)(c)&=\pi\va_+(0,c)L(F)(c)
\Big(\frac{I^{re}_+(c)\chi_-(c)}{\sigma_-(c)}+\frac{I^{re}_-(c)\chi_+(c)}{\sigma_+(c)}\Big)\\
&\quad
+\frac{\pi\va_+(0,c)^2T_-(F)(c)\chi_+(c)}{\sigma_+(c)}
+\frac{\pi(\va_+\va_-)(0,c)T_+(F)(c)\chi_-(c)}{\sigma_-(c)}.
\end{align*}
And we also introduce for $c\in D_0\setminus\{0, W_+(1), W_+(-1), W_-(1),W_-(-1)\}$,
\begin{align*}
\mu_+^+(F)(c)&=\frac{\mathcal{U}^{re}_+(F)(c)
+i\mathcal{U}^{im}_+(F)(c)}{\mathcal{D}^{re}(c)+ i\mathcal{D}^{im}(c)},
\ \  \mu_+^-(F)(c)=\frac{\mathcal{U}^{re}_+(F)(c)- i\mathcal{U}^{im}_+(F)(c)}{\mathcal{D}^{re}(c)- i\mathcal{D}^{im}(c)},\\
\mu_-^+(F)(c)&=\frac{\mathcal{U}^{re}_-(F)(c)+i\mathcal{U}^{im}_-(F)(c)}{\mathcal{D}^{re}(c)+ i\mathcal{D}^{im}(c)},\ \
\mu_-^-(F)(c)=\frac{\mathcal{U}^{re}_-(F)(c)-i\mathcal{U}^{im}_-(F)(c)}{\mathcal{D}^{re}(c)- i\mathcal{D}^{im}(c)},\\
\nu_+^+(F)(c)&=\frac{\mathcal{V}^{re}_+(F)(c)+i\mathcal{V}^{im}_+(F)(c)}{\mathcal{D}^{re}(c)+ i\mathcal{D}^{im}(c)}, \ \
\nu_+^-(F)(c)=\frac{\mathcal{V}^{re}_+(F)(c)-i\mathcal{V}^{im}_+(F)(c)}{\mathcal{D}^{re}(c)- i\mathcal{D}^{im}(c)},\\
\nu_-^+(F)(c)&=\frac{\mathcal{V}^{re}_-(F)(c)+i\mathcal{V}^{im}_-(F)(c)}
{\mathcal{D}^{re}(c)+ i\mathcal{D}^{im}(c)},\ \
\nu_-^-(F)(c)=\frac{\mathcal{V}^{re}_-(F)(c)-i\mathcal{V}^{im}_-(F)(c)}
{\mathcal{D}^{re}(c)- i\mathcal{D}^{im}(c)}.
\end{align*}
\begin{proposition}\label{prop: mu nu lim}
(1) Let $c_{\ep}=c\pm i\ep\in \Omega_{\ep_0}\setminus D_0$, $0<\ep<\ep_0$. There holds that
\begin{align*}
&|\mu_+(F)(c_{\ep})|\leq \frac{C\|F\|_{L^{\infty}}|c_{\ep}|l\big(c_{\ep}\big)
\Big(l\big(W_+(-1)-c_{\ep}\big)+l\big(W_-(-1)-c_{\ep}\big)\Big)}
{l\big(W_+(1)-c_{\ep}\big)l\big(W_-(1)-c_{\ep}\big)
l\big(W_+(-1)-c_{\ep}\big)l\big(W_-(-1)-c_{\ep}\big)},\\
&|\mu_-(F)(c_{\ep})|\leq \frac{C\|F\|_{L^{\infty}}|c_{\ep}|l\big(c_{\ep}\big)
\Big(l\big(W_+(1)-c_{\ep}\big)+l\big(W_-(1)-c_{\ep}\big)\Big)}
{l\big(W_+(1)-c_{\ep}\big)l\big(W_-(1)-c_{\ep}\big)
l\big(W_+(-1)-c_{\ep}\big)l\big(W_-(-1)-c_{\ep}\big)},\\
&|\nu_{\pm}(F)(c_{\ep})|\leq C\|F\|_{L^{\infty}}l\big(c_{\ep}\big).
\end{align*}
(2) For $c_{\ep}=c+i\ep$, $0<\ep<1$, $c\in D_0\setminus\{0, W_+(1), W_+(-1), W_-(1),W_-(-1)\}$, there holds that
\beno
\lim_{\ep\rightarrow0^{\pm}}\mu_+(F)(c_{\ep})=\mu_+^{\pm}(F)(c),\ \
\lim_{\ep\rightarrow0^{\pm}}\nu_+(F)(c_{\ep})=\nu_+^{\pm}(F)(c),
\eeno
\beno
\lim_{\ep\rightarrow0^{\pm}}\mu_-(F)(c_{\ep})=\mu_-^{\pm}(F)(c),\ \
\lim_{\ep\rightarrow0^{\pm}}\nu_-(F)(c_{\ep})=\nu_-^{\pm}(F)(c).
\eeno
\end{proposition}
\begin{proof}
By Proposition \ref{prop: sol. hom. [0,1]} and Proposition \ref{prop: sol. hom. [-1,0]}, we have
$
|\pa_y\va_{\pm}(0,c_{\ep})|\leq C|c_{\ep}|,
$
which gives for $c_{\ep}\in\Omega_{\ep_0}$,
\beqno
\big|P(c_{\ep})\big|=\big|(\va_-^2\va_+\pa_y\va_+)(0,c_{\ep})
-(\va_+^2\va_-\pa_y\va_-)(0,c_{\ep})\big|\leq C|c_{\ep}\big|,
\eeqno
and we also have
\beq\label{eq: L(F) bdd}
\begin{split}
\big|L(F)(c_{\ep})\big|&=\Big|
\va_-(0,c_{\ep})\int_0^{y_{c_+}}(F\va_+)(y,c_{\ep})dy
-\va_+(0,c_{\ep})\int^{y_{c_-}}_0(F\va_-)(y,c_{\ep})dy\Big|\\
&\leq C|c_{\ep}|\|F\|_{L^{\infty}}.
\end{split}
\eeq
 From which and by using Lemma \ref{prop: D est},  Lemma \ref{lem: I bdd} and Lemma \ref{lem: T(F)} gives rise to
\begin{align*}
&|\mu_+(F)(c_{\ep})|\leq 
\frac{C\|F\|_{L^{\infty}}|c_{\ep}|\big(\big|\ln|c_{\ep}|\big|+1\big)
\Big(l\big(W_+(-1)-c_{\ep}\big)+l\big(W_-(-1)-c_{\ep}\big)\Big)}
{l\big(W_+(1)-c_{\ep}\big)l\big(W_-(1)-c_{\ep}\big)
l\big(W_+(-1)-c_{\ep}\big)l\big(W_-(-1)-c_{\ep}\big)},\\
&|\mu_-(F)(c_{\ep})|\leq \frac{C\|F\|_{L^{\infty}}|c_{\ep}|\big(\big|\ln|c_{\ep}|\big|+1\big)
\Big(l\big(W_+(1)-c_{\ep}\big)+l\big(W_-(1)-c_{\ep}\big)\Big)}
{l\big(W_+(1)-c_{\ep}\big)l\big(W_-(1)-c_{\ep}\big)
l\big(W_+(-1)-c_{\ep}\big)l\big(W_-(-1)-c_{\ep}\big)}.
\end{align*}
Similarly, we get
\beno
|\nu_{\pm}(F)(c_{\ep})|\leq C\|F\|_{L^{\infty}}\big(\big|\ln|c_{\ep}|\big|+1\big).
\eeno

On the other hand, for $c\in D_0\setminus\{0, W_+(1), W_+(-1), W_-(1),W_-(-1)\}$,
by Lemma \ref{lem: D>0}, Lemma \ref{lem: T(F)}, \eqref{I+ lim} and \eqref{I- lim},
we can easily get that
\beno
&\ & \lim_{\ep\rightarrow0^{\pm}}\mu_+(F)(c_{\ep})
=\mu_+^{\pm}(F)(c),\ \ \lim_{\ep\rightarrow0^{\pm}}\nu_+(F)(c_{\ep})=\nu_+^{\pm}(F)(c),\\
&\ & \lim_{\ep\rightarrow0^{\pm}}\mu_-(F)(c_{\ep})
=\mu_-^{\pm}(F)(c),\ \ \lim_{\ep\rightarrow0^{\pm}}\nu_-(F)(c_{\ep})=\nu_-^{\pm}(F)(c).
\eeno

Thus we complete the proof of Proposition \ref{prop: mu nu lim}.
\end{proof}
\begin{remark}\label{rmk: munu lim up rem}
From Proposition \ref{prop: mu nu lim}, we have for $c\in D_0\setminus\{0, W_+(1), W_+(-1), W_-(1),W_-(-1)\}$,
\begin{align*}
|\mu^{\pm}_+(F)(c)|&\leq\frac{C\|F\|_{L^{\infty}}|c|l(c)
\Big(l\big(W_+(-1)-c\big)+l\big(W_-(-1)-c\big)\Big)}
{l\big(W_+(1)-c\big)l\big(W_-(1)-c\big)
l\big(W_+(-1)-c\big)l\big(W_-(-1)-c\big)},\\
|\mu^{\pm}_-(F)(c)|&\leq\frac{C\|F\|_{L^{\infty}}|c|l(c)
\Big(l\big(W_+(1)-c\big)+l\big(W_-(1)-c\big)\Big)}
{l\big(W_+(1)-c\big)l\big(W_-(1)-c\big)
l\big(W_+(-1)-c\big)l\big(W_-(-1)-c\big)},\\
|\nu^{\pm}_{\pm}(F)(c)|&\leq C\|F\|_{L^{\infty}}l\big(c\big).
\end{align*}
\end{remark}

\begin{lemma}\label{lem: XYST bdd}
For $c\in D_0\setminus\{0, W_+(1), W_+(-1), W_-(1),W_-(-1)\}$, we have the following estimates:
\begin{align*}
|\mathcal{U}^{re}_+(F)(c)|&\leq C\|F\|_{L^{\infty}}l\big(W_+(-1)-c\big)l\big(W_-(-1)-c\big)l(c),\\
|\mathcal{U}^{re}_-(F)(c)|&\leq C\|F\|_{L^{\infty}}l\big(W_+(1)-c\big)l\big(W_-(1)-c\big)l(c),\\
|\mathcal{U}^{im}_{\pm}(F)(c)|&\leq C\|F\|_{L^{\infty}},\\
|\mathcal{V}^{re}_{\pm}(F)(c)|&\leq \frac{C\|F\|_{L^{\infty}}}{|c|}l\big(W_+(1)-c\big)l\big(W_-(1)-c\big)
l\big(W_+(-1)-c\big)l\big(W_-(-1)-c\big)l(c),\\
|\mathcal{V}^{im}_{\pm}(F)(c)|&\leq \frac{C\|F\|_{L^{\infty}}}{|c|}\Big(l\big(W_+(1)-c\big)l\big(W_-(1)-c\big)
+l\big(W_+(-1)-c\big)l\big(W_-(-1)-c\big)+l(c)\Big).
\end{align*}
\end{lemma}
\begin{proof}
By Lemma \ref{lem: limit up bdd}, Remark \ref{rmk: T(F)}, \eqref{eq: P est} and  \eqref{eq: L(F) bdd}, we can get
\begin{align*}
|\mathcal{U}^{re}_+(F)(c)|&\leq C\|F\|_{L^{\infty}}|c|^2\big(\big|\ln|c|\big|+1\big)\Big(1+\big|\ln|W_+(-1)-c|\big|
+\big|\ln|c-W_-(-1)|\big|\Big)\nonumber\\
&\quad +C\|F\|_{L^{\infty}}\Big(1+\big|\ln|W_+(-1)-c|\big|+
\big|\ln|c-W_-(-1)|\big|\Big)+C\big(\big|\ln|c|\big|+1\big)\nonumber\\
&\leq C\|F\|_{L^{\infty}}\Big(1+\big|\ln|W_+(-1)-c|\big|+\big|\ln|c-W_-(-1)|\big|
+\big|\ln|c|\big|\Big)
\end{align*}
and 
\begin{align*}
|\mathcal{U}^{re}_-(F)(c)|&\leq C\|F\|_{L^{\infty}}|c|^2\big(\big|\ln|c|\big|+1\big)\Big(1+\big|\ln|W_+(1)-c|\big|
+\big|\ln|c-W_-(1)|\big|\Big)\nonumber\\
&\quad +C\|F\|_{L^{\infty}}\Big(1+\big|\ln|W_+(1)-c|\big|+
\big|\ln|c-W_-(1)|\big|\Big)+C\big(\big|\ln|c|\big|+1\big)\nonumber\\
&\leq C\|F\|_{L^{\infty}}\Big(1+\big|\ln|W_+(1)-c|\big|+\big|\ln|c-W_-(1)|\big|
+\big|\ln|c|\big|\Big)
\end{align*}
From Remark \ref{rmk: sigma bdd}, Remark \ref{rmk: T(F)},  \eqref{eq: P est} and  \eqref{eq: L(F) bdd}, we have
\begin{align*}
|\mathcal{U}^{im}_{\pm}(F)(c)|\leq C\|F\|_{L^{\infty}}+C\|F\|_{L^{\infty}}|c|^2\big(1+\big|\ln|c|\big|\big)\leq C\|F\|_{L^{\infty}}.
\end{align*}
Similarly, combining Lemma \ref{lem: limit up bdd}, Remark \ref{rmk: sigma bdd}, Remark \ref{rmk: T(F)} and \eqref{eq: L(F) bdd}, we deduce
\begin{align*}
|\mathcal{V}^{re}_{\pm}(F)(c)|
&\leq \frac{C\|F\|_{L^{\infty}}}{|c|}\Big\{\Big(1+\big|\ln|W_+(1)-c|\big|+\big|
\ln|c-W_-(1)|\big|\Big)\\
&\ \ \times\Big(1+\big|\ln|W_+(-1)-c|\big|+\big|
\ln|c-W_-(-1)|\big|\Big)\Big\}\\
&\quad +\frac{C\|F\|_{L^{\infty}}}{|c|}\big(1+\big|\ln|c|\big|\big)
\Big\{1+\big|\ln|W_+(1)-c|\big|+\big|
\ln|c-W_-(1)|\big|\\
&\ \ +\big|\ln|W_+(-1)-c|\big|+\big|
\ln|c-W_-(-1)|\big|\Big\}+\frac{C\|F\|_{L^{\infty}}}{|c|}\\
&\leq \frac{C\|F\|_{L^{\infty}}}{|c|}l\big(W_+(1)-c\big)l\big(W_-(1)-c\big)
l\big(W_+(-1)-c\big)l\big(W_-(-1)-c\big)l(c),
\end{align*}
and
\begin{align*}
|\mathcal{V}^{im}_{\pm}(F)(c)|
&\leq C\|F\|_{L^{\infty}}\Big(1+\big|\ln|W_+(1)-c|\big|+\big|
\ln|c-W_-(1)|\big|+\big|\ln|W_+(-1)-c|\big|\\
&\ \ +\big|
\ln|c-W_-(-1)|\big|\Big)+\frac{C\|F\|_{L^{\infty}}\big(1+\big|\ln|c|\big|\big)}{|c|}\\
&\leq \frac{C\|F\|_{L^{\infty}}}{|c|}\Big(1+\big|\ln|W_+(1)-c|\big|
+\big|\ln|c-W_-(1)|\big|\\
&\ \ +\big|\ln|W_+(-1)-c|\big|+\big|
\ln|c-W_-(-1)|\big|+\big|\ln|c|\big|\Big)\\
&\leq \frac{C\|F\|_{L^{\infty}}}{|c|}\Big(l\big(W_+(1)-c\big)+l\big(W_-(1)-c\big)
+l\big(W_+(-1)-c\big)\\
&\ \ +l\big(W_-(-1)-c\big)+l(c)\Big).
\end{align*}
Thus we prove the lemma.
\end{proof}
\begin{proposition}\label{prop: F y bdd lim}
1. Let $c_{\ep}=c+i\ep\in D_{\ep_0}\cup D_0$. Then it holds that,\\
for $0\leq y<y_{c_+}\leq 1$ or $0\leq y\leq 1<y_{c_+}\leq a_+$
\begin{align*}
&\left|\va_+(y,c_{\ep})
\int_0^y\frac{1}{\mathcal{H}(y',c_{\ep})\va_+(y',c_{\ep})^2}dy'\right|
\leq  \frac{C\left(\big|\ln|y-y_{c_+}|\big|+1\right)}{|c_{\ep}|},\\
&\left|\va_+(y,c_{\ep})
\int_0^y\frac{\int_{y_{c_+}}^{y'}(F\va_+)(z,c_{\ep})dz}
{\mathcal{H}(y',c_{\ep})\va_+(y',c_{\ep})^2}dy'\right|
\leq C\|F\|_{L^{\infty}}\big(\big|\ln(|y|+|c_{\ep}|)\big|+1\big),
\end{align*}
and for $0\leq y_{c_+}<y\leq 1$,
\begin{align*}
&\left|\va_+(y,c_{\ep})
\int_1^y\frac{1}{\mathcal{H}(y',c_{\ep})\va_+(y',c_{\ep})^2}dy'\right|
\leq \frac{C\left(\big|\ln|y-y_{c_+}|\big|+1\right)}{|y|},\\
&\left|\va_+(y,c_{\ep})
\int_1^y\frac{\int_{y_{c_+}}^{y'}(F\va_+)(z,c_{\ep})dz}
{\mathcal{H}(y',c_{\ep})\va_+(y',c_{\ep})^2}dy'\right|
\leq C\|F\|_{L^{\infty}}\big(\big|\ln(|y|+|c_{\ep}|)\big|+1\big);
\end{align*}
and for $-1\leq y<y_{c_-}\leq 0$,
\begin{align*}
&\left|\va_-(y,c_{\ep})
\int_{-1}^y\frac{1}{\mathcal{H}(y',c_{\ep})\va_-(y',c_{\ep})^2}dy'\right|\leq \frac{C\left(\big|\ln|y-y_{c_-}|\big|+1\right)}{|y|},\\
&\left|\va_-(y,c_{\ep})
\int_{-1}^y\frac{\int_{y_{c_-}}^{y'}(F\va_-)(z,c_{\ep})dz}
{\mathcal{H}(y',c_{\ep})\va_-(y',c_{\ep})^2}dy'\right|
\leq C\|F\|_{L^{\infty}}\big(\big|\ln(|y|+|c_{\ep}|)\big|+1\big),
\end{align*}
and for $-1\leq y_{c_-}<y \leq0$ or $a_-\leq y_{c_-}<-1 \leq y\leq 0$,
\begin{align*}
&\left|\va_-(y,c_{\ep})
\int_0^y\frac{1}{\mathcal{H}(y',c_{\ep})\va_-(y',c_{\ep})^2}dy'\right|
\leq \frac{C\left(\big|\ln|y-y_{c_-}|\big|+1\right)}{|c_{\ep}|},\\
&\left|\va_-(y,c_{\ep})
\int_0^y\frac{\int_{y_{c_-}}^{y'}(F\va_-)(z,c_{\ep})dz}
{\mathcal{H}(y',c_{\ep})\va_-(y',c_{\ep})^2}dy'\right|
\leq C\|F\|_{L^{\infty}}\big(\big|\ln(|y|+|c_{\ep}|)\big|+1\big).
\end{align*}

2. Let $c_{\ep}=c+i\ep\in D_{\ep_0}$. Then it holds that for $0\leq y<y_{c_+}\leq 1$ or $0\leq y\leq1<y_{c_+}\leq a_+$,
\beno
\lim_{\ep\rightarrow0}\va_+(y,c_{\ep})
\int_0^y\frac{1}{\mathcal{H}(y',c_{\ep})\va_+(y',c_{\ep})^2}dy'=\va_+(y,c)
\int_0^y\frac{1}{\mathcal{H}(y',c)\va_+(y',c)^2}dy',
\eeno
\beno
\lim_{\ep\rightarrow0}\va_+(y,c_{\ep})
\int_0^y\frac{\int_{y_{c_+}}^{y'}(F\va_+)(z,c_{\ep})dz}
{\mathcal{H}(y',c_{\ep})\va_+(y',c_{\ep})^2}dy'=\va_+(y,c)
\int_0^y\frac{\int_{y_{c_+}}^{y'}(F\va_+)(z,c)dz}
{\mathcal{H}(y',c)\va_+(y',c)^2}dy',
\eeno
and for $0\leq y_{c_+}<y\leq 1$,
\beno
\lim_{\ep\rightarrow0}\va_+(y,c_{\ep})
\int_1^y\frac{1}{\mathcal{H}(y',c_{\ep})\va_+(y',c_{\ep})^2}dy'=\va_+(y,c)
\int_1^y\frac{1}{\mathcal{H}(y',c)\va_+(y',c)^2}dy',
\eeno
\beno
\lim_{\ep\rightarrow0}\va_+(y,c_{\ep})
\int_1^y\frac{\int_{y_{c_+}}^{y'}(F\va_+)(z,c_{\ep})dz}
{\mathcal{H}(y',c_{\ep})\va_+(y',c_{\ep})^2}dy'=\va_+(y,c)
\int_1^y\frac{\int_{y_{c_+}}^{y'}(F\va_+)(z,c)dz}
{\mathcal{H}(y',c)\va_+(y',c)^2}dy'.
\eeno
For $-1\leq y<y_{c_-}\leq 0$,
\beno
\lim_{\ep\rightarrow0}\va_-(y,c_{\ep})
\int_{-1}^y\frac{1}{\mathcal{H}(y',c_{\ep})\va_-(y',c_{\ep})^2}dy'=\va_-(y,c)
\int_{-1}^y\frac{1}{\mathcal{H}(y',c)\va_-(y',c)^2}dy',
\eeno
\beno
\lim_{\ep\rightarrow0}\va_-(y,c_{\ep})
\int_{-1}^y\frac{1}{\mathcal{H}(y',c_{\ep})\va_-(y',c_{\ep})^2}dy'=\va_-(y,c)
\int_{-1}^y\frac{1}{\mathcal{H}(y',c)\va_-(y',c)^2}dy',
\eeno
and for $-1\leq y_{c_-}<y \leq0$ or $a_-\leq y_{c_-}<-1 \leq y\leq 0$,
\beno
\lim_{\ep\rightarrow0}\va_-(y,c_{\ep})
\int_0^y\frac{1}{\mathcal{H}(y',c_{\ep})\va_-(y',c_{\ep})^2}dy'=\va_-(y,c)
\int_0^y\frac{1}{\mathcal{H}(y',c)\va_-(y',c)^2}dy',
\eeno
\beno
\lim_{\ep\rightarrow0}\va_-(y,c_{\ep})
\int_0^y\frac{1}{\mathcal{H}(y',c_{\ep})\va_-(y',c_{\ep})^2}dy'=\va_-(y,c)
\int_0^y\frac{1}{\mathcal{H}(y',c)\va_-(y',c)^2}dy'.
\eeno

\end{proposition}

\begin{proof}
We consider the case of $y\in[0,1]$ and the case of $y\in[-1,0]$ can be proved by the same argument.

By Proposition \ref{prop: sol. hom. [0,1]}, we get for $c\in D_{\ep_0}\cup D_0$
\beno
\left|\frac{1}{\mathcal{H}(y',c)\va_+(y',c)^2}\right|\leq \f{C}{|y'-y_{c_+}|(|y'|+|c|)},
\eeno
and
\beno
\left|\frac{\int_{y_{c_+}}^{y'}(F\va_+)(z,c)dz}
{\mathcal{H}(y',c)\va_+(y',c)^2}\right|\leq \f{C\|F\|_{L^{\infty}}}{|y'|+|c|},
\eeno
which implies for $0\leq y<y_{c_+}\leq 1$  or $0\leq y\leq1<y_{c_+}\leq a_+$,
\begin{align*}
&\Big|\va_+(y,c)
\int_0^y\frac{1}{\mathcal{H}(y',c)\va_+(y',c)^2}dy'\Big|\leq
\frac{C\left(\big|\ln|y_{c_+}-y|\big|+1\right)}{|c|}\\
&\Big|\va_+(y,c)\int_0^y\frac{\int_{y_{c_+}}^{y'}(F\va_+)(z,c)dz}
{\mathcal{H}(y',c)\va_+(y',c)^2}dy'\Big|\leq C\|F\|_{L^{\infty}}\big(\big|\ln(|y|+|c|)\big|+1\big).
\end{align*}
and for $0\leq y_{c_+}<y\leq 1$,
\begin{align*}
&\Big|\va_+(y,c)
\int_1^y\frac{1}{\mathcal{H}(y',c)\va_+(y',c)^2}dy'\Big|\leq
\frac{C\left(\big|\ln|y_{c_+}-y|\big|+1\right)}{|y|}\\
&\Big|\int_1^y\va_+(y,c)\frac{\int_{y_{c_+}}^{y'}(F\va_+)(z,c)dz}
{\mathcal{H}(y',c)\va_+(y',c)^2}dy'\Big|\leq C\|F\|_{L^{\infty}}\big(\big|\ln(|y|+|c|)\big|+1\big)
\end{align*}

Since $F(y,c_{\ep})$, $\mathcal{H}(y,c_{\ep})$ and $\va_{\pm}(y,c_{\ep})$ are continuous functions, then by Lebesgue's dominated convergence theorem, as $\ep\rightarrow0$, we can obtain the second part.
\end{proof}
\begin{proposition}\label{prop: Br Bl Theta bdd}
Let $c_{\ep}\in B_{\ep_0}^l$ or $c_{\ep}\in B_{\ep_0}^r$. Then it holds that
\beno
|\Th(y,c_{\ep})|\leq  C\|F\|_{L^{\infty}}.
\eeno
\end{proposition}
\begin{proof}
We only show the proof of the case $0\leq y\leq 1$ and $c_{\ep}\in  B_{\ep_0}^r$, the proofs of the other three cases are similar.

In this case, $c_{\ep}=\max\{W_+(1),W_-(-1)\}+\ep e^{i\th}$ and $c_r=\max\{W_+(1),W_-(-1)\}$ with $y_{c_+}=W_+^{-1}\big(\max\{W_+(1),W_-(-1)\}\big)=a_+$ and by Proposition \ref{prop: inhom solu}, we can write $\Th(y,c_{\ep})$ in the following way,
\begin{align*}
\Th(y,c_{\ep})&=\va_+(y,c_{\ep})\int_0^y\frac{\int_{a_+}^{y'}(F\va_+)(z,c_{\ep})dz}
{\mathcal{H}(y',c_{\ep})\va_+(y',c_{\ep})^2}dy'\\
&\quad+\mu_+(F)(c_{\ep}) \va_+(y,c_{\ep})\int_0^y\frac{y_{c_+}}{\mathcal{H}(y',c_{\ep})\va_+(y',c_{\ep})^2}dy'
+\nu_+(F)(c_{\ep})\va_+(y,c_{\ep}).
\end{align*}
Then we have
\begin{align*}
&\left|\va_+(y,c_{\ep})\int_0^y\frac{\int_{a_+}^{y'}(F\va_+)(z,c_{\ep})dz}
{\mathcal{H}(y',c_{\ep})\va_+(y',c_{\ep})^2}dy'\right|\\
&\leq C\|F\|_{L^{\infty}}\left|\int_0^y\f{|y'-a_+|}{|W_+(y')-W_+(1)-\ep e^{i\th}||W_-(y')-W_+(1)-\ep e^{i\th}|}dy'\right|\leq C\|F\|_{L^{\infty}}.
\end{align*}
By Proposition \ref{prop: mu nu lim}, we have
\begin{align*}
&\left|\mu_+(F)(c_{\ep}) \va_+(y,c_{\ep})\int_0^y\frac{1}{\mathcal{H}(y',c_{\ep})\va_+(y',c_{\ep})^2}dy'\right|\\
&\leq \f{C\|F\|_{L^{\infty}}}{1+|\ln |\ep||}\left|\int_0^y\f{dy'}{|W_+(y')-c_r-\ep e^{i\th}||W_-(y')-c_r-\ep e^{i\th}|}\right|
\leq C\|F\|_{L^{\infty}},
\end{align*}
and
\beno
\left|\nu_+(F)(c_{\ep})\va_+(y,c_{\ep})\right|\leq C\|F\|_{L^{\infty}}.
\eeno
Thus we prove the proposition.
\end{proof}
\section{The proof of main theorem}
Now we present the proof of Theorem \ref{main thm}.
\begin{proof}
We recall that $\big(cI-M_{\al}\big)^{-1}\Big(\begin{array}{l} \widehat{\psi}_0\\ \widehat{\phi}_0\end{array}\Big)(\al,y)=\Big(\begin{array}{l} \Psi_1\\ \Phi_1\end{array}\Big)(\al,y,c)$ and let $\Phi_1(\al,y,c)=b(y)\Phi(\al,y,c)+\widehat{\phi}_0(\al,0)\chi(y)/c$, then $\Psi_1(\al,y,c)=(u(y)-c)\Phi(\al,y,c)
+(u(y)-c)\widehat{\phi}_0(\al,0)\f{\chi(y)}{cb(y)}+\f{\widehat{\phi}_0(\al,y)}{b(y)}$ and $\Phi(\al, y,c)$ satisfies
\beno
\pa_y\Big[\Big(\big(u(y)-c\big)^2-b(y)^2\Big)\pa_y{\Phi(\al, y,c)}\Big]
-\al^2\Big(\big(u(y)-c\big)^2-b(y)^2\Big){\Phi(\al,y,c)}=G(\al,y,c).
\eeno
In Proposition \ref{prop: inhom solu} we proved that $\Th(y,c)$ satisfies
\beno
\pa_y\Big[\Big(\big(u(y)-c\big)^2-b(y)^2\Big)\pa_y\Th(y,c)\Big]
-\al^2\Big(\big(u(y)-c\big)^2-b(y)^2\Big)\Th(y,c)=F(y,c)=cG(\al,y,c).
\eeno
Thus $\Th(y,c)=c{\Phi(\al,y,c)}$ and by \eqref{eq: matrix-form}, we obtain that for $y\in[-1,1]$,
\begin{align*}
&\widehat{\psi}(t,\al,y)\\
&=\frac{1}{2\pi i}
\int_{\pa\Omega_{\ep_0}}e^{-i\al tc_{\ep}}\Big(\frac{u(y)-c_{\ep}}{c_{\ep}}\Theta(y,c_{\ep})
+\frac{u(y)-c_{\ep}}{c_{\ep}b(y)}\widehat{\phi}_0(\al,0)\chi(y)
+\frac{\widehat{\phi}_0(\al,y)}{b(y)}\Big)dc_{\ep}\\
&=\frac{1}{2\pi i}
\int_{\pa\Omega_{\ep_0}}e^{-i\al tc_{\ep}}\Big(\frac{u(y)-c_{\ep}}{c_{\ep}}\Theta(y,c_{\ep})
+\frac{u(y)\widehat{\phi}_0(\al,0)\chi(y)}{c_{\ep}b(y)}
+\frac{\widehat{\phi}_0(\al,y)-\chi(y)\widehat{\phi}(\al,0)}{b(y)}\Big)dc_{\ep},
\end{align*}
and
\begin{align*}
\widehat{\phi}(t,\al,y)&=\frac{1}{2\pi i}
\int_{\pa\Omega_{\ep_0}}e^{-i\al tc_{\ep}}\Big(\frac{b(y)}{c_{\ep}}\Theta(y,c_{\ep})
+\frac{\widehat{\phi}_0(\al,0)\chi(y)}{c_{\ep}}\Big)dc_{\ep}.
\end{align*}
By the fact that $\Theta(y,c_{\ep})$ is an analytic function in $\Om_{\ep_0}\setminus D_0$, we obtain
\beq\label{eq: psi and phi}
\begin{split}
\widehat{\psi}(t,\al,y)&=\frac{u(y)}{b(y)}\widehat{\phi}_0(\al,0)\chi(y)+\lim_{\ep_0\rightarrow0+}\frac{1}{2\pi i}
\int_{\pa\Omega_{\ep_0}}e^{-i\al tc_{\ep}}\Big(\frac{u(y)-c_{\ep}}{c_{\ep}}\Theta(y,c_{\ep})\Big)dc_{\ep},\\
\widehat{\phi}(t,\al,y)&=\widehat{\phi}_0(\al,0)\chi(y)+\lim_{\ep_0\rightarrow0+}\frac{1}{2\pi i}
\int_{\pa\Omega_{\ep_0}}e^{-i\al tc_{\ep}}\Big(\frac{b(y)}{c_{\ep}}\Theta(y,c_{\ep})\Big)dc_{\ep}.
\end{split}
\eeq
In the following, we denote 
$M_+=\max\{W_+(1),W_-(-1)\}$ and $m_-=\min\{W_+(-1),W_-(1)\}$ 
for brevity.

\no{\bf{Proof \ of \ 1.}}
For $y=0$ and $\chi(0)=1$, we get that,
\beno
\widehat{\phi}(t,\al,0)=\lim_{\ep_0\rightarrow0+}\frac{1}{2\pi i}
\int_{\pa\Omega_{\ep_0}}e^{-i\al tc_{\ep}}\frac{\widehat{\phi}_0(\al,0)}{c_{\ep}}dc_{\ep}
=\widehat{\phi}_0(\al,0),
\eeno
and
\begin{align*}
\widehat{\psi}(t,\al,0)&=
\frac{u'(0)}{b'(0)}\widehat{\phi}_0(\al,0)
-\lim_{\ep_0\rightarrow0+}\frac{1}{2\pi i}
\int_{\pa\Omega_{\ep_0}}e^{-i\al tc_{\ep}}\Theta(0,c_{\ep})dc_{\ep}\\
&=\frac{u'(0)}{b'(0)}\widehat{\phi}_0(\al,0)
-\lim_{\ep_0\rightarrow0+}\frac{1}{2\pi i}
\int_{\pa\Omega_{\ep_0}}e^{-i\al tc_{\ep}}\nu_+(F)(c_{\ep})\va_+(0,c_{\ep})dc_{\ep}.
\end{align*}
We have for the second term
\begin{align*}
&\frac{1}{2\pi i}\int_{\pa\Omega_{\ep_0}}e^{-i\al tc_{\ep}}
\nu_+(F)(c_{\ep})\va_+(0,c_{\ep})dc_{\ep}\\
&=
\frac{1}{2\pi i}\int_{\{|c_{\ep}|\leq \sqrt{2}\ep_0\}
\cap\pa\Omega_{\ep_0}}e^{-i\al tc_{\ep}}\nu_+(F)(c_{\ep})\va_+(0,c_{\ep})dc_{\ep}\\
&\quad+\frac{1}{2\pi i}\int_{\{|c_{\ep}|>\sqrt{2}\ep_0\}
\cap\pa\Omega_{\ep_0}}e^{-i\al tc_{\ep}}\nu_+(F)(c_{\ep})\va_+(0,c_{\ep})dc_{\ep}\\
&=I(c_{\ep_0})+J(c_{\ep_0}).
\end{align*}
By Proposition \ref{prop: mu nu lim}, for $c_{\ep}\in \pa\Omega_{\ep_0}$ and $|c_{\ep}|\leq \sqrt{2}\ep_0$, we have
\beno
|\nu_+(F)(c_{\ep})|\leq C\|F\|_{L^{\infty}}\big(1+\big|\ln|c_{\ep}|\big|\big)\leq C\|F\|_{L^{\infty}}\left(\big|\ln|\ep_0|\big|+1\right).
\eeno
Thus, we deduce that
$I(c_{\ep})\leq C\|F\|_{L^{\infty}}\ep_0\left(\big|\ln|\ep_0|\big|+1\right),$
and then $\lim\limits_{\ep_0\rightarrow0+}I(c_{\ep_0})=0.$

As for $J(c_{\ep_0})$, we have
\begin{align*}
J(c_{\ep_0})
&=-\frac{1}{2\pi i}\int_{\ep_0}^{M_+}e^{-i\al t(c+i\ep_0)}\nu_+(F)(c+i\ep_0)\va_+(0,c+i\ep_0)dc\\
&\quad-\frac{1}{2\pi i}\int_{m_-}^{\ep_0}e^{-i\al t(c+i\ep_0)}\nu_+(F)(c+i\ep_0)\va_+(0,c+i\ep_0)dc\\
&\quad+\frac{1}{2\pi i}\int_{m_-}^{\ep_0}e^{-i\al t(c-i\ep_0)}\nu_+(F)(c-i\ep_0)\va_+(0,c-i\ep_0)dc\\
&\quad+\frac{1}{2\pi i}\int_{\ep_0}^{M_+}e^{-i\al t(c-i\ep_0)}\nu_+(F)(c-i\ep_0)\va_+(0,c-i\ep_0)dc\\
&\quad+\frac{1}{2\pi i}\int_{\pa B^{l}_{\ep_0}}e^{-i\al tc_{\ep}}\nu_+(F)(c_{\ep})\va_+(0,c_{\ep})dc_{\ep}
+\frac{1}{2\pi i}\int_{\pa B^r_{\ep_0}}e^{-i\al tc_{\ep}}\nu_+(F)(c_{\ep})\va_+(0,c_{\ep})dc_{\ep},
\end{align*}
by Proposition \ref{prop: mu nu lim} and the Lebesgue's dominated convergence theorem, we have
\begin{align*}
\lim_{\ep_0\to 0+}J(c_{\ep_0})
&=\frac{1}{2\pi i}\int_{m_-}^{M_+}e^{-i\al tc}\big(\nu_+^-(F)(c)-\nu_+^+(F)(c)\big)\va_+(0,c)dc.
\end{align*}
By Remark \ref{rmk: munu lim up rem}, we have
$\big|\nu_+^{\pm}(F)(c)\big|\leq C\|F\|_{L^{\infty}}\big(\big|\ln|c|\big|+1\big)\in L_c^1$ and
\beno
\big(\nu_+^-(F)(c)-\nu_+^+(F)(c)\big)\va_+(0,c)\in L_c^1,
\eeno
and the Riemann-Lebesgue Lemma implies that
$
\lim\limits_{t\to +\infty}\lim\limits_{\ep_0\rightarrow0+}J(c_{\ep_0})\rightarrow0.
$
From which,  it implies that
\beno
\widehat{\psi}(t,\al,0)\rightarrow
\frac{u'(0)}{b'(0)}\widehat{\phi}_0(\al,0)
, \quad \textrm{as} \quad t\rightarrow+\infty.
\eeno

\no{\bf{Proof \ of \ 2.}}  For the case of $0<y\leq 1$,  for any $0<\ep\leq\ep_0$, let
$$
K(t,\al,y)=\lim_{\ep_0\to 0+}\frac{1}{2\pi i}\int_{\pa\Omega_{\ep_0}}e^{-i\al tc_{\ep}}\frac{u(y)-c_{\ep}}{c_{\ep}}\Theta_+(y,c_{\ep})dc_{\ep},
$$
and then
$
\widehat{\psi}(t,\al,y)
=\frac{u(y)}{b(y)}\widehat{\phi}_0(\al,0)\chi(y)+K(t,\al,y).
$

We divide $K(t,\al,y)$ into 6 parts and let
\begin{align*}
{\color{green}K_1(t,\al,y)}
&=\lim_{\ep\to0+}\frac{1}{2\pi i}
\int_{W_+(y)}^{M_+}
e^{-i\al t(c-i\ep)}\frac{u(y)-c+i\ep}{c-i\ep}\Theta_+^0(y,(c-i\ep))dc\\
&\quad+\lim_{\ep\to0+}\frac{1}{2\pi i}\int_{M_+}^{W_+(y)}
e^{-i\al t(c+i\ep)}\frac{u(y)-c-i{\ep}}{c+i{\ep}}\Theta_+^0(y,c+i{\ep})dc\\
&\quad+\lim_{\ep\to0+}\frac{1}{2\pi i}\int_{\pa B_{\ep}^r}
e^{-i\al tc_{\ep}}\frac{u(y)-c_{\ep}}{c_{\ep}}\Theta_+(y,c_{\ep})dc_{\ep},\\
{\color{blue}K_2(t,\al,y)}
&=\lim_{\ep\to0+}\frac{1}{2\pi i}
\int_{W_-(y)}^{m_-}
e^{-i\al t(c+i{\ep})}\frac{u(y)-c-i{\ep}}{c+i{\ep}}\Theta_+^0(y,c+i{\ep})dc\\
&\quad+\lim_{\ep\to0+}\frac{1}{2\pi i}\int_{m_-}^{ W_-(y)}
e^{-i\al t(c-i{\ep})}\frac{u(y)-c+i{\ep}}{c-i{\ep}}\Theta_+^0(y,c-i{\ep})dc\\
&\quad+\lim_{\ep\to0+}\frac{1}{2\pi i}\int_{\pa B_{\ep}^l}
e^{-i\al tc_{\ep}}\frac{u(y)-c_{\ep}}{c_{\ep}}\Theta_+(y,c_{\ep})dc_{\ep},
\end{align*}
and
\begin{align*}
&\lim_{\ep\to0+}\frac{1}{2\pi i}
\int_{W_-(\frac{y}{2})}^{W_+(\frac{y}{2})}e^{-i\al t(c-i{\ep})}
\frac{u(y)-c+i{\ep}}{c-i{\ep}}\Theta_+^1(y,c-i{\ep})dc\\
&\quad+\lim_{\ep\to0+}\frac{1}{2\pi i}\int_{W_+(\frac{y}{2})}^{ W_-(\frac{y}{2})}e^{-i\al t(c+i{\ep})}
\frac{u(y)-c-i{\ep}}{c+i{\ep}}\Theta_+^1(y,c+i{\ep})dc\\
&=\lim_{\ep\to0+}\frac{1}{2\pi i}
\int_{W_-(\frac{y}{2})}^{W_+(\frac{y}{2})}e^{-i\al t(c-i{\ep})}
\frac{u(y)-c+i{\ep}}{c-i{\ep}}\int_1^y\frac{\va_+(y,c-i{\ep})\int_{y_{c_+}}^{y'}(F\va_+)(z,c-i{\ep})dz}
{\mathcal{H}(y',c-i{\ep})\va_+(y',c-i{\ep})^2}dy'dc\\
&\quad+\lim_{\ep\to0+}\frac{1}{2\pi i}\int_{W_+(\frac{y}{2})}^{ W_-(\frac{y}{2})}e^{-i\al t(c+i{\ep})}
\frac{u(y)-c-i{\ep}}{c+i{\ep}}\int_1^y\frac{\va_+(y,c+i{\ep})\int_{y_{c_+}}^{y'}(F\va_+)(z,c+i{\ep})dz}
{\mathcal{H}(y',c+i{\ep})\va_+(y',c+i{\ep})^2}dy'dc\\
&\quad+\lim_{\ep\to0+}\frac{1}{2\pi i}
\int_{W_-(\frac{y}{2})}^{W_+(\frac{y}{2})}e^{-i\al t(c-i{\ep})}
\frac{u(y)-c+i{\ep}}{c-i{\ep}}\int_{1}^y\frac{\mu_+(F)(c-i{\ep})\va_+(y,c-i{\ep})}
{(\mathcal{H}\va_+^2)(y',c-i{\ep})}dy'dc\\
&\quad+\lim_{\ep\to0+}\frac{1}{2\pi i}\int_{W_+(\frac{y}{2})}^{ W_-(\frac{y}{2})}e^{-i\al t(c+i{\ep})}
\frac{u(y)-c-i{\ep}}{c+i{\ep}}\int_{1}^y\frac{\mu_+(F)(c+i{\ep})\va_+(y,c+i{\ep})}
{(\mathcal{H}\va_+^2)(y',c+i{\ep})}dy'dc\\
&\eqdef{\color{red}K_3(t,\al,y)+K_4(t,\al,y)},
\end{align*}
and let
\begin{align*}
K_5(t,\al,y)&=\lim_{\ep\to0+}\frac{1}{2\pi i}
\int_{W_+(\frac{y}{2})}^{W_+(y)}
e^{-i\al t(c-i{\ep})}\frac{u(y)-c+i{\ep}}{c-i{\ep}}\Theta_+^1(y,c-i{\ep})dc\\
&\quad+\lim_{\ep\to0+}\frac{1}{2\pi i}
\int_{W_+(y)}^{ W_+(\frac{y}{2})}
e^{-i\al t(c+i{\ep})}\frac{u(y)-c-i{\ep}}{c+i{\ep}}\Theta_+^1(y,c+i{\ep})dc\\
K_6(t,\al,y)&=\lim_{\ep\to0+}\frac{1}{2\pi i}
\int_{W_-(\frac{y}{2})}^{W_-(y)}
e^{-i\al t(c+i{\ep})}\frac{u(y)-c-i{\ep}}{c+i{\ep}}\Theta_+^1(y,c+i{\ep})dc\\
&\quad+\lim_{\ep\to0+}\frac{1}{2\pi i}
\int_{W_-(y)}^{ W_-(\frac{y}{2})}
e^{-i\al t(c-i{\ep})}\frac{u(y)-c+i{\ep}}{c-i{\ep}}\Theta_+^1(y,c-i{\ep})dc,
\end{align*}
so that $K(t,\al,y)=\sum\limits_{i=1}^{6}K_i(t,\al,y)$.
And for convenience, we give a picture to show that how we depart the contour domain:\\
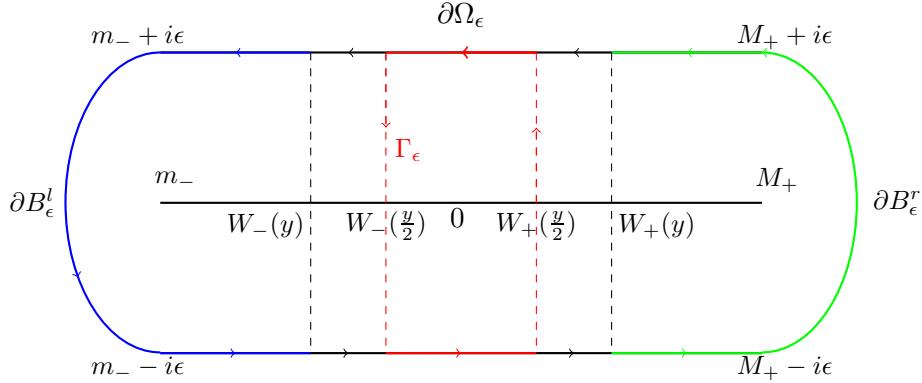
\begin{tikzpicture}[domain=0:4]
    \draw[thick] (-4,0) -- (4,0);
     \draw[thick] (-4,2) -- (4,2);
      \draw[thick] (-4,-2) -- (4,-2);
     \draw[blue,thick] (-4,2) -- (-2,2);
      \draw[blue,thick] (-4,-2) -- (-2,-2);
      \draw (-0.05,-0.2) node {$0$};
    \draw (-3.8,0.3) node {\small $m_-$};
    \draw (4.2,0.3) node {\small $M_+$};
 \draw[dashed, color=red] (-1,2) -- (-1,-2);
 \draw[thick,red] (-1,2) -- (1,2);
 \draw[->,dashed, color=red](-1,2) -- (-1,1);
  \draw[->,dashed,red] (-1,-2) -- (0,-2);
  \draw[red,thick] (-1,-2) -- (1,-2);
 \draw (-1,-0.3) node {\small $W_-(\frac{y}{2})$};
 \draw[dashed, color=red] (1,2) -- (1,-2);
  \draw[->,dashed, color=red] (1,0) -- (1,1);
  \draw[->,thick,red] (1,2) -- (0,2);
  \draw[->] (-1,2) -- (-1.5,2);
  \draw[->] (-2,-2) -- (-1.5,-2);
  \draw[->] (2,2) -- (1.5,2);
  \draw[->] (1,-2) -- (1.5,-2);
   \draw[->,blue] (-2,2) -- (-3,2);
  \draw[->,blue] (-4,-2) -- (-3,-2);
  \draw[->,green] (4,2) -- (3,2);
  \draw[->,green] (2,-2) -- (3,-2);
   \draw[thick,green] (2,-2) -- (4,-2);
  \draw[thick,green] (2,2) -- (4,2);
  \draw (1,-0.3) node {\small $W_+(\frac{y}{2})$};
 \draw[dashed] (-2,2) -- (-2,-2);
   \draw (-2.6,-0.3) node {\small $W_-(y)$};
 \draw[dashed] (2,2) -- (2,-2);
   \draw (2.6,-0.3) node {\small $W_+(y)$};
    \draw[green, thick]  (4, 2) arc [start angle = 90, end angle = -90,
    x radius = 12.6mm, y radius = 19.95mm];
    \draw[<-,green]  (4, 2) arc [start angle = 90, end angle = -90,
    x radius = 12.6mm, y radius = 19.95mm];
    \draw (5.8,0) node {\small $\pa B_{\epsilon}^r$};
    \draw[blue, thick]  (-4, 2) arc [start angle = 90, end angle = -90,
    x radius = -12.6mm, y radius = 19.95mm];
    \draw[->,blue]  (-4, 2) arc [start angle = 90, end angle = -30,
    x radius = -12.6mm, y radius = 19.95mm];
    \draw (-5.7,0) node {\small $\pa B_{\epsilon}^l$};
\draw (0,2.5) node {$\pa\Om_{\ep}$};
\draw (-0.7,0.7) node {\small \color{red}$\Gamma_{\ep}$};
\draw (-4.3,-2.2) node {\small $m_--i\epsilon$};
\draw (-4.3,2.2) node {\small$m_-+i\epsilon$};
\draw (4.3,-2.2) node {\small$M_+-i\epsilon$};
\draw (4.3,2.2) node {\small $M_++i\epsilon$};
\end{tikzpicture}

\no As for $K_1$, we have
\begin{align*}
&K_1(t,\al,y)\\
&=\lim_{\ep\to0+}\frac{1}{2\pi i}
\int_{W_+(y)}^{M_+}e^{-i\al t(c-i\ep)}\frac{u(y)-c+i\ep}{c-i\ep}
\Big\{\va_+(y,c-i\ep)\int_0^y\frac{\int_{y_{c_+}}^{y'}(F\va_+)(z,c-i\ep)dz}
{(\mathcal{H}\va_+^2)(y',c-i\ep)}dy'\\
&\quad +\mu_+(F)(c-i\ep)\va_+(y,c-i\ep)
\int_0^y\frac{1}{(\mathcal{H}\va_+^2)(y',c-i\ep)}dy'
+\nu_+(F)(c-i\ep)\va_+(y,c-i\ep)\Big\}dc\\
&\quad  -\lim_{\ep\to0+}\frac{1}{2\pi i}
\int_{W_+(y)}^{M_+}e^{-i\al t(c+i\ep)}
\frac{u(y)-c-i\ep}{c+i\ep}\Big\{\va_+(y,c+i\ep)
\int_0^y\frac{\int_{y_{c_+}}^{y'}(F\va_+)(z,c+i\ep)dz}
{(\mathcal{H}\va_+^2)(y',c+i\ep)}dy'\\
&\quad  +\mu_+(F)(c+i\ep)\va_+(y,c+i\ep)
\int_0^y\frac{1}{(\mathcal{H}\va_+^2)(y',c+i\ep)}dy'
+\nu_+(F)(c+i\ep)\va_+(y,c+i\ep)\Big\}dc\\
&\quad  -\lim_{\ep\to0+}\frac{\ep}{2\pi}
\int_{\frac{\pi}{2}}^{\frac{3\pi}{2}}e^{-i\al t\big(M_++\ep e^{i\theta}\big)}
\frac{u(y)-M_++\ep e^{i\theta}}{M_+-\ep e^{i\theta}}
\Theta_+\big(y,M_+-\ep e^{i\theta}\big) e^{i\theta}d\theta\nonumber\\
&=K_{11}(t,\al,y)+K_{12}(t,\al,y)+K_{13}(t,\al,y).
\end{align*}
Proposition \ref{prop: Br Bl Theta bdd} implies
$
K_{13}(t,\al,y)=0.
$\\
For $c\in \big[W_+(y),M_+\big]$ with $y\in (0,1]$ fixed, by Proposition \ref{prop: mu nu lim}, Proposition \ref{prop: F y bdd lim}, Remark \ref{rmk: D} and the Lebesgue's dominated convergence theorem, we obtain that
\begin{align*}
&K_1(t,\al,y)=K_{11}(t,\al,y)+K_{12}(t,\al,y)\\
&=\frac{1}{2\pi i}\int_{W_+(y)}^{M_+}e^{-i \al tc}\frac{u(y)-c}{c}
\Big\{\big(\mu_+^-(F)(c)-\mu_+^+(F)(c)\big)
\int_0^y\frac{\va_+(y,c)}{(\mathcal{H}\va_+^2)(y',c)}dy'\\
\nonumber&\quad +\big(\nu_+^-(F)(c)-\nu_+^+(F)(c)\big)\va_+(y,c)\Big\}dc\\
&= \frac{1}{\pi}\int_{W_+(y)}^{M_+}e^{-i \al tc}\frac{u(y)-c}{c}
\bigg(\frac{\mathcal{D}^{im}(c)\mathcal{U}^{re}_+(F)(c)-\mathcal{D}^{re}(c)
\mathcal{U}^{im}_+(F)(c)}
{\mathcal{D}^{re}(c)^2+\mathcal{D}^{im}(c)^2}
\int_0^y\frac{\va_+(y,c)}{(\mathcal{H}\va_+^2)(y',c)}dy'\\
&\quad
+\frac{\mathcal{D}^{im}(c)\mathcal{V}^{re}_+(F)(c)-\mathcal{D}^{re}(c)
\mathcal{V}^{im}_+(F)(c)}
{\mathcal{D}^{re}(c)^2+\mathcal{D}^{im}(c)^2}\va_+(y,c)\bigg)dc.
\end{align*}
Here for $c\in[W_+(y), M_+]$, by Remark \ref{rmk: D}, Lemma \ref{lem: limit up bdd}, Lemma \ref{lem: XYST bdd} and Proposition \ref{prop: F y bdd lim}, we have
\begin{align*}
&\Big|\frac{u(y)-c}{c}
\frac{\mathcal{D}^{im}(c)\mathcal{U}^{re}_+(F)(c)-\mathcal{D}^{re}(c)
\mathcal{U}^{im}_+(F)(c)}
{\mathcal{D}^{re}(c)^2+\mathcal{D}^{im}(c)^2}
\int_0^y\frac{\va_+(y,c)}{(\mathcal{H}\va_+^2)(y',c)}dy'\Big|\\
&\quad\leq C\|F\|_{L^{\infty}}\big(\big|\ln|y-y_{c_+}|\big|+1\big)\in L_c^1\big(W_+(y), M_+\big),\\
&\Big|\frac{u(y)-c}{c}\frac{\mathcal{D}^{im}(c)\mathcal{V}^{re}_+(F)(c)
-\mathcal{D}^{re}(c)\mathcal{V}^{im}_+(F)(c)}
{\mathcal{D}^{re}(c)^2+\mathcal{D}^{im}(c)^2}\va_+(y,c)\Big|\\
&\quad\leq C\|F\|_{L^{\infty}}\in L_c^1\big(W_+(y), M_+\big).
\end{align*}
And then the Riemann-Lebesgue lemma gives
\beno
\lim_{t\rightarrow+\infty}K_{11}(t,\al,y)+K_{12}(t,\al,y)=0.
\eeno
Thus we get
$\lim\limits_{t\to+\infty}K_1(t,\al,y)=0.$

By the same argument, we obtain
\begin{align*}
&K_2(t,\al,y)\\
&=\frac{1}{\pi}\int_{m_-}^{W_-(y)}e^{-i \al tc}
\frac{u(y)-c}{c}\Big\{\frac{\mathcal{D}^{im}(c)\mathcal{U}^{re}_+(F)(c)
-\mathcal{D}^{re}(c)\mathcal{U}^{im}_+(F)(c)}{\mathcal{D}^{re}(c)^2+\mathcal{D}^{im}(c)^2}
\int_0^y\frac{\va_+(y,c)}{(\mathcal{H}\va_+^2)(y',c)}dy'\\
&\quad
+\frac{\mathcal{D}^{im}(c)\mathcal{V}^{re}_+(F)(c)
-\mathcal{D}^{re}(c)\mathcal{V}^{im}_+(F)(c)}
{\mathcal{D}^{re}(c)^2+\mathcal{D}^{im}(c)^2}\va_+(y,c)\Big\}dc,
\end{align*}
and
\begin{align*}
&\Big|\frac{u(y)-c}{c}
\frac{\mathcal{D}^{im}(c)\mathcal{U}^{re}_+(F)(c)-\mathcal{D}^{re}(c)
\mathcal{U}^{im}_+(F)(c)}
{\mathcal{D}^{re}(c)^2+\mathcal{D}^{im}(c)^2}
\int_0^y\frac{\va_+(y,c)}{(\mathcal{H}\va_+^2)(y',c)}dy'\Big|\\
&\quad\leq C\|F\|_{L^{\infty}}\big(\big|\ln|y-y_{c_+}|\big|+1\big)\in L_c^1\big(m_-,W_-(y)\big),\\
&\Big|\frac{u(y)-c}{c}\frac{\mathcal{D}^{im}(c)\mathcal{V}^{re}_+(F)(c)
-\mathcal{D}^{re}(c)\mathcal{V}^{im}_+(F)(c)}
{\mathcal{D}^{re}(c)^2+\mathcal{D}^{im}(c)^2}\va_+(y,c)\Big|\\
&\quad\leq C\|F\|_{L^{\infty}}\in L_c^1\big(m_-,W_-(y)\big).
\end{align*}
Thus
$\lim\limits_{t\rightarrow+\infty}K_2(t,\al,y)=0.$

We rewrite $K_5$ as follows
\begin{align*}
&K_5(t,\al,y)\\
&=\lim_{\ep\rightarrow0^+}\frac{1}{2\pi i}
\int_{W_+(\frac{y}{2})}^{W_+(y)}e^{-i\al t(c-i\ep)}
\frac{u(y)-c+i\ep}{c-i\ep}
\Big\{\va_+(y,c-i\ep)\int_1^y\frac{\int_{y_{c_+}}^{y'}(F\va_+)(z,c-i\ep)dz}
{(\mathcal{H}\va_+^2)(y',c-i\ep)}dy'\\
&\quad +\mu_+(F)(c-i\ep)
\int_1^y\frac{\va_+(y,c-i\ep)}{(\mathcal{H}\va_+^2)(y',c-i\ep)}dy'\Big\}dc\\
&\quad -\lim_{\ep\rightarrow0^+}\frac{1}{2\pi i}
\int_{W_+(\frac{y}{2})}^{W_+(y)}e^{-i\al t(c+i\ep)}
\frac{u(y)-c-i\ep}{c+i\ep}
\Big\{\va_+(y,c+i\ep)\int_1^y\frac{\int_{y_{c_+}}^{y'}(F\va_+)(z,c+i\ep)dz}
{(\mathcal{H}\va_+^2)(y',c+i\ep)}dy'\\
&\quad +\mu_+(F)(c+i\ep)
\int_1^y\frac{\va_+(y,c+i\ep)}{(\mathcal{H}\va_+^2)(y',c+i\ep)}dy'\Big\}dc.
\end{align*}
For $c\in[W_+(\frac{y}{2}), W_+(y)]$, then $0<\frac{y}{2}\leq y_{c_+}\leq y\leq1$. Thus by Proposition \ref{prop: F y bdd lim}, Remark \ref{rmk: D} and the Lebesgue's dominated convergence theorem, we get
\begin{align*}
&K_5(t,\al,y)\\
&=\frac{1}{2\pi i}
\int_{W_+(\frac{y}{2})}^{W_+(y)}e^{-i\al tc}\frac{u(y)-c}{c}\big(\mu_+^-(F)(c)-\mu_+^+(F)(c)\big)
\int_1^y\frac{\va_+(y,c)}{(\mathcal{H}\va_+^2)(y',c)}dy'dc\\
&=\frac{1}{\pi}
\int_{W_+(\frac{y}{2})}^{W_+(y)}e^{-i\al tc}\frac{u(y)-c}{c}
\frac{\mathcal{D}^{im}(c)\mathcal{U}^{re}_+(F)(c)-\mathcal{D}^{re}(c)\mathcal{U}^{im}_+(F)(c)}
{\mathcal{D}^{re}(c)^2+\mathcal{D}^{im}(c)^2}
\int_1^y\frac{\va_+(y,c)}{(\mathcal{H}\va_+^2)(y',c)}dy'dc.
\end{align*}
And for $c\in[W_+(\frac{y}{2}), W_+(y)]$,  by Remark \ref{rmk: D} and Lemma \ref{lem: limit up bdd} and Lemma \ref{lem: XYST bdd}, we obtain
\begin{align*}
&\Big|\frac{u(y)-c}{c}
\frac{\mathcal{D}^{im}(c)\mathcal{U}^{re}_+(F)(c)
-\mathcal{D}^{re}(c)\mathcal{U}^{im}_+(F)(c)}
{\mathcal{D}^{re}(c)^2+\mathcal{D}^{im}(c)^2}
\int_1^y\frac{\va_+(y,c)}{(\mathcal{H}\va_+^2)(y',c)}dy'\Big|\\
&\leq C(y)\|F\|_{L^{\infty}}\big(\big|\ln|y-y_{c_+}|\big|+1\big)\in L_c^1\big(W_+(y/2), W_+(y)\big),
\end{align*}
and then the Riemann-Lebesgue lemma implies $\lim\limits_{t\rightarrow+\infty}K_5(t,\al,y)=0.$\\
By the same argument, we can also get
\begin{align*}
&K_6(t,\al,y)\\
&=\frac{1}{\pi}
\int^{W_-(\frac{y}{2})}_{W_-(y)}e^{-i\al tc}\frac{u(y)-c}{c}
\frac{\mathcal{D}^{im}(c)\mathcal{U}^{re}_+(F)(c)-\mathcal{D}^{re}(c)\mathcal{U}^{im}_+(F)(c)}
{\mathcal{D}^{re}(c)^2+\mathcal{D}^{im}(c)^2}
\int_1^y\frac{\va_+(y,c)}{(\mathcal{H}\va_+^2)(y',c)}dy'dc,
\end{align*}
and $\lim\limits_{t\rightarrow+\infty}K_6(t,\al,y)=0.$

In the following, we mainly calculate the term $K_3(t,\al,y)$ and $K_4(t,\al,y)$.
Recall that
\begin{align*}
K_4(t,\al,y)&=\lim_{\ep\to0+}\frac{1}{2\pi i}
\int_{W_-(\frac{y}{2})}^{W_+(\frac{y}{2})}e^{-i\al t(c-i{\ep})}
\frac{u(y)-c+i{\ep}}{c-i{\ep}}\int_{1}^y\frac{\mu_+(F)(c-i{\ep})\va_+(y,c-i{\ep})}
{(\mathcal{H}\va_+^2)(y',c-i{\ep})}dy'dc\\
&\quad+\lim_{\ep\to0+}\frac{1}{2\pi i}\int_{W_+(\frac{y}{2})}^{ W_-(\frac{y}{2})}e^{-i\al t(c+i{\ep})}
\frac{u(y)-c-i{\ep}}{c+i{\ep}}\int_{1}^y\frac{\mu_+(F)(c+i{\ep})\va_+(y,c+i{\ep})}
{(\mathcal{H}\va_+^2)(y',c+i{\ep})}dy'dc,
\end{align*}
By Proposition \ref{prop: mu nu lim} and by Proposition \ref{prop: F y bdd lim} and the Lebesgue's dominated convergence theorem, we obtain
\begin{align*}
K_4(t,\al,y)=\frac{1}{2\pi i}
\int_{W_-(\frac{y}{2})}^{W_+(\frac{y}{2})}e^{-i\al tc}
\frac{u(y)-c}{c}(\mu_+^-(F)(c)-\mu_+^+(F)(c))
\int_1^y\frac{\va_+(y,c)}
{(\mathcal{H}\va_+^2)(y',c)}dy'dc.
\end{align*}
And from Remark \ref{rmk: munu lim up rem} and Proposition \ref{prop: F y bdd lim}, we have for $c\in[W_-(\frac{y}{2}), W_+(\frac{y}{2})]$,
\begin{align*}
&\Big|\frac{u(y)-c}{c}\big(\mu_+^-(F)(c)-\mu_+^+(F)(c)\big)
\int_1^y\frac{\va_+(y,c)}
{(\mathcal{H}\va_+^2)(y',c)}dy'dc\Big|\\
&\ \leq \frac{C\|F\|_{L^{\infty}}\big(\big|\ln|y-y_{c_+}|\big|+1\big)}{|y|}\big(\big|\ln|c|\big|+1\big)
\in L_c^1\big(W_-(y/2), W_+(y/2)\big).
\end{align*}
And then the Riemann-Lebesgue lemma gives $
 \lim\limits_{t\rightarrow+\infty}K_4(t,\al,y)=0.$

Let
\beno
H(y,c_{\ep})=\int_1^y\frac{\va_+(y,c_{\ep})}
{(\mathcal{H}\va_+^2)(y',c_{\ep})}\int_{y_{c_+}}^{y'}(F\va_+)(z,c_{\ep})dz
dy',
\eeno
and $\Gamma_{\ep}$ be the boundary of $\big\{c:~W_-(\f{y}{2})\leq Re c\leq W_+(\f{y}{2}),\ |Im c|\leq \ep\big\}$. Then we have
\begin{align*}
&K_3(t,\al,y)\\
&=\lim_{\ep\to0+}\frac{1}{2\pi i}
\int_{W_-(\frac{y}{2})}^{W_+(\frac{y}{2})}e^{-i\al t(c-i{\ep})}
\frac{u(y)-c+i{\ep}}{c-i{\ep}}\big(H(y,c-i\ep)-H(y,0)\big)dc\\
&\quad+\lim_{\ep\to0+}\frac{1}{2\pi i}\int_{W_+(\frac{y}{2})}^{ W_-(\frac{y}{2})}e^{-i\al t(c+i{\ep})}
\frac{u(y)-c-i{\ep}}{c+i{\ep}}\big(H(y,c+i\ep)-H(y,0)\big)dc\\
&\quad+H(y,0)\lim_{\ep\to0+}\frac{1}{2\pi i}\int_{\Gamma_{\ep}}
e^{-i\al tc}\frac{u(y)-c}{c}dc\\
&\quad
-H(y,0)\lim_{\ep\to0+}\frac{1}{2\pi}
\int_{-\ep}^{\ep}e^{-i\al t\big(W_+(\frac{y}{2})+i\tau\big)}
\frac{u(y)-W_+(\frac{y}{2})-i\tau}{W_+(\frac{y}{2})+i\tau}d\tau\\
&\quad
-H(y,0)\lim_{\ep\rightarrow0^+}\frac{1}{2\pi }
\int_{\ep}^{-\ep}e^{-i\al t\big(W_-(\frac{y}{2})+i\tau\big)}
\frac{u(y)-W_-(\frac{y}{2})-i\tau}{W_-(\frac{y}{2})+i\tau}d\tau\\
&=K_{31}(t,\al,y)+K_{32}(t,\al,y)+K_{33}(t,\al,y)+K_{34}(t,\al,y)+K_{35}(t,\al,y),
\end{align*}
where we can easily get that
\beno
K_{33}(t,\al,y)=H(y,0)\lim_{\ep\rightarrow0+}\frac{1}{2\pi i}
\int_{\Gamma_{\ep}}e^{-i\al tc_{\ep}}\frac{u(y)-c_{\ep}}{c_{\ep}}dc_{\ep}=u(y)H(y,0),
\eeno
and $K_{34}(t,\al,y)=K_{35}(t,\al,y)=0$.

As for $K_{31}$ and $K_{32}$, we have for $c_{\ep}=c\pm i\ep$,
\begin{align*}
&H(y,c_{\ep})-H(y,0)\\
&=\big(\va_+(y,c_{\ep})-\va_+(y,0)\big)
\int_1^y\frac{1}{(\mathcal{H}\va_+^2)(y',c_{\ep})}
\int_{y_{c_+}}^{y'}(F\va_+)(z,c_{\ep})dzdy'\\
&\quad +\va_+(y,0)\int_1^y\Big(\frac{1}{(\mathcal{H}\va_+^2)(y',c_{\ep})}
-\frac{1}{(\mathcal{H}\va_+^2)(y',0)}\Big)
\int_{y_{c_+}}^{y'}(F\va_+)(z,c_{\ep})dzdy'\\
&\quad-\va_+(y,0)\int_1^y\frac{1}{(\mathcal{H}\va_+^2)(y',0)}
\int_0^{y_{c_+}}(F\va_+)(z,c_{\ep})dzdy'\\
&\quad +\va_+(y,0)\int_1^y\frac{1}{(\mathcal{H}\va_+^2)(y',0)}
\int_0^{y'}\big[(F\va_+)(z,c_{\ep})-(F\va_+)(z,0)\big]dzdy'.
\end{align*}
Then by Proposition \ref{prop: sol. hom. [0,1]} and Proposition \ref{prop: F y bdd lim}, we have
\beno
|H(y,c_{\ep})-H(y,0)|\leq C(y)|c_{\ep}|,
\eeno
and $\lim\limits_{\ep\to 0+}(H(y,c\pm i{\ep})-H(y,0))=H(y,c)-H(y,0)$
and then the Lebesgue's dominated convergence theorem gives
\beno
K_{31}(t,\al,y)=\frac{1}{2\pi i}
\int_{W_-(\frac{y}{2})}^{W_+(\frac{y}{2})}e^{-i\al tc}
\frac{u(y)-c}{c}\big(H(y,c)-H(y,0)\big)dc=-K_{32}(t,\al,y).
\eeno
Thus we obtain $K_3(t,\al,y)=u(y)H(y,0)$ with
\beno
H(y,0)=\va_+(y,0)\int_1^y\frac{\int_0^{y'}F(z,0)\va_+(z,0)dz}
{\big(u(y')^2-b(y')^2\big)\va_+(y',0)^2}dy'.
\eeno

Therefore we get for $y\in(0,1]$,
\begin{align*}
\widehat{\psi}(t,\al,y)
&=\frac{u(y)}{b(y)}\widehat{\phi}_0(\al,0)\chi(y)+K(t,\al,y)\\
&=\frac{u(y)}{b(y)}\widehat{\phi}_0(\al,0)\chi(y)+u(y)H(y,0)+R_1^+(t,\al,y)+R_2^+(t,\al,y)+R_3^+(t,\al,y),
\end{align*}
where
\begin{align*}
R_{1}^+(t,\al,y)&=\frac{1}{\pi}\int_{W_+(y)}^{M_+}e^{-i \al tc}\frac{u(y)-c}{c}
\bigg(\frac{\mathcal{D}^{im}(c)\mathcal{U}^{re}_+(F)(c)-\mathcal{D}^{re}(c)
\mathcal{U}^{im}_+(F)(c)}
{\mathcal{D}^{re}(c)^2+\mathcal{D}^{im}(c)^2}
\int_0^y\frac{\va_+(y,c)}{(\mathcal{H}\va_+^2)(y',c)}dy'\\
&\quad
+\frac{\mathcal{D}^{im}(c)\mathcal{V}^{re}_+(F)(c)-\mathcal{D}^{re}(c)
\mathcal{V}^{im}_+(F)(c)}
{\mathcal{D}^{re}(c)^2+\mathcal{D}^{im}(c)^2}\va_+(y,c)\bigg)dc,\\
R_2^+(t,\al,y)
&=\frac{1}{\pi}\int_{m_-}^{W_-(y)}e^{-i \al tc}
\frac{u(y)-c}{c}\bigg(\frac{\mathcal{D}^{im}(c)\mathcal{U}^{re}_+(F)(c)
-\mathcal{D}^{re}(c)\mathcal{U}^{im}_+(F)(c)}{\mathcal{D}^{re}(c)^2+\mathcal{D}^{im}(c)^2}
\int_0^y\frac{\va_+(y,c)}{(\mathcal{H}\va_+^2)(y',c)}dy'\\
&\quad
+\frac{\mathcal{D}^{im}(c)\mathcal{V}^{re}_+(F)(c)
-\mathcal{D}^{re}(c)\mathcal{V}^{im}_+(F)(c)}
{\mathcal{D}^{re}(c)^2+\mathcal{D}^{im}(c)^2}\va_+(y,c)\bigg)dc,\\
R_3^+(t,\al,y)
&=\frac{1}{\pi}
\int_{W_-(y)}^{W_+(y)}e^{-i\al tc}\frac{u(y)-c}{c}
\frac{\mathcal{D}^{im}(c)\mathcal{U}^{re}_+(F)(c)-\mathcal{D}^{re}(c)\mathcal{U}^{im}_+(F)(c)}
{\mathcal{D}^{re}(c)^2+\mathcal{D}^{im}(c)^2}
\int_1^y\frac{\va_+(y,c)}{(\mathcal{H}\va_+^2)(y',c)}dy'dc,
\end{align*}
and $R_i^+(t,\al,y)\to 0$ as $t\to +\infty$ for $i=1,2,3,$ and $y>0$.

Similarly,  we also get that for $0<y\leq 1$,
\begin{align*}
\widehat{\phi}(t,\al,y)
&=\widehat{\phi}_0(\al,0)\chi(y)+b(y)H(y,0)+R_4^+(t,\al,y)+R_5^+(t,\al,y)+R_6^+(t,\al,y),
\end{align*}
where
\begin{align*}
R_{4}^+(t,\al,y)&=\frac{1}{\pi}\int_{W_+(y)}^{M_+}e^{-i \al tc}\frac{b(y)}{c}
\bigg(\frac{\mathcal{D}^{im}(c)\mathcal{U}^{re}_+(F)(c)-\mathcal{D}^{re}(c)
\mathcal{U}^{im}_+(F)(c)}
{\mathcal{D}^{re}(c)^2+\mathcal{D}^{im}(c)^2}
\int_0^y\frac{\va_+(y,c)}{(\mathcal{H}\va_+^2)(y',c)}dy'\\
&\quad
+\frac{\mathcal{D}^{im}(c)\mathcal{V}^{re}_+(F)(c)-\mathcal{D}^{re}(c)
\mathcal{V}^{im}_+(F)(c)}
{\mathcal{D}^{re}(c)^2+\mathcal{D}^{im}(c)^2}\va_+(y,c)\bigg)dc,\\
R_5^+(t,\al,y)
&=\frac{1}{\pi}\int_{m_-}^{W_-(y)}e^{-i \al tc}
\frac{b(y)}{c}\bigg(\frac{\mathcal{D}^{im}(c)\mathcal{U}^{re}_+(F)(c)
-\mathcal{D}^{re}(c)\mathcal{U}^{im}_+(F)(c)}{\mathcal{D}^{re}(c)^2+\mathcal{D}^{im}(c)^2}
\int_0^y\frac{\va_+(y,c)}{(\mathcal{H}\va_+^2)(y',c)}dy'\\
&\quad
+\frac{\mathcal{D}^{im}(c)\mathcal{V}^{re}_+(F)(c)
-\mathcal{D}^{re}(c)\mathcal{V}^{im}_+(F)(c)}
{\mathcal{D}^{re}(c)^2+\mathcal{D}^{im}(c)^2}\va_+(y,c)\bigg)dc,\\
R_6^+(t,\al,y)
&=\frac{1}{\pi}
\int_{W_-(y)}^{W_+(y)}e^{-i\al tc}\frac{b(y)}{c}
\frac{\mathcal{D}^{im}(c)\mathcal{U}^{re}_+(F)(c)-\mathcal{D}^{re}(c)\mathcal{U}^{im}_+(F)(c)}
{\mathcal{D}^{re}(c)^2+\mathcal{D}^{im}(c)^2}
\int_1^y\frac{\va_+(y,c)}{(\mathcal{H}\va_+^2)(y',c)}dy'dc,
\end{align*}
and $R_i^+(t,\al,y)\to 0$ as $t\to +\infty$ for $i=4,5,6,$ and $y>0$.\\

\no{\bf{Proof \ of \ 3.}} For $y\in [-1,0)$, we can get the conclusion by the same method and we obtain that
\begin{align*}
\widehat{\psi}(t,\al,y)
&=\frac{u(y)}{b(y)}\widehat{\phi}_0(\al,0)\chi(y)+u(y)\widetilde{H}(y,0)-R_1^-(t,\al,y)-R_2^-(t,\al,y)-R_3^-(t,\al,y),
\end{align*}
and
\begin{align*}
\widehat{\phi}(t,\al,y)
&=\widehat{\phi}_0(\al,0)\chi(y)+b(y)\widetilde{H}(y,0)-R_4^-(t,\al,y)-R_5^-(t,\al,y)-R_6^-(t,\al,y),
\end{align*}
where
\beno
\widetilde{H}(y,0)=\va_-(y,0)\int_{-1}^y\frac{1}{\big(u(y')^2-b(y')^2\big)\va_-(y',0)^2}
\int_0^{y'}F(z,0)\va_-(z,0)dzdy'.
\eeno
and
\begin{align*}
R_{1}^-(t,\al,y)&=\frac{1}{\pi}\int_{W_+(y)}^{m_-}e^{-i \al tc}\frac{u(y)-c}{c}
\bigg(\frac{\mathcal{D}^{im}(c)\mathcal{U}^{re}_-(F)(c)-\mathcal{D}^{re}(c)
\mathcal{U}^{im}_-(F)(c)}
{\mathcal{D}^{re}(c)^2+\mathcal{D}^{im}(c)^2}
\int_0^y\frac{\va_-(y,c)}{(\mathcal{H}\va_-^2)(y',c)}dy'\\
&\quad
+\frac{\mathcal{D}^{im}(c)\mathcal{V}^{re}_-(F)(c)-\mathcal{D}^{re}(c)
\mathcal{V}^{im}_-(F)(c)}
{\mathcal{D}^{re}(c)^2+\mathcal{D}^{im}(c)^2}\va_-(y,c)\bigg)dc,\\
R_2^-(t,\al,y)
&=\frac{1}{\pi}\int_{M_+}^{W_-(y)}e^{-i \al tc}
\frac{u(y)-c}{c}\bigg(\frac{\mathcal{D}^{im}(c)\mathcal{U}^{re}_-(F)(c)
-\mathcal{D}^{re}(c)\mathcal{U}^{im}_-(F)(c)}{\mathcal{D}^{re}(c)^2+\mathcal{D}^{im}(c)^2}
\int_0^y\frac{\va_-(y,c)}{(\mathcal{H}\va_-^2)(y',c)}dy'\\
&\quad
+\frac{\mathcal{D}^{im}(c)\mathcal{V}^{re}_-(F)(c)
-\mathcal{D}^{re}(c)\mathcal{V}^{im}_-(F)(c)}
{\mathcal{D}^{re}(c)^2+\mathcal{D}^{im}(c)^2}\va_-(y,c)\bigg)dc,\\
R_3^-(t,\al,y)
&=\frac{1}{\pi}
\int_{W_-(y)}^{W_+(y)}e^{-i\al tc}\frac{u(y)-c}{c}
\frac{\mathcal{D}^{im}(c)\mathcal{U}^{re}_-(F)(c)-\mathcal{D}^{re}(c)\mathcal{U}^{im}_-(F)(c)}
{\mathcal{D}^{re}(c)^2+\mathcal{D}^{im}(c)^2}
\int_{-1}^y\frac{\va_-(y,c)}{(\mathcal{H}\va_-^2)(y',c)}dy'dc,\\
R_{4}^-(t,\al,y)&=\frac{1}{\pi}\int_{W_+(y)}^{m_-}e^{-i \al tc}\frac{b(y)}{c}
\bigg(\frac{\mathcal{D}^{im}(c)\mathcal{U}^{re}_-(F)(c)-\mathcal{D}^{re}(c)
\mathcal{U}^{im}_-(F)(c)}
{\mathcal{D}^{re}(c)^2+\mathcal{D}^{im}(c)^2}
\int_0^y\frac{\va_-(y,c)}{(\mathcal{H}\va_-^2)(y',c)}dy'\\
&\quad
+\frac{\mathcal{D}^{im}(c)\mathcal{V}^{re}_-(F)(c)-\mathcal{D}^{re}(c)
\mathcal{V}^{im}_-(F)(c)}
{\mathcal{D}^{re}(c)^2+\mathcal{D}^{im}(c)^2}\va_-(y,c)\bigg)dc,\\
R_5^-(t,\al,y)
&=\frac{1}{\pi}\int_{M_+}^{W_-(y)}e^{-i \al tc}
\frac{b(y)}{c}\bigg(\frac{\mathcal{D}^{im}(c)\mathcal{U}^{re}_-(F)(c)
-\mathcal{D}^{re}(c)\mathcal{U}^{im}_-(F)(c)}{\mathcal{D}^{re}(c)^2+\mathcal{D}^{im}(c)^2}
\int_0^y\frac{\va_-(y,c)}{(\mathcal{H}\va_-^2)(y',c)}dy'\\
&\quad
+\frac{\mathcal{D}^{im}(c)\mathcal{V}^{re}_-(F)(c)
-\mathcal{D}^{re}(c)\mathcal{V}^{im}_-(F)(c)}
{\mathcal{D}^{re}(c)^2+\mathcal{D}^{im}(c)^2}\va_-(y,c)\bigg)dc,\\
R_6^-(t,\al,y)
&=\frac{1}{\pi}
\int_{W_-(y)}^{W_+(y)}e^{-i\al tc}\frac{b(y)}{c}
\frac{\mathcal{D}^{im}(c)\mathcal{U}^{re}_-(F)(c)-\mathcal{D}^{re}(c)\mathcal{U}^{im}_-(F)(c)}
{\mathcal{D}^{re}(c)^2+\mathcal{D}^{im}(c)^2}
\int_{-1}^y\frac{\va_-(y,c)}{(\mathcal{H}\va_-^2)(y',c)}dy'dc,
\end{align*}
and $R_i^-(t,\al,y)\to 0$ as $t\to +\infty$ for $i=1,2,...,6,$ and $y<0$.

At last, we give the calculation of the term $H(y,0)$.

Firstly, due to the fact that
$$
F(y,0)=-\widehat{\phi}_0(\al,0)
\left(\pa_y\Big[\big(u(y)^2-b(y)^2\big)\pa_y\Big(\frac{\chi(y)}{b(y)}\Big)\Big]-\al^2 \big(u(y)^2-b(y)^2\big)\frac{\chi(y)}{b(y)}\right),
$$
we have by using  the integration by parts,
\begin{align*}
&\int_0^{y'}(F\va_+)(z,0)dz\\
&=\widehat{\phi}_0(\al,0)\Big\{\int_0^{y'} \big(u(z)^2-b(z)^2\big)\pa_z\Big(\frac{\chi(z)}{b(z)}\Big)\pa_z\va_+(z,0)dz
- \big(u(z)^2-b(z)^2\big)\pa_z\Big(\frac{\chi(z)}{b(z)}\Big)\va_+(z,0)\Big|_0^{y'}\\
&\ \ +\int_0^{y'}\al^2 \big(u(z)^2-b(z)^2\big)\frac{\chi(z)}{b(z)}\va_+(z,0)dz\Big\}\\
&=-\widehat{\phi}_0(\al,0)\Big\{\int_0^{y'} \pa_z\Big(\big(u(z)^2-b(z)^2\big)\pa_z\va_+(z,0)\Big)\frac{\chi(z)}{b(z)}dz
-\big(u(z)^2-b(z)^2\big)\pa_z\va_+(z,0)\frac{\chi(z)}{b(z)}\Big|_0^{y'}\\
&\ \
+\big(u(z)^2-b(z)^2\big)\pa_z\Big(\frac{\chi(z)}{b(z)}\Big)\va_+(z,0)\Big|_0^{y'}
-\int_0^{y'}\al^2 \big(u(z)^2-b(z)^2\big)\frac{\chi(z)}{b(z)}\va_+(z,0)dz\Big\}\\
&=\widehat{\phi}_0(\al,0)\Big\{\big(u(z)^2-b(z)^2\big)
\pa_z\va_+(z,0)\frac{\chi(z)}{b(z)}\Big|_0^{y'}
- \big(u(z)^2-b(z)^2\big)\pa_z\Big(\frac{\chi(z)}{b(z)}\Big)\va_+(z,0)\Big|_0^{y'}\Big\}\\
&=-\widehat{\phi}_0(\al,0)b'(0)\frac{u'(0)^2-b'(0)^2}{b'(0)^2}-\widehat{\phi}_0(\al,0)(u(y')^2-b(y')^2)\pa_{y'}\big(\f{\chi}{b}\big)(y')\va_+(y',0)\\
&\quad+\widehat{\phi}_0(\al,0)(u(y')^2-b(y')^2)\f{\chi(y')}{b(y')}\pa_{y'}\va_+(y',0).
\end{align*}
Thus we obtain for $y>0$,
\beno
H(y,0)=-\widehat{\phi}_0(\al,0)\frac{u'(0)^2-b'(0)^2}{b'(0)}
\va_+(y,0)\int_1^y\frac{1}{\big(u(y')^2-b(y')^2\big)\va_+(y',0)^2}dy'-\f{\chi(y)\widehat{\phi}_0(\al,0)}{b(y)},
\eeno
and then for $0<y\leq 1$, we get
\begin{align*}
&\f{u(y)}{b(y)}\widehat{\phi}_0(\al,0)\chi(y)+u(y)H(y,0)=
-\frac{u'(0)^2-b'(0)^2}{b'(0)}\widehat{\phi}_0(\al,0)
\int_1^y\frac{u(y)\va_+(y,0)}{\big(u(y')^2-b(y')^2\big)\va_+(y',0)^2}dy',\\
&\widehat{\phi}_0(\al,0)\chi(y)+b(y)H(y,0)=
-\frac{u'(0)^2-b'(0)^2}{b'(0)}\widehat{\phi}_0(\al,0)
\int_1^y\frac{b(y)\va_+(y,0)}{\big(u(y')^2-b(y')^2\big)\va_+(y',0)^2}dy'.
\end{align*}

Similar as the case of $y\in(0,1]$, we have
\begin{align*}
&\f{u(y)}{b(y)}\widehat{\phi}_0(\al,0)\chi(y)+u(y)\widetilde{H}(y,0)=
-\frac{u'(0)^2-b'(0)^2}{b'(0)}\widehat{\phi}_0(\al,0)
\int_{-1}^y\frac{u(y)\va_+(y,0)}{\big(u(y')^2-b(y')^2\big)\va_-(y',0)^2}dy',\\
&\widehat{\phi}_0(\al,0)\chi(y)+b(y)\widetilde{H}(y,0)=
-\frac{u'(0)^2-b'(0)^2}{b'(0)}\widehat{\phi}_0(\al,0)
\int_{-1}^y\frac{b(y)\va_+(y,0)}{\big(u(y')^2-b(y')^2\big)\va_-(y',0)^2}dy'.
\end{align*}
Thus we complete the proof of Theorem \ref{main thm}.
\end{proof}

\section*{Acknowledgement}
C. Zhai is partially supported by China Postdoctoral Science Foundation under Grant 2018M630014. Z. Zhang is partially supported by NSF of China under Grant 11425103.

\end{CJK*}
\end{document}